\newif\ifsmfart
\IfFileExists{aclart.cls}
   {\documentclass[12pt,utf8,english,nothms]{aclart}
    \IfFileExists{smfenum.sty}{\usepackage{smfenum}}{}
    \usepackage{bull}
    \smfarttrue}
   {\message{^^J*** It would be better to typeset this file with smfart.cls ***^^J^^J}
    \documentclass[12pt,a4paper]{amsart}}

\usepackage[utf8]{inputenc}
\DeclareUnicodeCharacter{00A0}{}

\usepackage[T1]{fontenc}

\usepackage{lmodern}
\let\precneq\prec
\usepackage{amssymb,xspace}
\usepackage{hyperref}
\usepackage[matrix,arrow]{xy}

 \let\mathsec\mathsf
 \IfFileExists{mathrsfs.sty}
  {\usepackage{mathrsfs}\let\mathcal\mathscr}
  {\let\mathscr\mathcal}

\makeatletter \let\smf@boldmath\relax\makeatother

\advance\headheight 2pt
\calclayout
\allowdisplaybreaks[3]

\theoremstyle{plain}
\newtheorem{prop}[subsubsection]{Proposition}

\newtheorem{theo}[subsubsection]{Theorem}
\newtheorem{coro}[subsubsection]{Corollary}

\newtheorem{lemm}[subsubsection]{Lemma}

\newtheorem{defi}[subsubsection]{Definition}

\theoremstyle{definition}

\theoremstyle{remark}
\newtheorem{rema}[subsubsection]{Remark}
\newtheorem{rems}[subsubsection]{Remarks}

 \makeatletter

\def\subsection{\@startsection{subsection}{2}%
   \z@{.7\linespacing\@plus.3\linespacing}{0.3\linespacing}
   {\normalfont\bfseries\smf@boldmath}}

 \let\c@equation\c@subsubsection
 \let\cl@equation\cl@subsubsection
 \let\theequation\thesubsubsection

\def\l@table{\@tocline{0}{3pt plus2pt}{0pt}{}{\itshape}}

\makeatother

\def\Card{\mathop\#}
\def\Cl{\mathord{\mathscr C\!\ell}}
\def\Cl{\mathscr C}
\def\Clan{\mathscr C^{\text{\upshape an}}}
\def\Clanmax{\mathscr C^{\text{\upshape an,max}}}
\def\Aff{{\mathbf A}}
\def\AD{{\mathbb A}}

\def\C{{\mathbf C}}
\def\F{{\mathbf F}}

\def\gm{{\mathbf G}_m}
\def\N{{\mathbf N}}
\def\P{{\mathbf P}}
\def\Q{{\mathbf Q}}
\def\R{{\mathbf R}}

\def\Z{{\mathbf Z}}
\def\GL{{\operatorname{GL}}}
\def\SL{{\operatorname{SL}}}
\def\End{{\operatorname{End}}}
\def\Sym{{\operatorname{Sym}}}

\def\ord{\operatorname{ord}}

\def\Tube{{\mathsf T}}

        \def\EP{\operatorname{EP}}

\let\ra\rightarrow
\let\hra\hookrightarrow

\let\epsilon\varepsilon \let\eps\epsilon
\let\epsilon\varepsilon
\let\phi\varphi
\let\emptyset\varnothing
\let\leq\leqslant
\let\geq\geqslant

\def\an{\text{\upshape an}}
\def\id{\operatorname {id}}
\def\pl{\textsc {pl}\xspace}
\def\Lie{\operatorname{Lie}}

\def\abs#1{\left\lvert{#1}\right\rvert}
\def\norm#1{\left\|{#1}\right\|}

\def\DeclareMathOperator#1#2{\def #1{\operatorname{#2}}}
 
\DeclareMathOperator{\Re}{Re} 
\DeclareMathOperator{\Im}{Im} 
\DeclareMathOperator{\Pic}{Pic}
\DeclareMathOperator{\Gal}{Gal}

\DeclareMathOperator{\tr}{tr}
\DeclareMathOperator{\Spec}{Spec}
\DeclareMathOperator{\Proj}{Proj}
\DeclareMathOperator{\rang}{rank}
\DeclareMathOperator{\codim}{codim}
\DeclareMathOperator{\div}{div}
\DeclareMathOperator{\Val}{Val}
\DeclareMathOperator{\vol}{vol}
\DeclareMathOperator{\Hom}{Hom}
\DeclareMathOperator{\Aut}{Aut}

\DeclareMathOperator{\Res}{Res}
\DeclareMathOperator{\Val}{Val}

\DeclareMathOperator{\Ind}{Ind}
\DeclareMathOperator{\Fr}{Fr}

\def\loccit{\emph{loccit.}~\ignorespaces}
\def\cf{\emph{cf.}~\ignorespaces}
\def\theenumi{\alph{enumi}}\def\labelenumi{\theenumi)}

\title{Igusa integrals and volume asymptotics 
   in analytic and adelic geometry}

\author{Antoine Chambert-Loir}
\address{Universit\'e de Rennes~1, IRMAR--UMR 6625 du CNRS, Campus de Beaulieu, 35042 Rennes Cedex, France}
\address{Institut universitaire de France}
\email{antoine.chambert-loir@univ-rennes1.fr}

\author{Yuri Tschinkel}
\address{Courant Institute, NYU, 251 Mercer St.  New York, NY 10012, USA}
\email{tschinkel@cims.nyu.edu}

\begin{document}
\date{\today}
 
\begin{abstract}
We establish asymptotic formulas  for volumes of
height balls in analytic varieties over local fields
and in adelic points of algebraic varieties over number fields,
relating the Mellin transforms of height functions
to Igusa integrals and to global geometric  invariants
of the underlying variety.
In the adelic setting, this involves the construction
of general Tamagawa measures.
\end{abstract}
 
\ifsmfart 
 \begin{altabstract}
Nous établissons un développement asymptotique du volume
des boules de hauteur dans des variétés analytiques
sur des corps locaux et sur les points adéliques de variétés algébriques sur des corps de nombres. Pour cela, nous relions les transformées de Mellin
des fonctions hauteur à des intégrales de type Igusa et à des
invariants géométriques globaux de la variété sous-jacente. 
Dans le cas adélique, nous construisons des mesures de Tamagawa
dans un cadre général.
 \end{altabstract}
\fi


\maketitle

\tableofcontents
\clearpage
\section{Introduction}

The study of rational and integral points on algebraic varieties
defined over a number field often leads to  considerations
of volumes  of real, $p$-adic or adelic spaces.
A typical problem in arithmetic geometry is to
establish asymptotic expansions, when $B\ra\infty$, for the number 
$N_f(B)$ of solutions in rational integers smaller than~$B$
of a polynomial equation $f(\mathbf x)=0$. 

When applicable,
the circle method gives an answer in terms of a ``singular integral''
and a ``singular series'', which  itself can be viewed as a product
of $p$-adic densities.
The size condition  is only reflected in a parameter in
the singular integral, whose asymptotic expansion 
therefore governs that of $N_f(B)$.

More generally, one considers systems of polynomial equations,
\emph{i.e.}, algebraic varieties over a number field
or schemes of finite type over rings of integers,
together with embedding into a projective or affine space.
Such an embedding induces a \emph{height function} (see, \emph{e.g.},
\cite{lang1983,serre1997,hindry-silverman2000})
such that there are only finitely many solutions of bounded height,
in a fixed number field, resp. ring of integers.
A natural generalization of the problem above is to understand
the asymptotic behavior of the number of such solutions, as well
as  their distribution in the ambient space
for the local or adelic topologies, when the bound grows to infinity.

Apart from applications of the circle method,
many other instances of this problem have been successfully
investigated in recent years, in particular, in the context
of linear algebraic groups and their homogeneous spaces.
For such varieties, techniques from ergodic theory and harmonic
analysis are very effective;
for integral points, see
\cite{duke-r-s1993}, \cite{eskin-mcmullen1993},
\cite{eskin-m-s1996}, \cite{borovoi-r1995},
\cite{maucourant2007}, \cite{gorodnik-o-s2006};
for rational points, see
\cite{batyrev-m90}, \cite{franke-m-t89},
\cite{peyre95}, \cite{batyrev-t96},
\cite{chambert-loir-t2002}.

In most cases, the proof  subdivides into two parts: 
\begin{enumerate}
\item comparison of the point counting with a volume asymptotic;
\item explicit computation of this volume asymptotic.
\end{enumerate}
In this paper, we develop a general geometric framework
for the second part, \emph{i.e.}, for the understanding
of densities and volumes occurring in the counting problems above.
We now explain the main results.

\subsection{Tamagawa measures for algebraic varieties}

Let $F$ be a number field. Let $\Val(F)$ be the set
of equivalence classes of absolute values of~$F$.
For $v\in\Val(F)$, we write $v\mid p$ if $v$ defines the $p$-adic
topology on~$\Q$, and $v\mid \infty$ if it is archimedean.
For $v\in\Val(F)$,  let $F_v$ be the corresponding completion of~$F$
and, if $v$ is ultrametric, let $\mathfrak o_v$ be its ring of integers.
We identify~$v$ with the specific absolute value~$\abs\cdot_v$
on~$F_v$ defined by the formula $\mu(a\Omega)=\abs a_v\mu(\Omega)$,
where $\mu$ is any Haar measure on the additive group~$F_v$,
$a\in F_v$ and $\Omega$ is a measurable subset of~$F_v$
of finite measure.

Let $X$ be a smooth projective algebraic variety over~$F$.
Fix an adelic metric on its canonical line bundle~$K_X$
(see Section~\ref{subsubsec.adelic-metrics}).
For any~$v\in\Val (F)$, the set $X(F_v)$ carries an analytic topology
and  the chosen $v$-adic metric on~$K_X$
induces a Radon measure~$\tau_{X,v}$ on~$X(F_v)$
(see Section~\ref{subsubsec.volume-forms}).

Let $\mathscr X$ be a projective flat model of~$X$ over the ring
of integers of~$F$. Then, for almost all finite places~$v$,
the measure~$\tau_{X,v}$ coincides with the measure 
on $\mathscr X(\mathfrak o_v)=X(F_v)$ defined by Weil in~\cite{weil82}.

Let $D$ be an effective divisor on~$X$ which, geometrically,
has strict normal crossings and set $U=X\setminus D$.
Let $\mathsec f_D$ denote the  canonical section 
of the line bundle~$\mathscr O_X(D)$ (corresponding to the regular
function~$1$ when $\mathscr O_X(D)$ is viewed as a sub-sheaf
of the sheaf of meromorphic functions); by construction,
its divisor is~$D$.
Let us also fix an adelic metric on this line bundle.
We let $\mathscr D$ be the Zariski closure of~$D$ in the model~$\mathscr X$
and $\mathscr U=\mathscr X\setminus\mathscr D$ be its complement.

For any place~$v\in\Val(F)$, we define a measure
\[ \tau_{(X,D),v}=\frac1{\norm{\mathsec f_D}_v} \tau_{X,v} \]
on $U(F_v)$. Note that it is still a Radon measure; however $U(F_v)$
has infinite volume. If $U$ is an algebraic group,
this construction allows to recover the Haar measure of~$U(F_v)$
(see~\ref{subsubsec.haar-tamagawa}).

Let $\AD_F$ be the adele ring of~$F$, that is, the restricted
product of the fields~$F_v$ with respect to the subrings~$\mathfrak o_v$.

A nonzero Radon measure~$\tau$ on the adelic space~$U(\AD_F)$ 
induces measures~$\tau_v$ on any of the sets~$U(F_v)$,
which are well-defined up to a factor. Conversely, we can recover
the Radon measure~$\tau$ as the product of such measures~$\tau_v$
if the set of measures $(\tau_v)_{v\nmid\infty}$
satisfies the convergence condition: \emph{the infinite product 
$\prod_{v\nmid\infty} \tau_v(\mathscr U(\mathfrak o_v))$
is absolutely convergent.}

A family of convergence factors for $(\tau_v)$
is a family $(\lambda_v)_{v\nmid\infty}$
of positive real numbers such that the family of
measures $(\lambda_v\tau_v)$ satisfies the above convergence condition.

Our first result in this article is a definition
of a measure on~$U(\AD_F)$ via an appropriate choice
of convergence factors.

Let $\bar F$ be an algebraic closure of~$F$ and let $\Gamma=\Gal(\bar F/F)$
be the absolute Galois group.
Let $M$ be a free $\Z$-module of finite rank endowed with a continuous
action of~$\Gamma$; we let $\mathrm L(s,M)$ be the corresponding
Artin L-function, and, for all finite places $v\in\Val(F)$,
$\mathrm L_v(s,M)$ its local factor at $v$.
The function $s\mapsto \mathrm L(s,M)$ is holomorphic
for $\Re(s)>1$ and admits a meromorphic continuation to~$\C$;
let $\rho$ be its order at $s=1$ and define
\[ \mathrm L^*(1,M) = \lim_{s\ra 1} (s-1)^{-\rho} \mathrm L(s,M);\]
it is a positive real number.

\begin{theo}\label{theo.1.tamagawa}
Assume that $\mathrm H^1(X,\mathscr O_X)=\mathrm H^2(X,\mathscr O_X)=0$.
The abelian groups $M^0=\mathrm H^0(U_{\bar F},\gm)/\bar F^*$
and $M^1=\mathrm H^1(U_{\bar F},\gm)/\text{torsion}$ are free
$\Z$-modules of finite rank with a continuous action of~$\Gamma$.
Moreover, the family $(\lambda_v)$
given by
\[ \lambda_v= \mathrm L_v(1,M^0)/\mathrm L_v(1,M^1) \]
is a family of convergence factors.
\end{theo}

Assume that the hypotheses 
of the Theorem hold.
We then define a Radon measure on~$U(\AD_F)$ by the formula
\[ \tau_{(X,D)} = \frac{\mathrm L^*(1,M^1)}{\mathrm L^*(1,M^0)}
   \prod_{v\in\Val(F)} \left(\lambda_v \tau_{(X,D),v}\right), \]
where $\lambda_v$ is given by the Theorem if $v\nmid\infty$
and $\lambda_v=1$ else.
We call it the \emph{Tamagawa measure on~$U$, or, more precisely, 
on $U(\mathbb A_F)$.} 
This generalizes
the construction of a Tamagawa measure on an algebraic group,
where there is no~$M^1$ (see, \emph{e.g.}, \cite{weil82})
or on a projective variety,
where there is no~$M^0$,
in~\cite{peyre95}.

\subsection{Volume asymptotics in analytic geometry}

Let $L$ be an effective divisor in~$X$ whose support contains
the support of~$D$; let again $\mathsec f_L$ be the canonical
section of the line bundle $\mathscr O_X(L)$. Fix an
adelic metric on~$\mathscr O_X(L)$ and a place~$v\in \Val(F)$.
For any positive real number~$B$, the set of all $x\in U(F_v)$
such that $\norm{\mathsec f_L(x)}_v\geq 1/B$ is a compact
subset in~$U(F_v)$. It has thus finite volume $V(B)$
with respect to the measure~$\tau_{(X,D),v}$.

Let us decompose the divisor~$D_v=D_{F_v}$ as a sum
of irreducible divisors:
\[ D_{v}=\sum_{\alpha\in\mathscr A_v} d_{\alpha,v} D_{\alpha,v}. \]
For $\alpha\in\mathscr A_v$,
let $\lambda_{\alpha,v}$ be the multiplicity of~$D_{\alpha,v}$
in~$L_v$; there exists an effective divisor~$E_v$ on~$X_{F_v}$
such that
\[ L_v= E_v+\sum_{\alpha\in\mathscr A} \lambda_{\alpha,v} D_{\alpha,v} .\]
For any subset $A\subseteq\mathscr A_v$, we let $D_{A,v}$ be the
intersections of the $D_{\alpha,v}$, for $\alpha\in A$.

Let now $a_v(L,D)$ be the least rational number such that 
for any $\alpha\in\mathscr A_v$, with $D_{\alpha,v}(F_v)\neq\emptyset$,
one has
$a_v(L,D)\lambda_{\alpha,v}\geq d_{\alpha,v}-1$.
Let $\mathscr A_v(L,D)$ be the set of those $\alpha\in\mathscr A_v$
where equality holds and
$b_v(L,D)$ the maximal cardinality of subsets $A\in\mathscr A_v(L,D)$
such that $D_{A,v}(F_v)\neq\emptyset$.
To organize the combinatorial structure of these subsets,
we introduce variants of the simplicial complex considered, e.g., 
in~\cite{clemens1969} in the context of Hodge theory.


\begin{theo}\label{theo.2.volume-local}
Assume that $v$ is archimedean.

If $a_v(L,D)>0$, then $b_v(L,D)\geq 1$ and 
there exists a positive real number~$c$ such that
\[ V(B) \sim c B^{a_v(L,D)} (\log B)^{b_v(L,D)-1}. \]

If $a_v(L,D)=0$, then there exists a positive real number~$c$
such that
\[ V(B) \sim c  B^{a_v(L,D)} (\log B)^{b_v(L,D)}. \]
\end{theo}

With the notation above, we also give an explicit formula for the constant~$c$.
It involves integrals  over the sets $D_{A,v}(F_v)$
such that $\Card(A)=b_v(L,D)$, with respect to measures induced
from~$\tau_{(X,D),v}$ via the adjunction formula.

To prove this theorem, we introduce the Mellin transform
\[ Z(s) = \int_{U(F_v)} \norm{\mathsec f_L(x)}_v^s \,\mathrm d\tau_{(X,D),v}(x)\]
and establish its analytic properties. We regard $Z(s)$ 
as an integral over the compact analytic manifold~$X(F_v)$
of the function~$\norm{\mathsec f_L}_v^s$ with respect 
to a singular measure,  connecting the study of such zeta functions 
with the theory of Igusa local zeta functions, 
see~\cite{igusa1974a,igusa1974b}.
In particular, we show that $Z(s)$ is holomorphic
for $\Re(s)>a_v(L,D)$ and admits a meromorphic
continuation to some half-plane $\{\Re(s)>a_v(L,D)-\eps\}$,
with a pole at $s=a_v(L,D)$ of order~$b_v(L,D)$.
This part of the proof works over any local field.

If $v$ is archimedean (and $\eps>0$ is small enough), then
$Z(s)$ has no other pole in this half-plane.
Our volume estimate then follows from a standard Tauberian theorem.
When $v$ in non-archimedean, we can only deduce a weaker
estimate, \emph{i.e.}, upper and lower bounds of the stated
order of magnitude (Corollary~\ref{coro.asymptotic.volume-ultrametric}).

\subsection{Asymptotics of adelic volumes}

Assume that
$\mathrm H^1(X,\mathscr O_X)=\mathrm H^2(X,\mathscr O_X)=0$.
Theorem~\ref{theo.1.tamagawa} gives us Tamagawa measures~$\tau_{(X,D)}$
and~$\tau_X$ on the adelic spaces~$U(\AD_F)$ and $X(\AD_F)$ respectively.

Suppose furthermore that the supports of the divisors~$L$ and~$D$
are equal.
We then define a \emph{height function} $H_L$ on
the adelic space~$U(\AD_F)$ by the formula
\[ H_L ((x_v)_v)= \prod_{v\in\Val(F)} \norm{\mathsec f_L(x_v)}_v^{-1}. \]
This function $H_L\colon U(\AD_F)\ra\R_+$ is continuous and proper.
In particular, for any real number~$B$, the subset of~$U(\AD_F)$
defined by the inequality $H_L(\mathbf x)\leq B$ is compact,
hence has finite volume~$V(B)$  with respect to~$\tau_{(X,D)}$.
We are interested in the asymptotic behavior of $V(B)$
as $B\ra\infty$.

Let us decompose the divisors~$L$  and~$D$ as the sum
of their irreducible components (over~$F$).  
Since $L$ and~$D$ have the same support, one can write
\[ D = \sum_{\alpha\in\mathscr A} d_\alpha D_\alpha, \qquad
   L = \sum_{\alpha\in\mathscr A} \lambda_\alpha D_\alpha \]
for some positive integers~$d_\alpha$ and~$\lambda_\alpha$.
Let $a(L,D)$ be the least positive rational number such that
the~$\Q$-divisor $E=a(L,D)L-D$ is effective;
in other words,
\[ a(L,D) = \max_{\alpha\in\mathscr A} {d_\alpha}/\lambda_\alpha. \]
Let moreover~$b(L,D)$ be the number of $\alpha\in\mathscr A$ for
which equality is achieved. 

To the $\Q$-divisor~$E$, we can also attach a height function
on~$X(\AD_F)$ given by
\[ H_E(\mathbf x) = \prod_{v\in \Val(F)} \norm{\mathsec f_L(x_v)}_v^{-a(L,D)} \norm{\mathsec f_D(x,v)}_v \]
if $x_v\in U(F_v)$ for all~$v$, and by $H_E(\mathbf x)=+\infty$
else.
The product can diverge to~$+\infty$ 
but $H_E$ has a positive lower bound, reflecting the
effectivity of~$E$. In fact,
the function $H_E^{-1}$ is continuous on~$X(\AD_F)$.

\begin{theo}\label{theo.3.volume-global}
When $B\ra\infty$, one has the following asymptotic expansion
\[ V(B)  \sim \frac1{a(L,D) (b(L,D)-1)!} B^{a(L,D)} (\log B)^{b(L,D)-1}
   \int_{X(\AD_F)} H_E(\mathbf x)^{-1} \,\mathrm d\tau_X(\mathbf x). \]
\end{theo}

As a particular case, let us take $L=D$. We see that $a(L,D)=1$,
$b(L,D)$ is the number of irreducible components of~$D$
and the integral in the Theorem is equal
to the Tamagawa volume $\tau_X(X(\AD_F))$ defined by Peyre.

\medskip

In both local and adelic situations, our techniques 
are valid for any metrization of the underlying line bundles.
As was explained by Peyre in~\cite{peyre95} in the context
of rational points, 
this implies \emph{equidistribution theorems},
see Corollary~\ref{coro.equidistribution.local} in
the local case, and Theorem~\ref{theo.volume.global}
in the adelic case.

\subsection*{Roadmap of the paper}

Section~\ref{sec.general} 
is concerned with heights and measures on adelic spaces.
We first recall notation and definitions
for adeles, adelic metrics and measures on analytic manifolds. 
In Subsection~\ref{subsec.heights.adelic},
we then define height functions on adelic spaces
and establish their basic properties.
The construction of global Tamagawa measures
is done in Subsection~\ref{subsec.global-tamagawa}.
We conclude this Section by a general equidistribution theorem.

Section~\ref{s.igusa} is devoted to the theory 
of geometric analogues of Igusa integrals,
both in the local and adelic settings.
These integrals define holomorphic functions in several variables
which admit meromorphic continuations. (In the adelic case,
these meromorphic continuations may have natural boundaries.) 
To describe their first poles we
introduce in Subsection~\ref{subsec.clemens} the geometric, algebraic
and analytic Clemens complexes which encode the 
incidence properties of divisors  involved in the definition
of our geometric Igusa integrals.
We then apply this theory in Subsections~\ref{subsec.geometric-volume} 
and~\ref{subsec.asymptotics.adelic},
where we establish 
Theorems~\ref{theo.2.volume-local} and~\ref{theo.3.volume-global}
about volume asymptotics.

In Section~\ref{sec.examples}, we make explicit the main results of our article
in the case of wonderful compactifications
of semi-simple groups. 
In particular, we explain how to recover
the volume estimates  established in~\cite{maucourant2007}
for Lie groups,
and in~\cite{gorodnik-m-o2009,shalika-tb-t2007} for adelic groups.

\subsection*{Acknowledgments}
During the preparation of this work, we have benefited from conversations
with Michel Brion and François Loeser. 
The first author would like to acknowledge the support of the Institut
universitaire de France, and to thank the organizers 
of a conference at the Maxwell Institute,
Edinburgh, 2008, for having given him the opportunity to present
the results of this paper.
The second author 
was partially supported by NSF grant DMS-0602333.

\clearpage
\section{Metrics, heights, and Tamagawa measures}
\label{sec.general}

\subsection{Metrics and measures on local fields}

\subsubsection{Haar measures and absolute values}\label{subsubsec.local-fields}
Let $F$ be a local field of characteristic zero, \emph{i.e.}, either $\R$,
$\C$, or a finite extension
of the field $\Q_p$ of $p$-adic numbers.
Fix a Haar measure $\mu$ on~$F$. Its ``modulus'' 
is an absolute value $\abs{\cdot}$ on~$F$, defined
by $\mu(a\Omega)=\abs a\mu(\Omega)$
for any $a\in F$ and any measurable subset $\Omega\subset F$.
For $F=\R$, this is the usual absolute value, for $F=\C$, it is its square.
For $F=\Q_p$, it is given by $\abs{p}=1/p$ and if $F'/F$ is a
finite extension, one has $\abs{a}_{F'}=\abs{\mathrm N_{F'/F}(a)}_F$.

\subsubsection{Smooth functions}
If $F$ is $\R$ or $\C$, we say that a function $f$ 
defined on an open subset
of the $n$-dimensional affine space~$F^n$ 
is \emph{smooth} if it is $\mathrm C^\infty$.
If $F$ is non-archimedean, this will mean that $f$ is locally
constant. This notion is local and extends to functions defined 
on open subsets of $F$-analytic manifolds. Observe moreover
than for any open subset~$U$ of~$F^n$ and any non-vanishing
analytic function~$f$  on~$U$, the function $x\mapsto \abs {f(x)}$
is smooth.

On a compact $F$-analytic manifold~$X$, a smooth function~$f$
has a sup-norm $\norm{f}=\sup_{x\in X}\abs{f(x)}$. In the
archimedean case, using charts (so, non-canonically),
we can also measure norms of derivatives
and define norms $\norm{f}_r$ (measuring the maximum of sup-norms
of all derivatives of~$f$ of orders~$\leq r$).
In the ultrametric case, using a distance~$d$,
we can define a norm $\norm{f}_1$
as follows:
\[ \norm{f}_1 = \norm{f}\big( 1 + \sup _{f(x)\neq f(y)} \frac
{1}{d(x,y)}\big) .\]
For $r> 1$, we define $\norm{f}_r=\norm{f}_1$.

\subsubsection{Metrics on line bundles}\label{subsubsec.metrics}
Let $X$ be an analytic variety over a locally compact valued field $F$
and let $\mathscr L$ be a line bundle on~$X$. 
We define a metric on $\mathscr L$
to be a collection of functions $\mathscr L(x)\ra\R_+$,
for all $x\in X$, 
denoted by $\ell\mapsto \norm{\ell}$
such that 
\begin{itemize}
\item for $\ell\in\mathscr L(x)\setminus\{0\}$, $\norm{\ell}>0$;
\item
for any $a\in F$, $x\in X$ and any $\ell\in\mathscr L(x)$,
$\norm{a\ell}=\abs{a}\norm{\ell}$ ;
\item for any open subset $U\subset X$ and any section $\ell\in
\Gamma(U,\mathscr L)$, the function $x\mapsto \norm{\ell(x)}$
is continuous on~$U$.
\end{itemize}
We say that a metric on a line bundle $\mathscr L$
is smooth if for any non-vanishing local section
$\ell\in\Gamma(U,\mathcal L)$, the function $x\mapsto \norm{\ell(x)}$
is smooth on~$U$.

For a metric to be smooth, it suffices that there exists
an open cover~$(U_i)$ of~$X$, and, for each~$i$, a non-vanishing 
section $\ell_i\in\Gamma(U_i,\mathcal L)$ such that the
function $x\mapsto\norm{\ell_i(x)}$ is smooth on~$U_i$.
Indeed, let $\ell\in\Gamma(U,\mathcal L)$ be a local non-vanishing
section of~$\mathcal L$; for each~$i$ there is a non-vanishing
regular function~$f_i\in\mathscr O_X(U_i\cap U)$ such
that $\ell=f_i\ell_i$ on~$U_i\cap U$, hence 
$\norm{\ell}=\abs{f_i}\norm{\ell_i}$. Since the absolute value
of a non-vanishing regular function is smooth, $\norm{\ell}$
is a smooth function on~$U_i\cap U$. Since this holds for all~$i$,
$\norm{\ell}$ is smooth on~$U$.

The trivial line bundle $\mathscr O_X$ admits a canonical metric,
defined by $\norm{1}=1$.
If $\mathscr L$ and $\mathscr M$ are two metrized line bundles on~$X$,
there are metrics on $\mathscr L\otimes\mathscr M$ and on
$\mathscr L^\vee$ defined by
\[ \norm{\ell\otimes m}=\norm{\ell}\norm{m},
\qquad \norm{\phi}=\abs{\phi(\ell)} / {\norm{\ell}}, \]
with $x\in X$,  $\ell\in\mathscr L(x)$, $m\in\mathscr M(x)$,
and $\phi\in\mathscr L^\vee(x)$.

\subsubsection{Divisors, line bundles and metrics}\label{subsubsec.Q-Cartier}
The theory of the preceding paragraph also applies if
$X$ is an analytic subspace of some $F$-analytic manifold, \emph{e.g.}, 
an algebraic variety, even if it possesses singularities.
Recall that by definition, a function on such a space~$X$ is
smooth if it extends to a smooth function
in a neighborhood of~$X$ in the ambient space.

Let $D$ be an effective Cartier divisor on~$X$ and let $\mathscr O_X(D)$
be the corresponding line bundle. It admits a canonical
section~$\mathsec f_D$, whose divisor 
is equal to~$D$. If $\mathscr O_X(D)$ is endowed with a metric,
the function~$\norm{\mathsec f_D}$ is positive on~$X$ and
vanishes along~$D$.

More generally, let $D$ be an effective $\Q$-divisor,
that is, a linear combination of irreducible divisors
with rational coefficients such that a multiple~$nD$,
for some positive integer~$n$, is a Cartier divisor.
By a metric on~$\mathscr O_X(D)$ we mean
a metric on $\mathscr O_X(nD)$. By $\norm{\mathsec f_D}$,
we mean the function $\norm{\mathsec f_{nD}}^{1/n}$.
It does not depend on the choice of~$n$.

\subsubsection{Metrics defined by a model}
Here we assume that $F$ is non-archimedean, and let
$\mathfrak o_F$ be its ring of integers.
Let $X$ be a projective variety over~$F$ and $L$
a line bundle on~$X$. Choose
 a projective flat $\mathfrak O_F$-scheme $\mathscr X$ and a line bundle
$\mathscr L$ on~$\mathscr X$ extending $X$ and $L$.
These choices determine a metric on the line bundle defined by $\mathscr
L$ on the analytic variety $X(F)$, by the following recipe:
for $x\in X(F)$ let $\tilde x\colon \Spec\mathfrak o_F\ra\mathscr
X$ be the unique morphism extending~$x$; by definition, the
set of $\ell\in L(x)$ such that $\norm{\ell}\leq 1$ is
equal to $\tilde x^*\mathscr L$, which is a lattice in $L(x)$.

Let~$\mathscr U$ be an open subset of~$\mathscr X$ over which
the line bundle~$\mathscr L$ is trivial and let $\eps\in\Gamma(\mathscr U,\mathscr L)$ be a trivialization of~$\mathscr L$ on~$\mathscr U$.
Then, for any point $x\in \mathscr U(F)$ such that $x$ 
extends to a morphism $\tilde x\colon\Spec\mathfrak o_F\ra\mathscr U$,
one has $\norm{\eps(x)}=1$. Indeed, $\tilde x^*\eps$ is a basis
of the free $\mathfrak o_F$-module~$\tilde x^*\mathscr L$.
Since $\mathscr X$ is projective, restriction to the
generic fiber identifies the set $\mathscr U(\mathfrak o_F)$ with
a compact open subset of~$\mathscr U_F(F)$,
still denoted by $\mathscr U(\mathfrak o_F)$.
Observe that these compact open subsets cover~$X(F)$.

Let $F'/F$ be a field extension. A model of $X$ determines
a model of $X_{F'}=X\otimes_F F'$ over the ring $\mathfrak o_{F'}$,
and thus metrics on the analytic variety $X(F')$.
Conversely, note that a model of $X_{F'}$ determines metrics
on $X(F')$ and hence, by restriction, also on $X(F)$.

\subsubsection{Example: projective space}\label{exam.projective-space.metric}
Let $F$ be a locally valued field and $V$ a finite-dimensional
vector space over~$F$. Let~$\abs\cdot$ be a norm on~$V$.
Let~$\P(V^*)=\Proj\Sym^\bullet V^*$ be the projective space 
of lines in~$V$; denote by~$[x]$ the point of~$\P(V^*)$
associated to a non-zero element~$x\in V$, \emph{i.e.},
the lines it generates.
This projective space carries a tautological
ample line bundle, denoted $\mathscr O(1)$, whose
space of global sections is precisely~$V^*$.
The formula
\[ \norm{\ell([x])} = \frac {\abs{\ell(x)}}{\norm{x}}, \]
for $\ell\in V^*$ and $x\in V$,
defines a norm on the $F$-vector space~$\mathscr O(1)_{[x]}$.
These norms define a metric on~$\mathscr O(1)$.

Assume moreover that $F$ is non-archimedean and let~$\mathfrak o_F$
be its ring of integers. If the norm of~$V$ takes
its values in the set~$\abs{F^\times}$,
the unit ball~$\mathscr V$ of~$V$ 
is an~$\mathfrak o_F$-submodule of~$V$ which is free of rank~$\dim V$
(\cite{weil1967}, p.~28, Proposition~6 and p.~29).
Then, the $\mathfrak o_F$-scheme $\P(\mathscr V^\vee)$
is a model of~$\P(V^\vee)$ and the metric on~$\mathscr O(1)$
is the one defined by this model and the canonical extension 
of the line bundle~$\mathscr O(1)$ it carries.

\subsubsection{Volume forms}\label{subsubsec.volume-forms}
Let $X$ be an analytic manifold over a local field $F$.
To simplify the exposition, assume that $X$ is equidimensional
and let~$n$ be its dimension.
It is standard when $F=\R$ or $\C$, and also true in general,
that any $n$-form $\omega$ on an open subset $U$ of~$X$ defines
a measure, usually denoted $\abs\omega$, on $U$, through the following
formula. Choose local coordinates $x_1,\dots,x_n$ on $U$ and write
\[ \omega= f(x_1,\dots,x_n) \,\mathrm dx_1\wedge \dots\wedge \mathrm dx_n.
\]
These coordinates allow to identify (a part of) $U$ with an open
subset of $F^n$ which we endow with the product measure $\mu^n$.
(Recall that a Haar measure has been fixed on~$F$
in~\ref{subsubsec.local-fields}.)
It pulls back to a measure which we denote
\[ \abs{\mathrm dx_1} \, \cdots \, \abs{\mathrm dx_n} \]
and by definition, we let 
\begin{equation}  \label{eq.abs.omega}
 \abs\omega= \abs{f(x_1,\dots,x_n)}\,
\abs{\mathrm dx_1} \, \cdots \, \abs{\mathrm dx_n} .\end{equation}
By the change of variables formula in multiple integrals,
this is independent of the choice of local coordinates.

\subsubsection{Metrics and measures}\label{ss.metrics-measures}
Certain manifolds~$X$ possess a canonical
(up to scalar) non-vanishing~$n$-form,
sometimes called a \emph{gauge form}.
Examples are analytic groups, or Calabi-Yau varieties. 
This property leads to the definition of a canonical measure on~$X$
(again, up to a scalar).
In the case of groups, this has been studied by Weil in the
context of Tamagawa numbers (\cite{weil82});
in the case of Calabi-Yau varieties, this measure has been used
by Batyrev to prove that smooth birational Calabi-Yau varieties
have equal Betti numbers (\cite{batyrev99}).

Even when $X$ has no global $n$-forms, it is still possible to define a measure
on $X$ provided the canonical line bundle $\omega_X=\bigwedge^n \Omega^1_X$
is endowed with a \emph{metric}. Indeed, we may then
attach to any local non-vanishing $n$-form
$\omega$ the local measure $\abs\omega/\norm\omega$. It
is immediate that these measures
patch and define a measure~$\tau_X$ on the whole of~$X$. This is a Radon
measure locally equivalent to any Lebesgue measure. 
In particular, if $X$ is compact, its  volume
with respect to this measure is a positive real number.
This construction is classical in differential geometry
as well as in arithmetic geometry (see~\cite{serre1981}, bottom
of page~146);
its introduction in the context of the counting problem
of points of bounded height on algebraic varieties over number
fields is due to E.~Peyre (\cite{peyre95}).

\subsubsection{Singular measures}
For the study of integral points, 
we will have to consider several variants of this construction.
We assume here that $X$ is an algebraic variety over a local field~$F$
and that we are given an auxiliary effective Cartier $\Q$-divisor~$D$ in~$X$.
Let $U=X\setminus\abs D$ be the complement of the support of~$D$
in~$X$.

Let us first consider the case where $D$ is a Cartier divisor
and assume that the line bundle $\omega_X(D)$ is endowed with
a \emph{metric}. 
If $\mathsec f_D$ is the canonical section of~$\mathcal O_X(D)$
and $\omega$ a local $n$-form, we can then consider the nonnegative 
function~$\norm{\mathsec f_D\omega}$; it vanishes on~$\abs D$.
In the general case, we follow the conventions
of~Section\ref{subsubsec.Q-Cartier} and will freely
talk of the function~$\norm{\omega\mathsec f_D}$,
defined as $\norm{\omega^n\mathsec f_{nD}}$ where $n$ is any
positive integer such that~$nD$ is a Cartier divisor,
assuming the line bundle $\omega_X^{\otimes n}(nD)$ is
endowed with a metric. 

Then, the measures $\abs\omega/\norm{\mathsec f_D\omega}$, 
for $\omega$ any local non-vanishing $n$-form
$\omega$, patch and define a measure~$\tau_{(X,D)}$ 
on the open submanifold~$U(F)$.
This measure is a Radon measure on~$U(F)$, locally equivalent to the Lebesgue
measure.
However, the  volume of~$U(F)$ is infinite;
this is in particular the case if $D$ is a divisor (with integer
coefficients) and the smooth locus of~$D$ has $F$-rational points.

We shall always identify this measure and the (generally ``singular'',
\emph{i.e.}, non-Radon) measure on~$X(F)$ obtained by push-out.

\subsubsection{Example: gauge forms}
Let us show how the definition of a measure using a gauge form
can be viewed as a particular case of this construction.
Let~$\omega$ be a meromorphic differential form on~$X$; let
$D$ be the \emph{opposite} of its divisor. One can write
$D=\sum d_\alpha D_\alpha$,
where the $D_\alpha$ are codimension~$1$ irreducible subvarieties in~$X$
and $d_\alpha$ are integers. 
Let $U=X\setminus \abs D$ be the complement of the support of~$D$.
In other words, $\omega$
defines a trivialization of the line bundle $\omega_X(D)$.
There is a unique metrization of~$\omega_X(D)$ such that~$\omega$,
viewed as a global section of~$\omega_X(D)$,
has norm~$1$ at every point in~$X$. 
The measure on the manifold~$U(F)$
defined by this metrization coincides with the measure~$\abs\omega$
defined by $\omega$ as a gauge form.

Moreover, if the line bundle $\mathcal O_X(D)$ is endowed
with a metric, then so is the line bundle~$\omega_X$.
In this case, the manifold~$X(F)$ is endowed with three measures,
$\tau_U$, $\tau_X$ and~$\tau_{(X,D)}$. Locally, one has:
\begin{itemize}
\item $\tau_X=\abs\omega/\norm{\omega}$;
\item the measure $\tau_U$ is its restriction to~$U$;
\item $\tau_{(X,D)}=\abs\omega=\abs\omega/\norm{\mathsec f_D\omega}
      = \norm{\mathsec f_D}^{-1} \tau_X$. 
\end{itemize}
Note that on each open, relatively compact subset of~$X(F)\setminus \abs D$,
the measures $\tau_X$ and~$\tau_{(X,D)}$ are comparable.

\subsubsection{Example: compactifications of algebraic groups}\label{subsubsec.haar-tamagawa}
We keep the notation of the previous paragraph,
assuming moreover that $U$ is an algebraic group~$G$ over~$F$
of which $X$ is an equivariant compactification,
and that the restriction to~$G$ of~$\omega$ is invariant.
If we consider~$\omega$ as a gauge form, it then defines
a invariant measure $\abs\omega$ on~$G(F)$, in other words,
a Haar measure on this locally compact group,
and also $\tau_{(X,D)}$.

\subsubsection{Residue measures}\label{subsubsec.residue-measures}
Let us return to the case of a general manifold~$X$.
Let $Z$ be a closed submanifold of $X$. To get a measure on~$Z$,
we need as before a metrization on $\omega_Z$. However, the ``adjunction
formula'' 
\[ \omega_X|_Z \simeq  \omega_Z \otimes \det \mathcal N_Z (X)  \]
implies that given a metric on $\omega_X$, it suffices
to endow the determinant of the normal bundle of $Z$ in $X$
with a metric.
The case where $Z=D$ is a divisor (\emph{i.e.},
locally in charts, $Z$ is defined by the vanishing of a coordinate) 
is especially interesting.
In that case, one has
\[ \omega_D = \omega_X(D)|_D \]
and a metric on $\omega_X$ plus a metric on $\mathcal O_X(D)$
automatically define a metric on $\omega_D$, hence a measure on
$D$.

Let us give an explicit formula for this measure. Let
$\omega\in\Gamma(\omega_D)$ be  a local $(n-1)$-form and $\tilde\omega$
any lift of $\omega$ to $\bigwedge^{n-1}\Omega^1_X$.
If $u\in\mathcal O(-D)$ is a local equation for $D$, 
the image of $\omega$ in $\omega_X(D)|_D$ is nothing but
the restriction to~$D$ of the differential form with logarithmic
poles
$\tilde\omega\wedge u^{-1}\mathrm du$. Let $\mathsec f_D$ be
the canonical section of $\mathcal O_X(D)$. Then,
\[ \norm{\omega}=\norm{\tilde\omega\wedge u^{-1}\mathrm du}
  = \norm{\tilde\omega\wedge \mathrm du}
  \, \norm{u^{-1}} . \]
By definition, $\mathsec f_D$ corresponds locally to the function~1,
hence, for $x\not\in D$, 
\[ \norm{u^{-1}(x)} = \frac{1}{\abs {u(x)}} \norm{\mathsec f_D(x)} \] 
and this possesses a finite limit when $x$ approaches~$D$,
by the definition of a metric.
(As a section of $\mathscr O_X(D)$, $\mathsec f_D$ vanishes at order~$1$
on $D$). Moreover, the function $\lim \norm{\mathsec f_D}/\abs u$
on~$D$ defined by 
\[ x\mapsto \lim_{\substack{y\ra x\\ y\not\in D}}
          \norm{\mathsec f_D(y)}/\abs{u(y)} \]
is continuous and positive on~$D$.

By induction, if $Z$ is the transverse intersection of smooth divisors $D_j$
($1\leq j\leq m$),
with metrizations on all $\mathscr O_X(D_j)$, we have a similar formula:
\begin{equation}
\label{eq.abs.omega.D}
 \norm{\omega} = \norm{\tilde\omega\wedge \mathrm du_1\wedge\cdots \wedge \mathrm du_m}
     \lim \frac{\norm{\mathsec f_{D_1}}}{\abs{u_1}}\, \cdots
   \lim \frac{\norm{\mathsec f_{D_m}}}{\abs{u_m}},
\end{equation}
where $u_1,\ldots,u_m$ are local equations for the divisors
$D_1,\ldots,D_m$.

\subsubsection{Residue measures in an algebraic context}
\label{sss.residue-algebraic}
Let $X$ be a smooth algebraic variety over~$F$; endow the
canonical line bundle $\omega_X$ on~$X(F)$ with a metric.
Let~$F'$ be a finite Galois extension of~$F$ and 
let~$D_j$, for $1\leq j\leq m$, be smooth
irreducible divisors on~$X_{F'}$, whose union is
a divisor~$D$ with strict normal crossings which is defined over~$F$.
Then the intersection $Z=\bigcap_{j=1}^m D_j$ is defined over~$F$;
if $Z(F)\neq\emptyset$, then $Z(F)$ is a smooth $F$-analytic submanifold
of $X(F)$, of dimension~$\dim Z=\dim X-m$.
Moreover, the normal bundle of~$Z$ in~$X_{F'}$
is isomorphic to the restriction
to~$Z$ of the vector bundle $\bigoplus_{j=1}^m\mathscr O_{X'}(D_j)$,
hence the isomorphism $\det\mathscr N_Z(X)\simeq\mathscr O_X(D)$,
at least after extending the scalars to~$F'$.

Endow the line bundles $\mathscr O_{X_{F'}}(D_j)$  
on the ${F'}$-analytic manifold~$X({F'})$ with metrics. 
By the previous formulas, we obtain from the 
metrics on~$\mathscr O_{X_{F'}}(D_j)$ a metric on
the determinant of the normal bundle of~$Z({F'})$ in the manifold~$X({F'})$.
The restriction of this metric to~$Z(F)$ gives us a metric
on the determinant of the normal bundle of~$Z(F)$ in~$X(F)$. 

Accordingly,
we obtain a positive Radon measure~$\tau_Z$ on~$Z(F)$
which is locally comparable to any Lebesgue measure.
In particular, one has $\tau_Z(Z(F))=0$ if and only if $Z(F)=\emptyset$.

\subsection{Adeles of number fields: metrics and heights}

\subsubsection{Notation}

We specify some common notation concerning number fields
and adeles that we will use throughout this text.

Let $F$ be a number field and 
$\Val(F)$ the set of places of~$F$.
For $v\in \Val(F)$, we let $F_v$ be the $v$-adic completion of~$F$.
This is a local
field; its absolute value, defined as in~\S\ref{subsubsec.local-fields},
is denoted by~$\abs\cdot_v$. 
With these normalizations, the \emph{product formula} holds. Namely,
for all $a\in F^*$, one has
$\prod_{v\in\Val(F)} \abs a_v=1$,
where only finitely many factors differ from~$1$.

Let $v\in\Val(F)$. The absolute value $\abs{\cdot}_v$
is archimedean when $F_v=\R$ or $F_v=\C$; we will
say that $v$ is infinite, or archimedean.
Otherwise, the absolute value~$\abs\cdot_v$ is
ultrametric and the place $v$ is called finite, or non-archimedean.

Let $v$ be a finite place of~$F$. 
The set~$\mathfrak o_v$ of all $a\in F_v$
such that $\abs a_v\leq 1$ is a subring of~$F_v$,
called the ring of $v$-adic integers. Its subset $\mathfrak m_v$
consisting of all $a\in F_v$ such that $\abs a_v<1$ is 
its unique maximal ideal. The residue field $\mathfrak o_v/\mathfrak m_v$
is denoted by~$k_v$. It is a finite field and we write~$q_v$
for its cardinality. The ideal~$\mathfrak m_v$
is principal; a generator will be called a uniformizing
element at~$v$. For any such element~$\varpi$, one has
$\abs\varpi_v=q_v^{-1}$.

Fix an algebraic closure~$\bar F$ of~$F$. The group~$\Gamma_F$
of all $F$-automorphisms of~$\bar F$ is called the absolute
Galois group of~$F$.
For any finite place~$v\in\Val(F)$, fix an extension~$\abs{\cdot}_{\bar v}$ of
the absolute value~$\abs\cdot_v$ to~$\bar F$.
The subgroup of~$\Gamma_F$ consisting of all $\gamma\in\Gamma_F$
such that $\abs{\gamma(a)}_{\bar v}=\abs{a}_{\bar v}$
for all $a\in\bar F$ is called the decomposition subgroup of~$\Gamma_F$
at~$v$ and is denoted~$\Gamma_v$.

These data determine an algebraic closure~$\bar k_v$ of the residue field~$k_v$,
together with a surjective group homomorphism $\Gamma_v\ra\Gal(\bar k_v/k_v)$.
Its kernel~$\Gamma^0_v$ is called the inertia subgroup of~$\Gamma_F$ at~$v$.
Any element in~$\Gamma_v$ 
mapping to the Frobenius automorphism $x\mapsto x^{q_v}$
in~$\Gal(\bar k_v/k_v)$ is called an arithmetic Frobenius element at~$v$;
its inverse is called a geometric Frobenius element at~$v$.
The subgroups~$\Gamma_v$ and~$\Gamma_v^0$, and Frobenius elements, depend 
on the choice of the chosen extension of the absolute value~$\abs\cdot_v$;
another choice changes them by conjugation in~$\Gamma_F$.

The ring of adeles~$\AD_F$ of the field~$F$ is the restricted product
of all local fields~$F_v$, for $v\in\Val(F)$,
with respect to the subrings $\mathfrak o_v$ for finite places~$v$.
It is a locally compact topological ring and carries
a Haar measure~$\mu$. The quotient space~$\AD_F/F$ is compact;
we shall often assume that $\mu$ is normalized so that
$\mu(\AD_F/F)=1$.

\subsubsection{The adelic space of an algebraic variety}
Let $U$ be an algebraic variety over~$F$.
The space $U(\AD_F)$ of adelic points of~$U$ has a natural locally
compact topology which we recall now.
Let $\mathscr U$ be a model of~$U$ over the integers of~$F$, 
\emph{i.e.}, a scheme which is flat and of finite type over~$\Spec\mathfrak o_F$
together with an isomorphism of~$U$ with $\mathscr U\otimes F$.
The natural maps $\AD_F\ra F_v$ induce a map from
$U(\AD_F)$ to $\prod_{v\in \Val(F)} U(F_v)$. It is not surjective
unless $U$ is proper over~$F$; indeed its image --- usually
called the restricted product --- can be described
as the set of all $(x_v)_v$ in the product such that
$x_v\in\mathscr U(\mathfrak o_v)$ for almost all finite places~$v$.
Since two models are isomorphic over a dense open subset 
of~$\Spec\mathfrak o_F$,
observe also that this condition is independent of the choice of the
specific model~$\mathscr U$.

We endow each $U(F_v)$ with the natural $v$-adic topology;
notice that it is locally compact.
Moreover, for any finite place~$v$,
$\mathscr U(\mathfrak o_v)$ is open and compact in~$U(F_v)$ for this topology.

We then endow $U(\AD_F)$ with the ``restricted product topology'',
a basis of which is given by products $\prod_v \Omega_v$, 
where for each~$v\in\Val(F)$, $\Omega_v$ is an open subset of~$U(F_v)$,
subject to the additional condition that $\Omega_v=\mathscr U(\mathfrak o_v)$
for almost all finite places~$v$.
Let $\Omega\subset U(\AD_F)$ be any subset of the form $\prod_v \Omega_v$,
with $\Omega_v\subset U(F_v)$, such that $\Omega_v=\mathscr U(\mathfrak o_v)$
for almost all finite places~$v$.
Then $\Omega$ is compact if and only each~$\Omega_v$ is compact.
It follows that the topology on~$U(\AD_F)$ is locally compact.

\subsubsection{Adelic metrics}\label{subsubsec.adelic-metrics}
Let $X$ be a proper variety  over a number field~$F$
and
 $\mathscr L$ a line bundle on~$X$.
An adelic metric on~$\mathscr L$ is a collection of $v$-adic metrics
on the associated line bundles on the $F_v$-analytic varieties $X(F_v)$,
for all places $v$ in~$F$,
which, except for a finite number of them, are defined by a single
model $(\mathscr X,\mathscr L)$ over the ring of integers of~$F$. 

By standard properties of schemes of finite presentation, any two
flat proper models are isomorphic at almost all places
and will therefore define the same $v$-adic metrics at these places.

\subsubsection{Example: projective space}
\label{exam.projective-space.adelic.metric}
For any valued field~$F$, let us endow the vector-space~$F^{n+1}$ 
with the norm given by 
$\abs{(x_0,\ldots,x_n)}=\max(\abs{x_0},\ldots,\abs{x_n})$.
When $F$ is a number field and $v\in\Val(F)$,
the construction~\ref{exam.projective-space.metric} on the field~$F_v$
furnishes
a metric on the line bundle~$\mathscr O(1)$ on~$\mathbf P^n$.
This collection of metrics is an adelic metric on~$\mathscr O(1)$,
which we will call the \emph{standard adelic metric}.

\subsubsection{Extension of the ground field}
If $F'$ is any finite extension of~$F$, observe that 
an adelic metric on the line bundle~$\mathscr L\otimes F'$
over~$X\otimes F'$ induces by restriction an adelic metric on
the line bundle~$\mathscr L$ over~$X$. We only need
to check that the family of metrics defined by restriction
are induced at almost all places by a model of~$X$ on~$F$.

Indeed, fix a model~$(\mathscr X',\mathscr L')$ of~$(X\otimes F',
\mathscr L\otimes F')$ over~$\Spec\mathfrak o_{F'}$,
as well as a model~$(\mathscr X,\mathscr L)$ over~$\Spec\mathfrak o_F$.
There is a non-zero integer~$N$ 
such that the identity map $X\otimes F'\ra X\otimes F'$
extends to an isomorphism $\mathscr X\otimes \mathfrak o_{F'}[\frac1N]
\ra\mathscr X'\otimes\mathfrak o_{F'}[\frac1N]$.
Consequently, at all finite places of~$F'$ which do not divide~$N$,
both models define the same metric on~$L\otimes F'$. 
It follows that the metrics of~$L$ are defined 
by the model~$(\mathscr X,\mathscr L)$
at all but finitely many places of~$F$.

\subsubsection{Heights}\label{subsubsec.heights}
Let $X$ be a proper variety over a number field~$F$
and let $\mathscr L$ be a line bundle on~$X$ endowed
with an adelic metric.
Let $x\in X(F)$ and let $\ell$ be any non-zero element in $\mathscr L(x)$.
For almost all places~$v$, one has $\norm{\ell}_v=1$;
consequently, the product $\prod_v \norm{\ell}_v$
converges absolutely. It follows from the product formula
that its value does not depend on the choice of~$\ell$.
We denote it by $H_{\mathscr L}(x)$ and call it
the (exponential) \emph{height} of~$x$ with respect to the 
metrized line bundle~$\mathscr L$.

As an example, assume that $X=\mathbf P^n$ and $\mathscr L=\mathscr O(1)$,
endowed with its standard adelic metric (\ref{exam.projective-space.adelic.metric}). Then for any point $x\in\mathbf P^n(F)$ with
homogeneous coordinates $[x_0\cdots:x_n]$ such that $x_0\neq 0$,
we can take the element~$\ell$ to be the value at~$x$ of
the global section~$X_0$. Consequently, the definition
of the standard adelic metric implies that
\[ H_{\mathscr L}(x)
  = \prod_{v\in\Val(F)} \left(\frac{\abs{x_0}_v}{\max(\abs{x_0}_v,\ldots,\abs{x_n}_v)}\right)^{-1}
  =  \prod_{v\in\Val(F)} \max(\abs{x_0}_v,\ldots,\abs{x_n}_v), \]
in view of the product formula $\prod_{v\in\Val(F)}\abs{x_0}_v=1$.
We thus recover the classical (exponential) height of the point~$x$
in the projective space.

In general, the function 
\[ \log H_{\mathscr L} \colon X(F)\ra \R \]
is a representative of the class of height functions 
attached to the line bundle~$\mathscr L$
using Weil's classical ``height machine'' 
(see~\cite{hindry-silverman2000}, especially Part~B, \S10).

When $\mathscr L$ is ample, the set of points $x\in X(F)$
such that $H_{\mathscr L}(x)\leq B$ is finite for any real
number~$B$ (Northcott's theorem). 

\subsection{Heights on adelic spaces}
\label{subsec.heights.adelic}
Let $\mathsec f$ be a non-zero global section of~$\mathscr L$
and let~$Z$ denote its divisor.
We extend the definition of $H_{\mathscr L}$
to the adelic space $X(\AD_F)$ by defining
\begin{equation}
\label{eq.height.f}
  H_{\mathscr L,\mathsec f}(\mathbf x) = \big(\prod_v \norm{\mathsec f}_v(x)
\big)^{-1}
\end{equation}
for any $\mathbf x\in X(\AD_F)$ such that the infinite product makes sense.

To give an explicit example, 
assume again that $X$ is the projective space~$\mathbf P^n$
and that $\mathscr L$ is the line bundle~$\mathscr O_{\mathbf P}(1)$ endowed
with its standard adelic metric;
let us choose $\mathsec f=x_0$.
Then, $U=X\setminus \div(s)$ can be identified to the set of points~$[1:\cdots:x_n]$ of~$\mathbf P^n$ whose first homogeneous coordinate 
is equal to~$1$, \emph{i.e.}, to the affine space~$\mathbf A^n$.
For any  $\mathbf x=(x_v)_\in U(\AD_F)$, given by $\mathbf x=[1:\mathbf x_1:\cdots:\mathbf x_n]$, with
$\mathbf x_i=(x_{i,v})_v\in \AD_F$, one has
\[ H_{\mathscr L,\mathsec f}(x)
 =  \prod_{v\in\Val(F)} \max(1,\abs{x_{1,v}}_v,\ldots,\abs{x_{n,v}}_v). \]

\begin{lemm}\label{lemm.northcott.adelic}
Let $U$ be any open subset of~$X$.

\begin{enumerate} \def\theenumi{\roman{enumi}}
\item For any $\mathbf x\in X(\AD_F)$,
the infinite product defining~$H_{\mathscr L,\mathsec f}(\mathbf x)$ 
converges to an element of $(0,+\infty]$.
\item 
The resulting function is lower semi-continuous and admits a positive
lower bound on the adelic space $X(\AD_F)$.
\item
If $\mathsec f$ does not vanish on~$U$,
then the restriction of~$H_{\mathscr L,\mathsec f}$
to $U(\AD_F)$ is continuous (for the topology of~$U(\AD_F)$).
\item
Assume that $X=U\cup Z$.  Then, for any real number~$B$, 
the set of points
$\mathbf x\in U(\AD_F)$ such that $H_{\mathscr L,\mathsec f}(\mathbf x)\leq
B$ is a compact subset of~$U(\AD_F)$.
\end{enumerate}
\end{lemm}
\begin{proof}
For any place~$v$, let us set $c_v=\max(1,\norm{\mathsec f}_v)$.
One has $c_v<\infty$ since $X(F_v)$ is compact and the
function $\norm{\mathsec f}_v$ on $X(F_v)$ is continuous.
Moreover, it follows from
the definition of an adelic metric that $\norm{\mathsec f}_v=1$
for almost all~$v$, so that $c_v=1$ for almost all places~$v$. 

Consequently, 
the infinite series $\sum_{v\in\Val(F)} \log\norm{\mathsec f_v}^{-1}$
has (almost all of) its terms nonnegative; it converges to a real
number or to~$+\infty$. This shows that the infinite
product $H_{\mathscr L,\mathsec f}(\mathbf x)$
converges to an element in~$(0,+\infty]$.
Letting $C=1/\prod_v c_v$,
one has $H_{\mathscr L,\mathsec f}(\mathbf x)\geq C$
for any~$\mathbf x\in X(\AD_F)$, which proves~(i) and
the second half of~(ii).

Let $V_1$ be the set of places $v\in\Val(F)$ such that $c_v>1$;
one has
\[ \log H_{\mathscr L,\mathsec f}(\mathbf x)
 = \sup_{V_1\subset V\subset\Val(F)}  \sum_{v\in V} \log\norm{\mathsec f_v}^{-1},
\]
where $V$ runs within the finite subsets of $\Val(F)$ containing~$V_1$.
This expression shows that 
the function $\mathbf x\mapsto \log H_{\mathscr L,\mathsec f}(\mathbf x)$
is a supremum of lower semi-continuous functions, hence
is lower semi-continuous on~$X(\AD_F)$, hence the remaining
part of Assertion~(ii).

Let us now prove~(iii). Let $\mathscr X$
be a flat projective model of~$X$ over~$\mathfrak o_F$,
let $\mathscr Z$  be the Zariski closure of~$Z$
and $\mathscr V=\mathscr X\setminus\mathscr Z$,
let $\mathscr D$ be the Zariski closure of~$X\setminus U$
and let $\mathscr U=\mathscr X\setminus\mathscr D$.

By definition of an adelic metric,
for almost all finite places~$v\in\Val(F)$,
$\norm{f}_v$ is identically equal to~$1$ on~$\mathscr U(\mathfrak o_v)$.
By definition of the topology of $U(\AD_F)$, this implies that
the function~$H_{\mathscr L,\mathsec f}$ 
is locally given by a finite product of continuous functions; 
it is therefore continuous.

Let us finally establish~(iv).
Consider integral models as above; since $U$ contains~$V$
by assumption, then $\mathscr U\supset \mathscr V$.

For any $\mathbf x=(x_v)_v\in X(\AD_F)$,
one has 
\[ \norm{\mathsec f}(x_v) = H_{\mathscr L,\mathsec f}(\mathbf x)^{-1}
  \prod_{w\neq v} \norm{\mathsec f}(x_v)^{-1}
 \geq C^{-1} H_{\mathscr L,\mathsec f}(\mathbf x)^{-1} .
\]
If moreover $H_{\mathscr L,\mathsec f}(\mathbf x)\leq B$,
then $\norm{\mathsec f}(x_v)\geq (BC)^{-1}$.
The set of points~$x_v\in X(F_v)$ satisfying this inequality
is a closed subset of the compact space~$X(F_v)$, hence is compact.
Moreover, this set is contained in $U(F_v)$ because,
by assumption, $\mathsec f$ vanishes on $X\setminus U$.
Consequently, this set is a compact subset of $U(F_v)$.

If the cardinality~$q_v$ of the residue field~$k_v$
is big enough, so that the $v$-adic metric on~$\mathscr L$
is defined by the line bundle~$\mathscr O(\mathscr Z)$ at the place~$v$,
then $-\log\norm{\mathsec f}_v(x_v)$ is a non-negative integer
times $\log q_v$. The inequality $-\log\norm{\mathsec f}_v(x_v)\leq\log(BC)$
then implies that $-\log\norm{\mathsec f}_v(x_v)=0$.

Let $E_B$ denote the set of points $\mathbf x\in U(\AD_F)$
such that $H_{\mathscr  L,\mathsec f}(\mathbf x)\leq B$.
By what we proved,  there exists a finite set of places~$V$
(depending on~$B$) such that
$E_B$ is the product of a compact subset of $\prod_{v\in V} U(F_v)$
and of $\prod_{v\not\in V} \mathscr V(\mathfrak o_v)$.
Since $\mathscr V(\mathfrak o_v)\subset\mathscr U(\mathfrak o_v)$,
we see that $E_B$ is relatively compact in $U(\AD_F)$.
By lower semi-continuity, it is also closed in~$U(\AD_F)$,
hence compact.
\end{proof}

We observe that the hypotheses in~(iii) and~(iv) are actually necessary.
They hold in the important case where
$U=X\setminus Z$; then, $H_{\mathscr L,\mathsec f}$
defines a continuous exhaustion of $U(\AD_F)$ by compact subsets.

\smallbreak
\subsection{Convergence factors and Tamagawa measures on adelic spaces}\label{subsec.global-tamagawa}
\nobreak
\subsubsection{Volumes and local densities} \label{ss.local.densities}
Let $X$ be a smooth, proper, and geometrically integral
algebraic variety over a number field~$F$. 
Fix an adelic metric on the canonical line bundle~$\omega_X$;
by the results recalled in~\ref{subsubsec.volume-forms},
for any place $v$ of~$F$, this induces  a measure~$\tau_{X,v}$ on $X(F_v)$
and its restriction~$\tau_{U,v}$ to~$U(F_v)$.

Let $U$ be a Zariski open subset of~$X$ and let $Z=X\setminus U$.
Fix a model $\mathscr X$ of~$X$, and
let $\mathscr Z$ be the Zariski closure of~$Z$ in~$\mathscr X$
and $\mathscr U=\mathscr X\setminus \mathscr Z$.

By a well-known formula going back to Weil (\cite{weil82},
see also~\cite{salberger98}, Corollary~2.15),
the equality
\begin{equation} \label{eq.weil-formula}
 \tau_{U,v}(\mathscr U(\mathfrak o_v)) = q_v^{-\dim X} \Card  \mathscr U(k_v) .
\end{equation}
holds for almost all non-archimedean places~$v$.

\subsubsection{Definition of an ${\rm L}$-function}

\begin{defi}\label{defi.P(u)}
We define $\EP (U)$  to be the following virtual
$\Q[\Gamma_F]$-module:
\[ \left[\mathrm H^0(U_{\bar F},\gm)/\bar F^* \right]_\Q - 
   \left[\mathrm H^1(U_{\bar F},\gm)\right]_\Q.
\] 
\end{defi}
It is an object of the Grothendieck group of the category
of finite dimensional $\Q$-vector spaces endowed with a continuous
action of~$\Gamma_F$, meaning that there is a finite extension~$F'$
of~$F$ such that the subgroup $\Gamma_{F'}$ acts trivially.
Such a virtual Galois-module 
(we shall often skip the word ``virtual'')
has an Artin L-function, given by an Euler product
\[ \mathrm L(s,\EP (U))= \prod_{\text{$v$ finite}} \mathrm L_v(s,\EP (U)),
\qquad
    \mathrm L_v(s,\EP (U))=\det\big(1-q_v^{-s}\Fr_v\big|\EP(U)^{\Gamma_v^0}\big)^{-1},
\]
where $\Fr_v$ is a geometric Frobenius element
and $\Gamma_v^0$ is an inertia subgroup at the place~$v$.

\begin{lemm}
For any finite place~$v$ of~$F$, $\mathrm L_v(s,\EP (U))$ 
is a positive real number. 
\end{lemm}
\begin{proof}
We have $\mathrm L_v(s,\EP (U))=q_v^{s\rang \EP (U)}f(q_v^s)$,
where the rational function
\[
f(X)=\det(X-\Fr_v|\EP (U))
\] 
is the virtual characteristic polynomial
of~$\Fr_v$ acting on~$\EP (U)$.
Since the action of~$\Gamma_F$ on~$\EP (U)$ factorizes
through a finite quotient, the rational function
has only roots of unity
for zeroes and poles.

By  irreducibility
of the cyclotomic polynomials~$\Phi_n$ in~$\Q[X]$, this
implies the existence of rational
integers $(a_n)_{n\geq 1}$, almost all zero,
such that 
\[ \det(X-\Fr_v|\EP (U)) = \prod_{n\geq 1} \Phi_n(X)^{a_n}. \]
From the inductive definition of the cyclotomic polynomials~$\Phi_n$,
namely $\prod_{d|n}\Phi_n(X)=X^n-1$,
we see that $\Phi_n(x)>0$ for any real number~$x>1$
and any positive integer~$n$. In particular, $f(x)>0$
for any real number~$x>1$ and $\mathrm L_v(1,\EP (U))>0$,
as claimed.
\end{proof}

(To simplify notation,
we shall also put $\mathrm L_v=1$ for any archimedean place~$v$.)
More generally, if $S$ is an arbitrary set of places of~$F$,
we define
\[ \mathrm L^S(s,\EP (U))
= \prod_{ v\not\in S } \mathrm L_v(s,\EP (U))
.\]
This Euler product converges for $\Re(s)>1$ to a holomorphic function of~$s$
in that domain;
by Brauer's theorem~\cite{brauer1947}, 
it admits a meromorphic continuation to the whole complex plane.
We denote by 
\[ \mathrm L^{S}_{*}(1,\EP (U))
      = \lim_{s\ra 1} \mathrm L^S(s,\EP (U)) (s-1)^{-r},
\qquad  r=\ord_{s=1} \mathrm L^S(s,\EP (U)), \]
its ``principal value'' at~$s=1$.

From the previous lemma, we deduce the following corollary.
\begin{coro}
For any finite set of places~$S$ of~$F$,
$\mathrm L^S_*(1,\EP (U))$ is a positive real number.
\end{coro}

We now show how this L-function furnishes renormalization factors
for the local measures~$\tau_{U,v}$.

\begin{theo}\label{prop.convergence}
In addition to the notation and assumptions of~\S\ref{ss.local.densities},
suppose that
\[ \mathrm H^1(X,\mathscr O_X)=\mathrm H^2(X,\mathscr O_X)=0. \]
Then 
$\mathrm L_v(1,\EP (U))\tau_{U,v}(\mathcal U(\mathfrak o_v))
=1+\mathrm O(q_v^{-3/2})$. In particular,
the infinite product 
\[ \prod_{v\not\in S} \mathrm L_v(1,\EP (U))
   \tau_{U,v}(\mathcal U(\mathfrak o_v)) \]
converges absolutely.
\end{theo}
\begin{rema}
For a smooth projective variety~$X$  and any integer~$i$,
the vector-space $\mathrm H^0(X,\Omega^i_X)$
defines a birational invariant of~$X$. 
Let us sketch the proof.
For any birational morphism $f\colon X\dashrightarrow X'$,
there are open subsets~$V\subset X$ and~$V'\subset X'$ 
whose complementary subsets have codimension at least~$2$
such that $f$ is defined on~$V$ and~$f^{-1}$ is defined on~$V'$.
The restriction morphism $\mathrm H^0(X,\Omega^i_X)\ra\mathrm H^0(V,\Omega^i_X)$ is then  injective since $V$ is dense, and surjective by the
Hartogs principle; similarly, the restriction morphism
$\mathrm H^0(X',\Omega^i_{X'})\ra\mathrm H^0(V',\Omega^i_{X'})$ is an isomorphism. 
Consequently, the corresponding regular maps 
$g\colon V\ra X'$ and $g'\colon V'\ra X$
define morphisms $\mathrm H^0(X',\Omega^i_{X'})\ra\mathrm H^0(X,\Omega^i_X)$
and $\mathrm H^0(X,\Omega^i_X)\ra\mathrm H^0(X',\Omega^i_{X'})$
both compositions of which equal the identity map,
hence are isomorphisms.

When the ground field is  the field of complex numbers, Hodge
theory identifies these vector-spaces with the conjugates
of the cohomology spaces $\mathrm H^i(X,\mathscr O_X)$.
Consequently, these spaces define birational invariants
of smooth complex projective varieties. Moreover,
by the Lefschetz  principle, this extends to smooth projective
varieties defined over a field of characteristic~$0$.
In particular,
the assumptions of the theorem are therefore 
conditions on the variety~$U$ and do not depend
on its smooth compactification~$X$.
\end{rema}
\begin{proof}
Blowing up the subscheme~$Z$ and resolving the singularities
of the resulting scheme, we obtain a smooth variety~$X'$, with
a proper morphism $\pi\colon X'\ra X$, which is an isomorphism
above~$U$, such that $Z'=\pi^{-1}(Z)$ is a divisor.

By the previous remark,
the variety~$X'$ satisfies 
$\mathrm H^1(X',\mathscr O_{X'})=\mathrm H^2(X',\mathscr O_{X'})=0$. 
Let $U'=\pi^{-1}(U)$; since~$\pi$ is an isomorphism,
any model $(\mathscr X',\mathscr Z')$ of~$(X',Z')$ 
will be such that $\mathscr U'=\mathscr X'\setminus\mathscr Z'$
is isomorphic to~$\mathscr U$, at least over a dense open
subset of~$\Spec\mathfrak o_F$. In particular, the
cardinalities of~$\mathscr U'(k_v)$ and~$\mathscr U(k_v)$
will be equal for almost all places~$v$.
In view of Weil's formula recalled as Equation~\eqref{eq.weil-formula},
this
implies that $\tau_{U',v}(\mathscr U'(\mathfrak o_v))=
\tau_{U,v}(\mathscr U(\mathfrak o_v))$ for almost all finite
places~$v$ of~$F$.

Consequently, assuming that~$Z$ is a divisor
does not reduce the generality of the argument which follows.

For almost all places~$v$, one has
\[ \tau_{U,v}(\mathcal U(\mathfrak o_v)) = q_v^{-\dim U} \Card  \mathcal U(k_v).\]
Evidently,
\[ \Card  \mathcal U(k_v) = \Card  \mathcal X(k_v) - \Card  \mathcal Z(k_v). \]

Let $\ell$ be a prime number.
By the smooth and proper base change theorems,
and by Poincaré duality,
the number of points of~$\mathscr X(k_v)$ can be computed,
for any finite place~$v$ of~$F$,
via the trace of a geometric Frobenius element~$\Fr_v\in\Gamma_v$ 
on the $\ell$-adic cohomology of~$Z_{\bar F}$,
as soon as 
$\mathscr X$ is smooth over the local ring of~$\mathfrak o_F$ at~$v$
and the residual characteristic at~$v$ is not equal to~$\ell$.
Specifically, for any such finite place~$v$, one has
the equality
\[ q_v^{-\dim X}\Card \mathcal X(k_v) = \sum_{i=0}^{2d} q_v^{-i/2}\tr(\Fr_v|\mathrm H^i(X_{\bar F},\Q_\ell)). \]
By Deligne's proof~\cite{deligne74} of the Weil conjectures (analogue of
the Riemann hypothesis) the eigenvalues of~$\Fr_v$
on $\mathrm H^i(X_{\bar F},\Q_\ell)$ are algebraic numbers with archimedean
eigenvalue~$q^{i/2}$. Therefore, the $i$th term of the sum above
is an algebraic number whose archimedean
absolute values are bounded by $\dim \mathrm H^i(X_{\bar F},\Q_\ell) q^{-i/2}$.
In particular, the sum of all terms corresponding to $i\geq 3$
is an algebraic number whose archimedean absolute values
are $\mathrm O(q_v^{-3/2})$ when $v$ varies through the
set of finite places of~$F$.

The term corresponding to $i=0$ is equal to~$1$ since $X$ is geometrically connected.

Moreover, 
it follows from Peyre's arguments (\cite{peyre95}, proof of Lemma~2.1.1
and Proposition~2.2.2) that one has 
\[ \mathrm H^1(X_{\bar F},\Q_\ell) = 0 \]
and that the cycle map induces an isomorphism of
$\Q_\ell[\Gamma_F]$-modules
\[ \Pic(X_{\bar F}) \otimes\Q_\ell \xrightarrow\sim
 \mathrm H^2(X_{\bar F},\Q_\ell(1)). \]

Let $\Pi(Z_{\bar F})$ denote the free $\Z$-module with basis
the set of irreducible components of $Z_{\bar F}$, endowed with
its Galois action.

\begin{lemm}
For almost all~$v$, one has the equality
\[ \Card  \mathcal Z(k_v) = n_v q_v^{\dim Z} + \mathrm O(q_v^{\dim Z-1/2}), \]
where $n_v=\tr(\Fr_v|\Pi(Z_{\bar F}))$ 
is the number of irreducible components of $\mathcal Z\otimes
k_v$ which are geometrically irreducible.
\end{lemm}
\begin{proof}[Proof of the lemma]\footnote{The proof
below uses Deligne's proof of Weil's conjecture but
the result can be deduced from the estimates
of Lang-Weil in~\cite{lang-w54}.}
In fact, we prove, for any scheme $\mathscr Z/\mathfrak o_F$
which is flat, of finite type, and 
whose generic fibre~$Z$ has dimension~$d$,
the following equality
\[ \Card \mathcal Z(k_v) = n_v q_v^d + \mathrm O(q_v^{d-1/2}), \]
where $n_v$ is the number of fixed points of~$\Fr_v$
in the set of $d$-dimensional irreducible components of~$Z_{\bar F}$.

First observe that there is a general upper bound
\[ \Card \mathcal Z(k_v) \leq C q_v^d , \]
for some integer~$C>0$. When $\mathcal Z$ is quasi-projective,
this follows for example from Lemma~3.9 
in~\cite{chambert-loir-t2000b}, where $C$ can be
taken to be the degree of the closure of~$Z$ 
in a projective compactification; the general case follows
from this since $\mathcal Z$ is assumed to be of finite type.

Consequently, adding or removing a subscheme of lower dimension
gives an equivalent inequality, so that
we may assume that $\mathcal Z$ is projective and equidimensional. 
By resolution of singularities, we may also assume that
$Z$ is smooth.

Let $\ell$ be a fixed prime number. As above, we can compute
$\Card\mathcal Z(k_v)$ using the $\ell$-adic cohomology
of~$Z_{\bar F}$: one has an upper bound
\[ q_v^{-d}\Card \mathcal Z(k_v) = \tr(\Fr_v|\mathrm H^i(Z_{\bar F},\Q_\ell))+\mathrm O(q_v^{-1/2}). \]
Moreover, 
$\mathrm H^{0}(Z_{\bar F},\Q_\ell) \simeq \Q_\ell^{\Pi(Z_{\bar F})}$,
where $\Pi(Z_{\bar F})$ denotes the set of connected
components of $Z_{\bar F}$, the Frobenius element~$\Fr_v$ acting
 trivially on this ring.
Consequently, 
\[  \tr(\Fr_v|\mathrm H^{0}(Z_{\bar F},\Q_\ell)) =  \tr(\Fr_v|\Pi(Z_{\bar F}))
\]
and the lemma follows from this.
\end{proof}

Let us return to the proof of Theorem~\ref{prop.convergence}.
By the preceding lemma, we have
\[ q_v^{-\dim X} \Card  \mathcal Z(k_v) = q_v^{-1}  \tr(\Fr_v|\Pi(Z_{\bar F}))
       + \mathrm O(q_v^{-3/2}), \]
hence
\[ q_v^{-\dim X} \Card  \mathcal U(k_v)
	= 1 - \frac{1}{q_v} \tr\big(\Fr_v\big|\Pi(Z_{\bar F})\otimes\Q_\ell
  - \mathrm H^2_{\text{\'et}}( X_{\bar F},\Q_\ell(1)) \big) + \mathrm O(q_v^{-3/2}). \]

Let us show that the trace appearing in the previous
formula is equal to $\tr(\Fr_v|\EP (U))$.
The order of vanishing/pole along any irreducible component
of $Z_{\bar F}$ defines an exact sequence of abelian sheaves
for the {\'e}tale site of $X_{\bar F}$:
\[ 1 \ra \gm \ra i_*\gm \xrightarrow{\ord} 
         \bigoplus j_{\alpha*} \Z \ra 0, \]
where $i\colon U_{\bar F}\ra X_{\bar F}$ is the inclusion 
and $j_\alpha\colon Z_\alpha\ra X_{\bar F}$
are the inclusions of the irreducible components of $Z_{\bar F}$.
These sheaves are endowed with an action of $\Gamma_F$
for which the above exact sequence is equivariant.
Taking {\'e}tale cohomology gives an exact sequence of $\Gamma_F$-modules:
\[ 0 \ra \mathrm H^0(X_{\bar F},\gm) \ra \mathrm H^0(U_{\bar F},\gm)
 \ra \bigoplus \mathrm H^0(Z_\alpha,\Z) 
        \ra \mathrm H^1(X_{\bar F},\gm) \ra \mathrm H^1(U_{\bar F},\gm) . \]
Moreover, this last map is surjective: it can be identified with 
the restriction map $\Pic(X_{\bar F})\ra \Pic(U_{\bar F})$
which is surjective because $X_{\bar F}$ is smooth and $U_{\bar F}$
open in~$X_{\bar F}$.
The scheme $X_{\bar F}$ is proper, smooth and connected
hence $\mathrm H^0(X_{\bar F},\gm)={\bar F}^*$.
Moreover for any~$\alpha$, $Z_\alpha$ is connected which implies
$\mathrm H^0(Z_\alpha,\Z)=\Z$, that is $\bigoplus \mathrm H^0(Z_\alpha,\Z)=\Pi(Z_{\bar
F})$.
Finally, one obtains an exact sequence of $\Z[\Gamma_F]$-modules
\[ 0 \ra \mathrm H^0(U_{\bar F},\gm)/\bar F^* \ra \Pi(Z_{\bar F})
 \ra  \Pic(X_{\bar F})\ra \Pic(U_{\bar F})\ra 0 , \]
which after tensoring with $\Q_\ell$ gives an equality
of virtual representations
\begin{equation}\label{eq.virtual-eq}
 \EP (U) = \Pi(Z_{\bar F})_{\Q_\ell}- \Pic(X_{\bar F})_{\Q_\ell} = \Pi(Z_{\bar F})_{\Q_\ell}- \mathrm H^2_{\text{\'et}}(X_{\bar F},\Q_\ell(1)). 
\end{equation}
This implies that
\[ q_v^{-\dim X} \Card\mathscr U(k_v) 
= 1 - \dfrac1{q_v} \tr(\Fr_v|\EP (U)) + \mathrm O(q_v^{-3/2}). \]
Since the eigenvalues of~$\Fr_v$ on the two Galois modules
defining $\EP (U)$ are algebraic numbers
whose absolute values are bounded by $1$,
one has
\[ \det(1-q_v^{-1}\Fr_v|\EP (U))
 = 1-\dfrac1{q_v} \tr(\Fr_v|\EP (U)) + \mathrm O(q_v^{-2}). \]
Now,
\[ \mathrm L_v(1,\EP (U))q_v^{-\dim X}\Card\mathscr U(k_v) 
=\det(1-q_v^{-1}\Fr_v|\EP (U))^{-1} q_v^{-\dim X}\Card\mathscr U(k_v)
 = 1+\mathrm O(q_v^{-3/2}), \]
and the asserted absolute convergence follows.
\end{proof}

\begin{defi}\label{defi.tamagawa-measure}
Let $F$ be a number field;
let $X$ be a smooth proper, geometrically integral variety over~$F$
such that $\mathrm H^1(X,\mathscr O_X)=\mathrm H^2(X,\mathscr O_X)=0$.
Let  $Z$ be a Zariski closed
subset in~$X$, and let $U=X\setminus Z$.
The Tamagawa measure on the adelic space $U(\AD_F^S)$
is defined as the measure
\[ \tau_U^S = \mathrm L^{S}_{*}(1,\EP (U))^{-1}
  \left(\prod_{v\not\in S} \mathrm L_v(1,\EP (U))\, \tau_{U,v} \right)
\]
\end{defi}

By Proposition~\ref{prop.convergence}, this infinite
product of measures converges;
moreover, non-empty, open and relatively compact,
subsets of $U(\AD_F^S)$ have a positive (finite) $\tau_U^S$-measure.
In particular,
compact subsets of~$U(\AD_F^S)$ have a finite~$\tau_U^S$-measure.

\begin{rems}
In the literature, 
families $(\lambda_v)$ of positive real numbers such that
the product of measures $\prod_v(\lambda_v \tau_{U,v})$
converges absolutely are called sets of convergence factors
(see, \emph{e.g.}, \cite{weil82} or~\cite{salberger98}). 
Our theorem can thus be stated as saying that for any smooth
geometrically integral algebraic variety~$U$ having
a smooth compactification~$X$ satisfying $\mathrm H^1(X,\mathscr O_X)=\mathrm H^2(X,\mathscr O_X)=0$, the family $(\mathrm L_v(1,\EP (U)))$
is a set of convergence factors. Let us compare our construction
with the previously  known cases.

\emph a)
The Picard group of an affine connected algebraic group~$U$
is finite; in that case, $\EP (U)$ is therefore
the Galois module of characters of~$U$ and we recover
Weil's definition (\cite{weil82}).
For semi-simple groups, there are no non-trivial characters
and the product of the natural local densities converges
absolutely, as is well-known.
This is also the case for homogeneous  varieties~$G/H$
studied by  Borovoi and Rudnick in~\cite{borovoi-r1995},
where $G\supset H$ are semi-simple algebraic groups
without non-trivial rational characters.

\emph b)
On the opposite side, integral projective varieties have
no non-constant global functions, so that $\EP (U)=\Pic(U_{\bar F})$
if $U$ is projective. We thus recover Peyre's definition
of the Tamagawa measure of a projective variety~\cite{peyre95}.

\emph c)
Salberger~\cite{salberger98} has developed a theory of Tamagawa 
measures on ``universal torsors''
introduced by Colliot-Thélène and Sansuc~\cite{colliot-thelene-sansuc1987}.
A universal torsor is a principal homogeneous space~$E$
over an algebraic variety~$V$ whose structure group is an algebraic
torus~$T$ dual to the Picard group of~$V$ such that the canonical 
map $\Hom(T,\mathbf G_m)\ra \mathrm H^1(V,\mathbf G_m)$,
sending a character~$\chi$ to the $\mathbf G_m$-torsor
$\chi_*[E]$ deduced from~$E$ by push-out along~$\chi$,
is an isomorphism.
Such torsors have been successfully applied to the study of rational
points  on the variety~$V$ (Hasse principle, weak approximation,
counting of rational points of bounded height);
see~\cite{colliot-thelene-sansuc1987,peyre98,breteche98b,skorobogatov2001}.

By a fundamental theorem of Colliot-Thélène and Sansuc\footnote
  {Apply~\cite{colliot-thelene-sansuc1987}, Prop.~2.1.1
   with $K=\bar k$, the map $\hat S\ra\Pic(X_{\bar k}))$
   being an isomorphism by the very definition of a universal torsor.},
universal  torsors have no non-constant invertible global functions 
and their Picard group is trivial (at least up to torsion).
It follows that the virtual $\Gamma_F$-module $\EP (E)$
is trivial.
This gives a conceptual explanation for the discovery by Salberger
in~\cite{salberger98}
that the Tamagawa measures on universal torsors could be defined
by an absolutely convergent product of ``naive'' local measures,
without any regularizing factors.
\end{rems}

\subsubsection{}
Assume moreover that~$Z$ is the support
of a divisor~$D$ in~$X$ (more generally, of a Cartier $\Q$-divisor)
and that $\omega_X(D)$ is endowed with an adelic metric;
this induces a natural metric on~$\mathcal O_X(D)$.
For any place~$v$, the metric on~$\omega_X(D)$ on $X(F_v)$
gives rise to  a measure~$\tau_{(X,D),v}$ on~$U(F_v)$,
which is related to the measure~$\tau_{X,v}$ by the formula
\[ \mathrm d\tau_{(X,D),v} (x) = \frac1{\norm{\mathsec f_D}(x)}
\,\mathrm d\tau_{X,v} (x), \]
where $\mathsec f_D$ is the canonical section of~$\mathscr O_X(D)$.
By definition of an adelic metric, the metric on~$\mathscr O_X(D)$
is induced, for almost all finite places~$v$, by the line
bundle~$\mathscr O_{\mathscr X}(\mathscr D)$ on the integral
model~$\mathscr X$, where $\mathscr D$ is a Cartier divisor
with generic fibre~$D$. 
For such places~$v$, 
one has $\norm{\mathsec f_D}(x)=1$ for any point~$x\in \mathscr U(\mathfrak o_v)$, so that 
the measures $\tau_{(X,D),v}$ and $\tau_{X,v}$ coincide
on $\mathscr U(\mathfrak o_v)$.

This shows that the product
\[ \tau_{(X,D)}^S = \mathrm L^{S}_{*}(1,\EP (U))^{-1}
  \left(\prod_{v\not\in S} \mathrm L_v(1,\EP (U))\, \tau_{(X,D),v} \right)
\]
also defines a Radon measure on $U(\AD_F^S)$.
In fact, one has
\[ \mathrm d\tau_{(X,D)}^S (x) = \big( \prod_{v} \norm{\mathsec f_D}(x_v)
\big)^{-1} \, \mathrm d\tau_U^S(x) = H_D(x) \,\mathrm d\tau_U^S(x) ,\]
and $H_D(x)$ is the height of~$x$ relative to the metrized
line bundle $\mathscr O_X(D)$.

\subsubsection{Example: compactifications of algebraic groups}
As in the case of local fields, we briefly explain
the case of algebraic groups and show how our theory of measures 
interacts with the construction  of a Haar measure on an adelic group.
We keep the notation of the previous paragraph,
and assume moreover that $U$ is an algebraic group~$G$ over~$F$
of which $X$ is an equivariant compactification.
Let us fix a invariant differential form of maximal degree~$\omega$
on~$G$; viewed as a meromorphic global section of~$\omega_X$,
it has poles on each component of~$X\setminus G$ 
(see~\cite{chambert-loir-t2002}, Lemma~2.4);
we therefore write $-D$ for its divisor.

For any place~$v\in\Val(F)$, viewing $\omega$ as a gauge
form on~$G(F_v)$ defines a Haar measure~$\abs\omega_v$
on $G(F_v)$.
For almost all finite places~$v$, $G(\mathfrak o_v)$ is a well-defined
compact subgroup of~$G(F_v)$. 
We normalize the Haar measure $\mathrm dg_v$ on $G(F_v)$ be dividing
the measure~$\abs{\omega}_v$ by the quantity $ \Card G(k_v)q_v^{-\dim G}$ for these places~$v$, and by~$1$ at other places. 
By construction, for almost all~$v$, 
$\mathrm dg_v$ assigns volume~1 to the compact
open subgroup $G(\mathfrak o_v)$, hence the product
$\prod_v \mathrm dg_v$ is a well-defined Haar measure on~$G(\AD_F)$.

By Proposition~\ref{prop.convergence}
we have the estimate
\[ \mathrm L_v(1,\EP (G)) \Card  G(k_v) q_v^{-\dim G} = 1+\mathrm O(q_v^{-3/2}), \]
from which we deduce that their infinite product
converges absolutely.
Consequently, $\prod_{v} \mathrm L_v(1,\EP (G))\abs{\omega}_v $
is also a Haar measure on $G(\AD_F)$.

Now, by the very definition of the adelic measures
$\tau_{(X,D)}$ and~$\tau_X$ on~$G(\AD_F)$, one has
\begin{align*}
 \mathrm d\tau_{(X,D)}(g)  &= H_D(g)\,\mathrm d \tau_X(g) \\
& = \prod_v \norm{\omega(g)}_v \quad\times\quad
\mathrm L_*(1,\EP (G))^{-1}  \prod_v \left(
\mathrm L_v(1,\EP (G)) 
   \frac{\mathrm d\abs{\omega}_v(g)}{\norm{\omega(g)}_v} \right)
\\
& = \mathrm L_*(1,\EP (G))^{-1} \prod_v \big(\mathrm L_v(1,\EP (G))
 \,\mathrm d\abs\omega_v(g) ). 
\end{align*}
This shows that,
in the particular case of an equivariant compactification
of an algebraic group, our general definition of the measure~$\tau_{(X,D)}$
on $X\setminus\abs D$ gives rise to a Haar measure on $G(\AD_F)$.

\subsection{An abstract equidistribution theorem}

In some cases, it is possible to count integral (resp.~rational) points
of bounded height 
with respect to almost any normalization of the height,
\emph{i.e.}, with respect to any metrization on a given line bundle.
Analogously, we obtain below an asymptotic expansion 
for the volume of height balls for almost any normalization of
the height.

In this subsection, we show how to extract from the obtained asymptotic
behavior
a measure-theoretic information 
on the distribution of points of bounded height
or of height balls. 

The following Proposition is  an abstract general version
of a result going back to Peyre (\cite{peyre95}, Proposition~3.3).

\begin{prop}\label{prop.abstract.equi}
Let $X$ be a compact topological space, $U$ a subset of $X$
and $H\colon U\ra\R_+$ a function.  Let $\nu$ be a positive measure
on~$X$ such that for any real number~$B$,
the set $\{x\in U\,;\, H(x)\leq B\}$ has finite~$\nu$-measure.

Let $\mathrm S$ be a dense subspace of the space $\mathrm C(X)$
of continuous functions on~$X$, endowed with the sup. norm.
Assume that there exist a function $\alpha\colon \R_+\ra\R_+^*$ and 
a Radon measure $\mu$ on $X$  such that
for any positive function $\theta\in \mathrm S$, 
\[ \nu \big(  \{ x\in U\,;\, H(x)\leq \theta(x) B \}\big) \sim \alpha(B) \int_X
\theta(x)\,\mathrm d\mu(x) \]
for $B\ra+\infty$.
Then, for $B\ra +\infty$, the measures
\[ \nu_B = \frac{1}{\alpha(B)} \mathbf 1_{\{ H(x)\leq B\}} \,\mathrm d\nu(x)
         \]
on $X$ converge vaguely to the measure $\mu$. In other words:
\begin{enumerate}\def\theenumi{\roman{enumi}}
\item 
for any continuous function $f\in\mathrm C(X)$,
\[ \frac1{\alpha(B)} \int_X \mathbf 1_{\{ H(x)\leq B\}} f(x)
\,\mathrm d\nu(x)
 \ra \int_X f(x)\,\mathrm d\mu(x), \qquad
\text{for $B\ra+\infty$;}\]
\item
for every open set $\Omega\subset X$ which is $\mu$-regular,
\[ \nu  \big(\{x\in \Omega\cap U\,;\, H(x)\leq B \}\big) = \alpha(B) \mu(\Omega)
 + \mathrm o(\alpha(B)). 
\]
\end{enumerate}
\end{prop}

Before entering the proof, let us introduce
one more notation.
For $\theta\in\mathrm C(X)$, $\theta>0$, let
\[ N(\theta;B) = \nu \big(  \{x\in U\,;\, H(x)\leq \theta(x) B \}\big). \]
For any open subset $\Omega\subset X$, let 
\[ N_\Omega(B) = \nu \big(  \{x\in \Omega\cap U\,;\, H(x)\leq B \}\big). 
\]
\begin{proof}
First remark that these hypotheses imply that for any $\lambda >0$,
$\alpha(\lambda B)/\alpha(B) \ra \lambda$ when $B\ra+\infty$.
Indeed, taking any positive~$\theta\in S$, 
\[ \frac{\alpha(\lambda B)}{\alpha(B)}
 = \frac{\alpha(\lambda B) \int_X \theta(x)\,\mathrm d\mu(x)}{\alpha(B) \int_X \theta(x)\,\mathrm d\mu(x)}
 \sim \frac{ N(\theta; \lambda B)}{N(\theta;B)}
 \sim \frac{ N(\lambda\theta; B)}{N(\theta;B)}
\sim \frac{\alpha(B) \int_X \lambda\theta(x)\,\mathrm d\mu(x)}{\alpha(B) \int_X \theta(x)\,\mathrm d\mu(x)} = \lambda. \]

Now, it is of course sufficient to prove that $N_\Omega(B)/
\alpha(B)\ra \mu(\Omega)$  for any $\mu$-regular open set $\Omega\subset
B$.
Fix $\eps>0$. Since $\Omega\subset X$
is $\mu$-regular, there exist continuous functions $f$ and $g$ on $X$
such that
\[  f \leq \mathbf 1_\Omega \leq g  \]
and such that $\int(g-f)\mathrm d\mu\leq \eps$.
Since $\mathrm S$ is dense in $\mathrm C(X)$, we may also assume
that $f$ and $g$ belong to~$\mathrm S$ and that they are 
non-negative.

Let $f_\eps=\eps+(1-\eps)f$, $\chi_\eps=\eps+(1-\eps)\mathbf 1_\Omega$
and $g_\eps=\eps+(1-\eps)g$. 
The inequality 
$ f_\eps\leq \chi_\eps\leq g_\eps $ implies that
\[ N(f_\eps;B)\leq N(\chi_\eps;B) \leq N(g_\eps;B). \]
Moreover, by definition, 
\[ 0\leq N(\chi_\eps;B)  -N_\Omega(B) \leq N_X(\eps B). \]
Therefore,
\begin{align*}
 \liminf N_\Omega(B)\alpha(B)^{-1}
 &  \geq \liminf N(\chi_\eps;B)\alpha(B)^{-1} 
            - \limsup  N_X(\eps B)\alpha(B)^{-1} \\
& \geq \liminf N(f_\eps;B) \alpha(B)^{-1} - \eps  \\
& \geq \int f_\eps \,\mathrm d\mu - \eps = \int f\,\mathrm d\mu + \mathrm
O(\eps). 
\end{align*}
Similarly,
\[ \limsup N_\Omega(B)\alpha(B)^{-1}
 \leq \int g \,\mathrm d\mu + \mathrm O(\eps). \]
When $\eps\ra 0$, one thus finds
\[ \lim N_\Omega(B)\alpha(B)^{-1} = \mu(\Omega), \]
as claimed.
\end{proof}

\clearpage
\def\lpar{\mathopen{(\!(}} \def\rpar{\mathclose{)\!)}}
\def\lbra{\mathopen{[\![}} \def\rbra{\mathclose{]\!]}}

\section{Geometric Igusa integrals: Preliminaries}
\label{s.igusa}

\subsection{Clemens complexes and variants}
\label{subsec.clemens}

In this section, we study analytic properties
of certain integrals of Igusa-type attached
to a divisor~$D$ in a variety~$X$.
This requires  the introduction of a simplicial
set which encodes the intersections of the various components of~$D$,
and which we call the \emph{Clemens complex}.
Such a simplicial set has been used by Clemens in~\cite{clemens1969}
in his study of the Picard-Lefschetz transformation
(see also~\cite{gordon1980}).

\subsubsection{Simplicial sets} \label{sss.simplicial-sets}
We recall classical definitions concerning simplicial sets.
A partially ordered set, abridged as \emph{poset},
is a set endowed with a binary relation~$\prec$ which is transitive
($a\prec b$ and $b\prec c$ implies $a\prec c$), antisymmetric
($a\prec b$ and $b\prec a$ implies $a=b$) and reflexive ($a\prec a$).
If $a\prec b$, we  say that $a$ is a face of~$b$;
in particular, the transitivity axiom for a poset means that a face of a face
is a face.

There are obvious definitions for morphisms of posets, sub-posets,
as well as actions of groups on posets.

As an example, if $\mathscr A$ is a set, the set $\mathscr P^*(\mathscr A)$
of non-empty subsets of~$\mathscr A$, together with the inclusion
relation, is a poset. We shall interpret it geometrically as
the simplex on the set of vertices~$\mathscr A$
and denote it by $\mathscr S_{\mathscr A}$.
More generally, a \emph{simplicial complex} on a set~$\mathscr A$ 
is a set of non-empty finite subsets of~$\mathscr A$, called \emph{simplices},
such that any non-empty subset of a simplex is itself a simplex.
The dimension of a simplex is defined as one less its cardinality:
points, edges,... are simplexes of dimension~$0$, $1$,...
The dimension of a simplicial complex is the supremum
of the dimensions of its faces. 
The order relation endows a simplicial complex with the structure of poset.
In fact, simplicial complexes are
sub-posets of the simplex~$\mathscr S_{\mathscr A}$.

Let $S$ be a poset. The dimension of an element~$s\in S$,
denoted $\dim(s)$,
is defined as the supremum of the lengths~$n$ of chains
$s_0\precneq\cdots\precneq s_n$, with $s_n\prec s$.
Similarly, the codimension~$\codim(s)$ of an element~$s\in S$
is defined as the supremum of the lengths of such chains 
with $s\prec s_0$.
Elements of dimension~$0$, $1$,\ldots, are called vertices,
edges,... The dimension~$\dim(S)$ of~$S$ is the supremum of all dimensions
of all of its elements.

It is important to observe that given a poset~$S$
and a sub-poset~$S'$, the dimension or the codimension of a face of~$S'$
may differ from their dimension or codimension as a face of~$S$.\footnote{This probably means that the terminology ``sub-poset'' is inappropriate;
sub-posets for which the dimension notion is compatible
are sometimes called ideals...}

Let $S$ be a poset. Categorically, an action of a group~$\Gamma$ on~$S$
is just a morphism of groups from~$\Gamma$ to the set $\Aut(S)$
of automorphisms of~$S$. In other words, it is the data, for any $\gamma\in\Gamma$, of a bijection $\gamma_*$ of~$S$ such that $\gamma_*s\prec \gamma_*s'$
is $s\prec s'$, subject to the compatibilities $(\gamma\gamma')_*=\gamma_*\gamma'_*$ for any~$\gamma$ and~$\gamma'\in\Gamma$.

 
Let there be given a poset~$S$ and an action of~$\Gamma$ on~$S$.
The set $S^\Gamma$ of fixed points of~$\Gamma$ in~$S$ is
a sub-poset of~$S$.
In the other direction, the set~$S/\Gamma$ of orbits of~$\Gamma$
in~$S$, can be endowed with the binary relation deduced
from~$\prec$ by passing to the quotient. Namely, for $s$ and~$s'$
in~$S$, with orbits $[s]$ and~$[s']$, we say that
$[s]\prec [s']$ if there exists $\gamma\in \Gamma$ such
that $s\prec \gamma_*s'$. (This condition does not depend 
on the actual choice of elements~$s$ and~$s'$ in their
orbits~$[s]$ and~$[s']$.)

There are obvious morphisms of posets, $S^\Gamma\ra S$ 
and $S\ra S_\Gamma$; these morphisms are the universal
morphisms of posets, respectively to~$S$ and from~$S$,
which commute with the action of~$\Gamma$ on~$S$.

%
 
As an example,
assume that $\mathscr S=\mathscr S_{\mathscr A}$, the simplex with
vertices in a set~$\mathscr A$. Actions of a group~$\Gamma$
on~$\mathscr S_{\mathscr A}$ correspond to actions of~$\Gamma$
on~$\mathscr A$. The simplicial sets $\mathscr S^\Gamma$
and~$\mathscr S_\Gamma$ are respectively the simplices with vertices
in the fixed-points~$\mathscr A^\Gamma$ and the orbits~$\mathscr A/\Gamma$.

\def\bmA{{\bar{\mathscr A}}}

\subsubsection{Incidence complexes}

Let $X$ be a (geometrically integral) variety 
over a perfect field~$F$ and let $D$ be a divisor in~$X$. 
Fix a separable closure~$\bar F$ of~$F$.
Let~$\bmA$ be the set of irreducible components of~$D_{\bar F}$;
for $\alpha\in\bmA$, we denote by $D_\alpha$ the corresponding component
of~$D_{\bar F}$. 
For any subset $A\subset \bmA$, we let $D_A=\bigcap_{\alpha\in A}D_\alpha$;
in  particular, $D_{\emptyset}=X_{\bar F}$.

We shall always make the assumption
that the divisor~$D_{\bar F}$ has simple normal crossings:
all irreducible components $D_\alpha$ of~$D_{\bar F}$ are supposed to be smooth
and to meet transversally.\footnote
  {This is equivalent to
  the weaker condition that $D$ has normal crossings together with
  the smoothness of  the geometric irreducible components of~$D$;
  the condition that~$D$ has strict normal crossings is however
  stronger since the smoothness of an irreducible component of~$D$
  implies that its geometric irreducible components don't meet.}
In particular, for any subset $A\subset\bmA$ such that $D_A\neq\emptyset$,
$D_A$ is a smooth subvariety of~$X_{\bar F}$ of codimension~$\Card A$. 

The closed subschemes~$D_A$,
for $A\subset\bmA$, are the closed strata of a stratification
$(D_A^\circ)_{A\subset\bmA}$, where, for any subset~$A\subset\bmA$,
$D_A^\circ$ is defined by the formula:
\[ D_A^\circ = D_A \setminus \bigcup _{B\supsetneq A} D_B. \]

There are in fact several natural posets that enter the picture,
encoding in various ways  the combinatorial data
of whether or not, for a given  subset~$A$ of~$\mathscr A$,
the intersection~$D_A$ is empty.
The following observation is crucial for our definitions to make sense:
\begin{prop}
Let $A$ and~$A'$ be two subsets of~$\bmA$ such that $A'\subset A$.
For any irreducible component~$Z$ of~$D_A$, there is a unique
irreducible component~$Z'$ of~$D_{A'}$ which contains~$Z$.
\end{prop}
\begin{proof}
If an irreducible component~$Z$ were contained in two
distinct irreducible components of~$D_{A'}$,
these two components would meet along~$Z$, which contradicts
the simple normal crossings assumption on~$D_{\bar F}$.
Indeed, $D_A$ being smooth, its irreducible components must be  disjoint.
\end{proof}

\subsubsection{The geometric Clemens complex}
The incidence complex~$\mathscr I_{\bar F}(D)$ 
defined by~$D$ is the sub-poset of~$\mathscr P^*(\bmA)$
consisting of non-empty subsets~$A$ of~$\bmA$ such that the intersection
$D_A$ is not empty.
More precisely, we define 
the geometric Clemens complex~$\Cl_{\bar F}(D)$ as the set
of all pairs $(A,Z)$, where $A\subset\bmA$ is any non-empty subset, and
$Z$ is an irreducible component of the scheme~$D_{A}$,
together with the partial order relation 
defined by $(A,Z)\prec (A',Z')$ if $A\subset A'$ and $Z\supset Z'$.
In other words, the set of vertices of~$\Cl_{\bar F}(D)$ is~$\bmA$,
there are edges corresponding to irreducible
component of each intersection $D_\alpha\cap D_{\alpha'}$, etc.

Mapping $(A,Z)$ to~$A$ induces a morphism of posets from
the geometric Clemens complex to the incidence complex.
In fact, we could have defined $\Cl_{\bar F}(D)$ without
any reference to~$\bmA$. Indeed,
thanks to the normal crossings condition, given an irreducible
component~$Z$ of a scheme~$D_A$, we may recover~$A$ as the set
of $\alpha\in \bmA$ such that $D_\alpha$ contains the generic point of~$Z$.

\medskip

Since $X$ is defined over~$F$, $\Gamma_F$ acts naturally 
on the set of integral subschemes of~$X_{\bar F}$.
Since $F$ is perfect, an integral subscheme~$Z$
of~$X_{\bar F}$ is defined over~$F$ if and only if
it is a fixed point for this action.

For any $\gamma\in\Gamma_F$ and any $\alpha\in\bmA$, observe
that $\gamma_* D_{\alpha}$ is an irreducible component of~$D_{\bar F}$,
because $D$ is defined over~$F$;
this induces an action of~$\Gamma_F$ on~$\bmA$, defined
by $\gamma_* D_\alpha=D_{\gamma\alpha}$, 
for $\alpha\in\bmA$ and $\gamma\in\Gamma_F$.
Moreover, if $Z$ is an irreducible component of
an intersection $D_A$, for $A\subset\bmA$, then
$\gamma_* Z$ is an irreducible component
of $\gamma_* D_A=D_{\gamma_* A}$.
Consequently, we have a natural action of~$\Gamma_F$ on
the geometric Clemens complex~$\Cl_{\bar F}(D)$,
given by $\gamma_*(A,Z)=(\gamma_* A,\gamma_* Z)$ for any  element~$(A,Z)$
of~$\Cl_{\bar F}(D)$.

The natural morphism of posets
$\Cl_{\bar F}(D)\ra \mathscr I_{\bar F}(D)$ mapping $(A,Z)$ to~$A$
is $\Gamma_F$-equivariant.

\subsubsection{Rational Clemens complexes}

Let us denote by $\mathscr I_F(D)$ and~$\Cl_F(D)$
the sub-posets of~$\mathscr I_{\bar F}(D)$ and $\Cl_{\bar F}(D)$
consisting of $\Gamma_F$-fixed faces.
These posets correspond to the intersections of the
divisors~$D_\alpha$ 
and to the irreducible components of these intersections
which are defined over the base field~$F$.
Alternatively, they can be defined without
any reference to the algebraic closure~$\bar F$ by considering
irreducible components of the divisor~$D$, their intersections
and the irreducible components of these which are geometrically
irreducible.

More generally, let $E$ be any extension of~$F$,
together with an embedding of~$\bar F$ in an algebraic closure~$\bar E$.
As an example, one may take any extension of~$F$ contained in~$\bar F$.
In our study below, $F$ will be a number field and $E$ will be the completion
of~$F$ at a place~$v$; the choice of an embedding $\bar F\hra\bar E$
corresponds to the choice of a  decomposition group at~$v$.
Under these conditions,  there is a natural  morphism
of groups from the Galois group~$\Gamma_E$ of~$\bar E/E$
to~$\Gamma_F$. In particular,
the posets~$\Cl_{\bar F}(D)$ and~$\mathscr I_{\bar F}(D)$ 
are endowed with an action of~$\Gamma_E$.

Let us define the $E$-rational Clemens complex, $\Cl_{E}(D)$, as
the sub-poset of~$\Cl_{\bar F}(D)$ fixed by~$\Gamma_E$.
In particular, 
for any face $(A,Z)$ of~$\Cl_E(D)$, the subschemes~$(D_A)_{\overline E}$
and $Z_{\overline E}$ are defined over~$E$. We shall
denote $(D_A)_E$ and~$Z_E$, or even $D_A$ and~$Z$,
the corresponding subschemes of~$X_E$. 
Observe that $Z_E$ is geometrically irreducible.

Conversely, let $A$ be any non-empty subset of~$\bar{\mathscr A}$
which is $\Gamma_E$-invariant; then, $(D_A)_{\overline E}$ is defined 
over~$E$ and corresponds to some subscheme $(D_A)_E$ of~$X_E$.
Let $Z$ be any irreducible component of~$(D_A)_E$
which is geometrically irreducible. Then, $Z_{\bar E}$
is an irreducible component of~$(D_A)_{\bar E}$.
By EGA IV (4.5.1), the set of irreducible components doesn't change
when one extends the ground field from an algebraically
closed field to any extension;
consequently, $Z_{\bar E}$ is defined over~$\bar F$
and $(A,Z)$ corresponds to a face of~$\Cl_E(D)$.

\subsubsection{$E$-analytic Clemens complexes}

Let $E$ be any perfect extension of~$F$ together with an embedding
$\bar F\hra\bar E$ as above.
We define the \emph{$E$-analytic Clemens complex},
denoted $\Clan_E(D)$, as
the sub-poset of~$\Cl_E(D)$ whose faces are those faces
$(A,Z)$ such that $Z(E)\neq\emptyset$.
(When we write $Z(E)$, we identify~$Z$  with the unique
$E$-subscheme of~$X_E$ whose extension to~$\bar E$ is~$Z_{\bar E}$.)

\smallskip

This complex can also be defined as follows.
Let first $A$ be any non-empty subset of~$\bar{\mathscr A}$ 
which is $\Gamma_E$-invariant and let $Z$ be an irreducible
component of~$D_A$. While $(D_A)_{\bar E}$ is defined over~$E$,
for the moment, $Z$ is just a subscheme of~$X_{\bar F}$.
However, we can define $Z(E)$ as the intersection in $X(\bar E)$
of $X(E)$ and of $Z(\bar E)$; this is indeed what
one would get if $Z_{\bar E}$ where defined over~$E$.

Assume that $Z(E)\neq\emptyset$.
Let $Z'$ be the smallest subscheme of~$X_E$ such that $Z'_{\bar E}$ 
contains~$Z$. It is irreducible: if $Z'$ were
a union $Z'_1\cup Z'_2$,
one of them, say~$Z'_1$, would be such that 
$(Z'_1)_{\bar E}$ contains $Z_{\bar E}$,
because $Z_{\bar E}$ is irreducible. Moreover, $Z'$ is contained
in~$(D_A)_E$, since $(D_A)_{\bar E}$ contains~$Z_{\bar E}$.
It follows that $Z'$ is an irreducible component of~$(D_A)_E$.
By assumption, $Z'(E)\neq\emptyset$ hence,
by EGA IV (4.5.17), $Z'$ is geometrically connected.
Since $D_A$ is smooth,  so are $(D_A)_{\bar E}$ and $Z'_{\bar E}$,
this last subscheme being a union of irreducible components
of~$(D_A)_{\bar E}$. But a connected smooth scheme is irreducible,
hence $Z'$ is irreducible and $Z'_{\bar E}=Z_{\bar E}$.
It follows that $Z_{\bar E}$ is defined over~$E$
and $(A,Z)$ corresponds to a face of $\Clan_E(D)$.

\smallskip

Let us finally observe that the dimension of a face~$(A,Z)$
of $\Clan_E(D)$ is equal to $\Card(A/\Gamma_E)-1$. 
Indeed, let $n=\Card(A/\Gamma_E)$ and consider
any sequence $A_1\subset\cdots\subset A_n=A$ of $\Gamma_E$-invariants
subsets of~$A$. Then, $D_{A_1}\supset D_{A_2}\supset\cdots\supset D_{A_n}$
and each of them has a unique irreducible component~$Z_i$ containing~$Z$.
Since $Z(E)\neq\emptyset$, $Z_i(E)$ is non-empty either,
and the $(A_i,Z_i)$ define a maximal increasing sequence
of faces of~$\Clan_E(D)$.

\smallskip

This notion will be of interest to us under the supplementary
assumption that $E$ is a locally compact valued field.
Indeed, such fields allow for a theory
of analytic manifolds as well as an implicit function theorem.
By the local description of~$X$ that
will be explained below, for any face $(A,Z)$ of $\Cl_E(D)$,
$Z(E)$ is either empty, or is an $E$-analytic submanifold of~$X(E)$
of codimension~$\Card(A)$ which is Zariski-dense in~$Z$.
However, as a face of~$\Clan_E(D)$, $(A,Z)$ has dimension~$\Card(A/\Gamma_E)-1$
and it is that invariant which will be relevant in our analysis below.


\subsubsection{Example}
Below we describe these Clemens complexes in special
cases, which are essentially governed by combinatorial data,
namely for toric varieties and equivariant compactifications
of semi-simple groups. 
We want to make clear that for general varieties, 
the three types of Clemens complexes
can be very different.

For example, let $X$ be the blow-up of the
projective space~$\P^n$ along
a smooth subvariety~$V$ defined over the ground field,
let $D$ be the exceptional divisor in~$X$
and let $\pi\colon X\ra\P^n$ be the canonical morphism.
The map~$\pi$ induces
a bijection between the set of irreducible  components of~$D_{\bar F}$ 
and the set of irreducible components of~$V_{\bar F}$,
as well as a bijection between the  set of irreducible
components of~$D$ and the set of irreducible components of~$V$.
Moreover, a component~$Z$ of~$D$ is geometrically irreducible
if and only if the corresponding component~$\pi(Z)$ of~$V$ is;
indeed, $Z$ is isomorphic to the projectivized normal bundle
of~$V$ in~$X$.
This description shows also that $Z(F)=\emptyset$
if and only $\pi(Z)(F)=\emptyset$.

To give a specific example, take $V$ to be the disjoint union
of a geometrically irreducible smooth curve~$C_1$, without rational points, 
of a geometrically irreducible smooth curve~$C_2$ with rational points
and of a smooth irreducible  curve~$C_3$ with two geometrically components~$C'_3$ and~$C''_3$.
Then, $\Cl(X,D)$ consists of four points corresponding to~$C_1, C_2, C'_2, C'_3$
(there are no intersections), 
the $F$-rational Clemens complex $\Cl_F(X,D)$ consists
of two points corresponding to~$C_1$ and~$C_2$,
and the $F$-analytic Clemens complex $\Clan_F(X,D)$
consists of the single point corresponding to~$C_1$.
We present further examples in Section~\ref{sec.examples}.

\subsection{Local description of a pair~$(X,D)$}

\subsubsection{Notation}
Let~$F$ be a perfect field, $\bar F$ an algebraic closure of~$F$,
$X$ a smooth algebraic variety over~$F$ and $D\subset X$ a reduced divisor
such that $D_{\bar F}$ has strict normal crossings in~$X_{\bar F}$.
Let us recall that 
this means that $D_{\bar F}$ is the union of irreducible smooth divisors
in~$X_{\bar F}$ which meet transversally. In other words,
each point of~$X_{\bar F}$ has a neighborhood~$U$, together
with an étale map $U\ra\Aff^n_{\bar F}$ such that for each irreducible
component~$Z$ of~$D_{\bar F}\cap U$, $Z\cap U$ is the preimage of some
coordinate hyperplane in~$\Aff^n_{\bar F}$.

Let~$\bmA$ be the set of irreducible components of~$D_{\bar F}$;
for $\alpha\in \bmA$, let~$D_\alpha$ denote the corresponding divisor,
so that $D_{\bar F}=\sum_{\alpha\in \mathscr A}D_\alpha$.
For $\alpha\in \bmA$, we denote by~$F_\alpha$ the field of definition
of~$D_\alpha$ in~$\bar F$; its Galois group~$\Gamma_{F_\alpha}$
is the stabilizer of~$D_\alpha$ in~$\Gamma_F$. 
For $A\subset\bmA$, we denote by~$D_A$ the intersection
of the divisors~$D_\alpha$, for $\alpha\in A$; by convention,
$D_\emptyset=X$.

Let $\alpha\in\bmA$.
The union~$\bigcup_{a\in\Gamma_F\alpha} D_a$
of the conjugates of~$D_\alpha$ is a divisor in~$X_{\bar F}$
which is defined over~$F$; the corresponding divisor of~$X$
will be denoted by~$\Delta_\alpha$.
Similarly, the intersection of the divisors~$D_a$,
for $a\in\Gamma_F\alpha$ is a smooth subscheme of~$X_{\bar F}$
which is defined over~$X$; the corresponding subscheme of~$X$
will be denoted by~$E_\alpha$. If $E_\alpha\neq\emptyset$,
then its codimension~$r_\alpha$ is equal to~$[F_\alpha:F]$.
Moreover, $\Delta_\alpha(F)=E_\alpha(F)$.

The choice of another element~$\alpha'=\gamma_*\alpha$ 
in the orbit~$\Gamma_F\alpha$
doesn't change~$\Delta_\alpha$ nor~$E_\alpha$. However, it changes
the field~$F_\alpha$ to the conjugate~$F_{\alpha'}=\gamma_*F_\alpha$
and its stabilizer~$\Gamma_{F_\alpha}$ to its conjugate subgroup
$\Gamma_{F_{\alpha'}}=\gamma\Gamma_{F_\alpha}\gamma^{-1}$.
Observe that $r_{\alpha'}=r_\alpha$.

\subsubsection{Local equations for the étale topology}
Let us now give local equations of~$D$ around each rational point of~$X$.
Fix a point~$\xi\in X(F)$ and let $\alpha\in\bmA$ be such that
$\xi\in D_\alpha$ (it is a $\Gamma_F$-invariant subset of~$\bmA$).

When $F_\alpha=F$, $D_\alpha=\Delta_\alpha\times\bar F$
is defined over~$F$ and the choice of a local equation for~$D_\alpha$
defines a smooth map from an open neighborhood~$U_\xi$ of~$\xi$ in~$X$
to the affine line which maps~$U_\xi\cap D_\alpha$ to~$\{0\}\subset\Aff^1$.

To explain the general case, let us introduce some more notation.
We write 
$\Aff_{F_\alpha}$  for the Weil restriction of scalars
$\Res_{F_\alpha/F}\Aff^1$
from~$F_\alpha$ to~$F$ of the affine line. It is endowed
with a canonical morphism $\Aff^1_{F_\alpha}\ra\left(\Res_{F_\alpha/F}\Aff^1\right)\times_F F_\alpha$ which induces, for any $F$-scheme~$V$,
a bijection
\[ \Hom(V,\Res_{F_\alpha/F}\Aff^1)=\Hom(V\times_F F_\alpha,\Aff^1). \]
In particular, the $F$-rational points of~$\Aff_{F_\alpha}$
are in canonical bijection with~$F_\alpha$.

The $F$-scheme~$\Aff_{F_\alpha}$
is an affine space of dimension~$r_\alpha=[F_\alpha:F]$.
Let us indeed choose a~$F$-linear basis $(u_1,\ldots,u_{r_\alpha})$
of~$F_\alpha$. Then, the morphism
$\Aff^{r_\alpha}\ra \Aff_{F_\alpha}$ given by 
$(x_1,\ldots,x_{r_\alpha})\mapsto \sum x_i u_i$
is an isomorphism.
In these coordinates, the norm-form~$\mathrm N_{F_\alpha/F}$ 
of~$F_\alpha$ 
is a homogeneous polynomial of degree~$[F_\alpha:F]$
which defines a hypersurface in~$\Aff_{F_\alpha}$; it has a single
$F$-rational point~$0$.

Moreover, the $r_\alpha$ $F$-linear embeddings of~$F_\alpha$ into~$\bar F$
induce an isomorphism $F_\alpha\otimes_F \bar F\simeq 
\bar F^{r_\alpha}$, 
hence an identification of~$\Aff_{F_\alpha}\times_F \bar F$
with the affine space~$\Aff^{r_\alpha}_{\bar F}$.
Under this identification, the divisor of the norm-form
corresponds to the union of the coordinates hyperplanes.

\begin{lemm}
There is a Zariski open neighborhood~$U_\xi$ of~$\xi$ in~$X$
and a smooth map
\[ 
x_{\alpha}\colon U_\xi\ra\Aff_{F_\alpha}
\]
which maps~$\Delta_\alpha\cap U_\xi$ to the hypersurface of~$\Aff_{F_\alpha}$
defined by the norm equation $\mathrm N_{F_\alpha/F}(x)=0$,
such that there is a diagram
\[  \xymatrix{  E_\alpha \ar@{^(->}[r] \ar[d] & \Delta_{\alpha} \ar@{^(->}[r] \ar[d] & **[r] U_\xi \subset X \ar[d] \\
   {\{0\}} \ar@{^(->}[r] &            {\{\mathrm N_{F_\alpha/F}(x)=0\}} \ar @{^(->}[r] & \Aff_{F\alpha}.} \]
\end{lemm}
\begin{proof}
Choose an element~$f$ in the local ring~$\mathscr O_{X_{\bar F},\xi}$ 
which is a generator of the ideal of~$D_\alpha$, so that
$f=0$ is a local equation of~$D_\alpha$.
For any~$\gamma\in \Gamma_F$, $f^\gamma=0$
is a local equation of~$\gamma_* D_\alpha$.  Consequently,
$f^\gamma/f$ is a local unit at~$\xi$, for each
$\gamma\in\Gamma_{F_\alpha}$, and the map~$\gamma\mapsto f^\gamma/f$
is a $1$-cocycle for the Galois group~$\Gamma_{F_\alpha}$ with
values in the local ring~$\mathscr O_{X_{\bar F},\xi}^*$.  
By Hilbert's Theorem~90 for the multiplicative group,
this cocycle is a coboundary 
(see~\cite{milne80}, Proposition~4.9 and Lemma~4.10 for a non-elementary proof
in that context) and there exists a unit $u\in \mathscr O_{X_{\bar F},\xi}^*$
such that $f^\gamma/f=u^\gamma/u$ for any~$\gamma\in\Gamma_{F_\alpha}$.
Consequently, $fu^{-1}$ is an element of~$ \mathscr O_{X_{\bar F},\xi}$,
fixed by $\Gamma_{F_\alpha}$, which generates the ideal
of~$D_\alpha$.
In other words, 
there is a Zariski open neighborhood~$V$ of~$\xi$ in~$X$, 
and a function $g\in\Gamma(V\times_F {F_\alpha},\mathscr O_{V\times F_\alpha})$
such that~$g=0$ is an equation of~$D_\alpha$ in~$V\times_F {F_\alpha}$.
To the morphism $V\times_F F_\alpha \ra\Aff^1_{F_\alpha}$
corresponds by the universal property of the Weil restriction of scalars
a morphism, still denoted~$\theta$, $V\ra \Aff_{F_\alpha}$.

Through the canonical identification, over~$\bar F$, of~$\Aff_{F_\alpha}$ with
the affine space~$\Aff_{\bar F}^{r_\alpha}$, the map~$g$ induces an isomorphism
from~$D_\alpha$ to one of the coordinate hyperplanes,
from~$\Delta_\alpha$ to the union of the coordinate hyperplanes,
while $\xi$ maps to~$0$.
At the level of $F$-rational points, the diagram
$(V,\Delta_\alpha)\ra (\Aff_{F_\alpha},H)$ corresponds
to the map~$(V(F),D_\alpha(F))\ra (F_\alpha,0)$.
\end{proof}

 
\subsubsection{Local charts}
\label{subsubsec.local-charts}
Let $\xi\in X(F)$.
More generally, let $A_\xi$ be the set of~$\alpha\in\mathscr A$
such that $\xi\in D_\alpha$. Let $Z$ be the irreducible
component of~$D_{A_\xi}$ which contains~$\xi$. It is geometrically
irreducible and $(A_\xi,Z)$  is a face of~$\Clan_F(D)$.
For a sufficiently small neighborhood~$V$ of~$\xi$, there exists,
for each~$\alpha\in A_\xi$, 
a smooth map~$x_\alpha\colon V\ra\Aff_{F_\alpha}$ as in the previous
paragraph. 
Since the divisors~$D_\alpha$ meet transversally,
the map $(x_\alpha)\colon V\ra \prod_{\alpha\in A_\xi} \Aff_{F_\alpha}$
is smooth. 
By choosing additional local coordinates, and shrinking~$V$ if needed, 
we can extend it to
an \'etale map~$q_\xi\colon U_\xi\ra\prod_{\alpha\in A_\xi}\Aff_{F_\alpha}
\times \Aff^r$,  with $r=\codim_\xi( D_{A_\xi},X)=\Card(A_\xi)$.
We will also assume, as we may, that the open subset~$U_\xi$ is affine
and that $U_\xi\cap D_\alpha=\emptyset$ if $\alpha\not\in A_\xi$.

Assume moreover that $F$ is a local field.
It follows from the preceding discussion that
for each $\alpha\in\mathscr A$, 
$\Delta_\alpha(F)=E_\alpha(F)$ is a smooth $F$-analytic subvariety
of~$X(F)$, either empty, or of codimension~$r_\alpha=[F:F_\alpha]$. 
For any $ \Gamma_F$-invariant subset~$A$ of~$\mathscr A$, let us
denote by~$\Delta_A$ the intersection of all~$\Delta_\alpha$,
for $\alpha\in A$. Then, $\Delta_A(F)=E_A(F)$ is either empty, 
or a smooth $F$-analytic subvariety
of~$X(F)$ of codimension $r_A = \Card(A)$.

\subsubsection{Partitions of unity}
\label{subsubsec.partition-unity}
In this section, we assume in addition that $F$ is a local field
and we consider the situation from the analytic point of view.
Observe that
the maps~$q_\xi$ induce analytic \'etale maps 
$U_\xi(F)\ra \prod_{\alpha\in A_\xi} F_\alpha\times F^r$. 
The open sets $U_\xi(F)$ cover~$X(F)$ which is compact, since $X$ is projective.
Consequently, there is a finite partition of unity $(\theta_\xi)$ subordinate
to this covering: for any~$x$, $\theta_\xi\colon X(F)\ra\R$
is a smooth function whose support is contained in~$U_\xi(F)$,
the map $q_\xi$ is one-to-one on that support,
$\sum_{\xi\in X(F)} \theta_\xi=1$ and only a finite number of~$\theta_\xi$ 
are nonzero.

\subsubsection{Metrized line bundles, volume forms}
Let $\alpha\in\bmA$ and 
let us consider the line bundle~$\mathscr O_X(\Delta_\alpha)$.
It has a canonical section~$\mathsec f_\alpha$ the divisor of
which is~$\Delta_\alpha$. 
Let us assume that $F$ is a local field. On each open set
$U_\xi$ of $X$, we have constructed
a regular map $x_\alpha\colon U_\xi \ra\Aff_{F_\alpha}$.
Composed with the norm $\Aff_{F_\alpha}\ra\Aff^1$, it gives
us a regular function $\mathrm N(x_\alpha)\in \Gamma(U_\xi,\mathscr O_X)$
which generates the ideal of~$\Delta_\alpha$ on~$U_\xi$.
This function induces therefore a trivialization
of $\mathscr O_X(-\Delta_\alpha)_{|U_\xi}$.
Consequently, a metric on
$\mathscr O_X(\Delta_\alpha)_{|U_\xi}$
takes the form $\norm{\mathsec f_\alpha}(x)=\abs{x_\alpha}_{F_\alpha}
 h_\xi(x)$ on $U_\xi(F)$, where we view $x_\alpha$
as a local coordinate $U_\xi(F)\ra F_\alpha$ and 
$\abs{x_\alpha}_{F_\alpha}=\abs{\mathrm N_{F_\alpha/F}(x_\alpha)}_F$,
and where $h_\xi\colon U_\xi(F)\ra\R_+^*$ is any continuous function.

Similarly, on such a chart~$U_\xi(F)$, we have a 
measure in the Lebesgue class
defined by $\mathrm d\mathbf x=\prod_{\alpha\in A_\xi}\mathrm dx_\alpha\times
 \mathrm dx_1\cdots\mathrm dx_r$. Any measure on $X(F)$ which
belongs to the Lebesgue class can therefore be expressed
in a chart $U_\xi(F)$ as the product of this measure $\mathrm d\mathbf x$
with a positive locally integrable function.
Let us also observe that there are canonical isomorphisms
\[ \mathscr O(\Delta_\alpha)|_{E_\alpha}\simeq 
\mathscr N_{\Delta_\alpha}(X) |_{E_\alpha} \simeq 
\det \mathscr N_{E_\alpha}(X). \]
These isomorphisms allow us to define residue measures on 
$\Delta_\alpha(F)=E_\alpha(F)$, and on their intersections,
hence on all faces of the $F$-analytic Clemens complex~$\Clan_F(D)$.

\subsection{Mellin transformation over local fields}

\subsubsection{Local zeta functions}\label{subsubsec.local-zeta}
Let~$F$ be a local field of characteristic zero, with normalized
absolute value~$\abs{\cdot}_F$.
Define a constant $\mathrm c_F$ as follows:
\begin{equation}\label{eq.cF}
 \mathrm c_F = \begin{cases}
  \mu ( [-1;1])  & \text{if $F=\R$;}\\
  \mu (B(0;1)) & \text{if $F=\C$;} \\
  (1-q^{-1}) (\log q)^{-1} \mu (\mathfrak o_F)  & \text{if $F\supset \Q_p$,
           $q=\abs{\varpi_F}_F^{-1}$.}
\end{cases}
\end{equation}
(In the last formula, $\varpi_F$ denotes  
a generator of the maximal ideal of~$\mathfrak o_F$.)

For any locally integrable function~$\phi\colon F\ra\C$, we let
\[ \mathscr M_F(\phi)(s)=\int_F \phi(x)\abs{x}_F^{s-1}\, \mathrm d\mu(x), \]
for any complex parameter~$s$ such that the integral
converges absolutely. We call it the Mellin transform of~$\phi$;
to shorten the notation and when no confusion can arise,
we will often write $\hat\phi$ instead
of~$\mathscr M_F(\phi)$,

We write~$\zeta_F$ for the Mellin transform~$\mathscr M_F(\phi_0)$ of
$\phi_0$, the characteristic function of the unit ball in~$F$.
The integral defining~$\zeta_F(s)$ converges if and only if~$\Res(s)>0$, and one has
\begin{align*}
\zeta_F(s) &= \mu([-1,1]) \int_0^1 x^{s-1}\, \mathrm dx = \frac{\mathrm c_F}{s} && \quad\text{for $F=\R$}\\
&=  \frac{\mu(B(0;1))}{\pi}  2\pi\int_0^1 r^{2s-1}\,\mathrm dr = \frac{\mathrm c_F}{s} && \quad\text{for $F=\C$} \\
& = \sum_{n=0}^\infty q^{-n(s-1)} \mu(\varpi_F^n\mathfrak o_F^*) 
 = \mu(\mathfrak o_F) \big(1-\frac1q\big) \sum_{n=0}^\infty q^{-ns}  
 = \mu(\mathfrak o_F) \frac{1-q^{-1}}{1-q^{-s}} && \quad\text{otherwise}.
\end{align*}
In all cases, the function~$s\mapsto \zeta_F(s)$ is holomorphic
on the half-plane defined by~$\Re(s)>0$ and extends to a meromorphic
function on~$\C$. It has a simple pole at~$s=0$, with residue~$\mathrm c_F$.
In the archimedean case, note that $s=0$ is the only pole of~$\zeta_F$
However, in the $p$-adic case, $\zeta_F$ is $(2i\pi/\log q)$-periodic
so that $s=2in\pi/\log q$ is also a pole for any integer~$n$; 
this will play a role in the use of Tauberian arguments later.

\subsubsection{Mellin transform of smooth functions}
One of our goals in that paper is to establish
the meromorphic continuation, together with growth estimates,
for integrals on varieties that are defined locally like  Mellin  transforms.
To explain our arguments, we begin with the one-dimensional toy example.

We say that a function $\phi\colon F\ra\C$ is smooth if it
is either $\mathrm C^\infty$ when $F=\R$ or~$\C$, or
locally constant otherwise.
If $a$ and~$b$ are real numbers, we define 
$\Tube_{>a}$, resp. 
$\Tube_{(a,b)}$, as the set of $s\in\C$
such that $a<\Re(s)$, resp. $a < \Re(s) <b$.

\begin{lemm}
Let $F$ be a local field of characteristic zero, let $\phi$
be a measurable, bounded and compactly supported function on~$F$.
Then, the Mellin transform 
\[  \mathscr M_F(\phi)(s) =\int_F \abs{x}_F^{s-1} \phi(x)\,\mathrm dx \]
converges for $s\in\C$ with $\Re(s)>0$ and defines
a holomorphic function on~$\Tube_{>0}$. Moreover,
for any real numbers $a$ and $b$ such that $0<a<b$, the function
$\mathscr M_F(\phi(s))$ is bounded in~$\Tube_{(a,b)}$.

Assume that $\phi$ is smooth.
Then there exists a holomorphic function $\phi_1$ on $\Tube_{>-1/2}$
such that $\phi_1(0)=\phi(0)$ and $\mathscr M_F(\phi)(s)=\zeta_F(s)\phi_1(s)$.
Moreover, for any real numbers~$a$ and~$b$ such that $-\frac12<a<b$,
$\phi_1$ is bounded in $\Tube_{(a,b)}$
if $F$ is ultrametric, while there is an upper bound 
\[ \abs{\phi_1(s)}\ll  1+\abs{\Im(s)}, \qquad s\in\Tube_{(a,b)},\]
if $F$ is archimedean.
In particular,
the function 
\[ s \mapsto \mathscr M_F(\phi)(s)-\phi(0)\zeta_F(s) \]
is holomorphic 
and
\[ \lim\limits_{s\ra 0}s\mathscr M_F(\phi)(s)=\mathrm c_F\phi(0). \]
\end{lemm}
\begin{proof}
The absolute convergence of~$\mathscr M_F(\phi)(s)$ for $\Re(s)>0$
and its holomorphy in that domain immediately follow from the absolute
convergence of~$\zeta_F$ and the fact that~$\phi$ 
is bounded and has compact support. So does also its boundedness
in vertical strips $\Tube_{(a,b)}$ for $0<a<b$.

Assume that $\phi$ is smooth.
Let us prove the stated meromorphic continuation.
Again, let~$\phi_0$ denote the characteristic function of the unit ball in~$F$.
Replacing $\phi$ by $\phi-\phi(0)\phi_0$,
it suffices to show that $\mathscr M_F(\phi)$ is holomorphic 
on the half-plane
$\Re(s)>-\frac12 $ whenever $\phi$ is a compactly supported function on~$F$
which is smooth in a neighborhood of~$0$ and satisfies $\phi(0)=0$.

In the real case, one can then write $\abs{\phi}(x)=\abs x_F\psi(x)$ for some
function~$\psi$ which is continuous at~$0$. Similarly, in the complex
case, one can write $\abs{\phi(z)}=\abs{z}_F^{1/2}\psi(z)$.
The required holomorphy of~$\mathscr M_F(\phi)(s)$ on the domain 
$\Tube_{>-\frac12}$ follows.
In the ultrametric case, $\phi$ vanishes in a neighborhood of~$0$
and $\mathscr M_F(\phi)$ even extends to a holomorphic function on~$\C$.

In all cases, observe that $\mathscr M_F(\phi)(s)$ is bounded 
in vertical strips $\Tube_{(\alpha,\beta)}$,
if $\alpha$ and~$\beta$ are real numbers such that $-\frac12<\alpha<\beta$.
Now, the existence of the function~$\phi_1$ and the asserted bound 
on this function follow from the fact that $\zeta_F^{-1}$ is holomorphic
and satisfies  this bound  in these vertical strips.
\end{proof}

\subsubsection{Higher-dimensional Mellin transforms}
We consider a finite family $(F_j)_{1\leq j\leq n}$ of local
fields. For simplicity, we assume these to be finite extensions
of a single completion of~$\Q$. Let us denote by~$V$
the space $\prod F_j$. 
Adopting the terminology introduced for functions of one
variable, we say that a function on~$V$ is smooth if it is 
either $\mathrm C^\infty$, in the case where the $F_j$ are archimedean,
or locally constant, when the~$F_j$ are ultrametric.
If $\alpha=(\alpha_1,\ldots,\alpha_n)$ and~$\beta=(\beta_1,\ldots,\beta_n)$ 
are families of real numbers, we define the sets
$\Tube_{>\alpha}$ and $\Tube_{(\alpha,\beta)}$ as the open
subsets of~$\C^n$ consisting of those $s\in\C^n$
such that $\alpha_j<\Re(s_j)$, resp. $\alpha_j < \Re(s_j) <\beta_j$, 
for $j=1,\dots,n$.
When all $\alpha_j$ are equal to a single one $\alpha$, 
and similarly for the~$\beta_j$, we will
also write $\Tube^n_{>\alpha}$
and $\Tube^n_{(\alpha,\beta)}$, or even $\Tube_{>\alpha}$
and~$\Tube_{(\alpha,\beta)}$ when the dimension~$n$
is clear from the context.
Let $\mathscr F$ be the vector space generated by functions
on $V\times\C^n$ of the form 
$(x;s)\mapsto u(x) v_1(x)^{s_1}\dots v_n(x)^{s_n}$,
where $u$ is smooth, compactly supported on~$V$,
and $v_1,\dots,v_n$ are smooth, positive, and equal to~$1$
outside of a compact subset of~$V$.
For any function $\phi\in\mathscr F$ and $s\in\C^n$, we let 
\[\mathscr M_V(\phi)  (s) = \int_V \abs{x_1}^{s_1-1}\dots\abs{x_n}^{s_n-1}
   \phi(x;s)\, \mathrm dx_1\,\dots \mathrm dx_n, \]
whenever the integral converges.

\begin{prop}\label{prop.mellin-n}
The integral $\mathscr M_V(\phi) (s)$ converges for any $s\in\C^n$ such that
$\Re(s_j)>0$ for all~$j$, and defines an holomorphic function in 
$\Tube_{>0}$.
For any positive real numbers $\alpha$ and~$\beta$,
$\mathscr M_V(\phi)$ is bounded on~$\Tube^n_{(\alpha,\beta)}$.

Moreover, there exists a unique holomorphic function $\phi_1$
on the domain~$\Tube^n_{>-1/2}$ 
such that 
\[ \mathscr M_V(\phi)(s) = \prod_{j=1}^n \zeta_{F_j}(s_j) \phi_1(s)  \]
for $s\in\Tube^n_{>0}$.
Let $\alpha$ and~$\beta$ be real numbers such that
$-\frac12<\alpha<\beta$. Then, $\phi_1$ is bounded on~$\Tube^n_{(\alpha,\beta)}$
if the $F_j$ are ultrametric, while there exists a real number~$c$
such that
\[ \abs{\phi_1(s)} \leq c \prod_{j=1}^n (1+\abs{s_j}), \qquad s\in\Tube_{(\alpha,\beta)}, \]
if the $F_j$ are archimedean.
Moreover, $\phi_1(0)=\phi(0;0)$.
\end{prop}

\section{Geometric Igusa integrals and volume estimates}
\label{sec.igusa-volumes}

\subsection{Igusa integrals over local fields}
\label{ss.geometric-igusa}

We return to our geometric situation: $X$ is a smooth quasi-projective 
variety over a local field~$F$, $D$ is a divisor on~$X$ such
that $D_{\bar F}$ has simple normal crossings. 
We let $\bmA$ be the set of irreducible components
of~$D_{\bar F}$. For $\alpha\in\bmA$, $D_\alpha$ denotes the corresponding
component, while $\Delta_\alpha$ and~$E_\alpha$ are respectively
the sum and the intersection of the conjugates of~$D_\alpha$.

For $\alpha\in\bmA$,
we denote by $\mathsec f_{\alpha}$ the canonical section
of the line bundle~$\mathscr O(\Delta_\alpha)$.
Endow the line bundles $\mathscr O(\Delta_\alpha)$,
as well as the canonical bundle~$\omega_X$ with smooth metrics.
Let $\tau_X$ the corresponding measure on~$X$. 
Following the definitions of~\ref{subsubsec.residue-measures},
the analytic variety $D_A(F)$, for any $\Gamma_F$-invariant subset 
$A\subset\bmA$ supports a natural ``residue measure''.
To simplify explicit formulas below, we define the
measure $\tau_{D_A}$ as the residue measure multiplied by
$\prod_{\alpha\in A} \mathrm c_{F_\alpha}$,
where the constants $\mathrm c_{F_\alpha}$ are defined in
Equation~\eqref{eq.cF}.

The integrals we are interested in 
take the form
\[ \mathscr I(\Phi;(s_\alpha)_{\alpha \in \mathscr A}) = 
\int_{X(F)} \prod_{\alpha\in \mathscr A}
\norm{\mathsec f_{D_\alpha}}(x)^{s_\alpha-1}\, \Phi(x)\, \mathrm d\tau_X(x), \]
where $\Phi$ is a smooth, compactly supported function on $X(F)$
and the~$s_\alpha$ are complex parameters.
Letting~$\Phi$ vary, it is convenient to view 
these integrals as distributions
(of order~$0$) on~$X(F)$.

For any subset $A$ of the $F$-analytic Clemens complex
$\Clan_F(D)$, define also
\[ \mathscr I_A(\Phi;(s_\alpha)_{\alpha \not\in A}) 
= \int_{D_A(F)}  \prod_{\alpha\not\in A} \norm{\mathsec f_{D_\alpha}}(x)^{s_\alpha-1}\, \Phi(x) \, \mathrm d\tau_{D_A}(x). \]


\begin{lemm}\label{lemm.local-convergence}
The integral $\mathscr I(\Phi;(s_\alpha)_{\alpha\in \mathscr A})$ 
converges for $s\in\Tube_{>0}^{\mathscr A}$ and the map
\[ (s_\alpha)_{\alpha\in\mathscr A} \mapsto \mathscr I(\Phi;(s_\alpha)_{\alpha\in \mathscr A}) \]
is holomorphic  on $\Tube_{>0}^{\mathscr A}$.

Similarly, for any~$A\subset\mathscr A$, 
the integral $\mathscr I_A(\Phi;(s_\alpha))$ converges for
$s\in\Tube_{>0}^{\mathscr A\setminus A}$ and 
the function $s\mapsto \mathscr I_A(\Phi;s)$ 
is holomorphic on that domain.
\end{lemm}
We shall therefore call such integrals
``\emph{holomorphic distributions on~$X$}''.

\begin{proof}
For $\xi\in X(F)$, let~$\mathscr A_\xi$ be the set of~$\alpha\in\mathscr A$
such that $\xi\in D_\alpha$ and let 
\[
q\colon U_\xi\ra\prod_{\alpha\in\mathscr A_\xi}\Aff_{F_\alpha}
\times \Aff^r
\] 
be an \'etale chart around~$\xi$ adapted to~$D$,
as in Section~\ref{subsubsec.local-charts}.
Let~$x_\alpha\colon U_\xi\ra \Aff_{F_\alpha}$ be the composite of~$q$
followed by the projection to~$\Aff_{F_\alpha}$,
and $y\colon U_\xi\ra\Aff^r$ be the composite of~$q$ with the projection
to~$\Aff^r$.

These maps induce local coordinates in a neighborhood of~$U_\xi$,
valued in $\prod_{\alpha\in\mathscr A_\xi}F_\alpha\times F^r$
in which the measure~$\tau$ takes the form $\kappa((x_\alpha),y)
\prod \mathrm dx_\alpha\,\mathrm dy$.

By definition of a smooth metric,
there is, 
for any $\alpha\in\mathscr A$, 
a smooth non-vanishing function~$u_\alpha$ on~$U_\xi(F)$
such that $\norm{\mathsec f_{D_\alpha}}=\abs{x_\alpha}_{F_\alpha}u_\alpha$
if $\alpha\in\mathscr A_\xi$, and $\norm{\mathsec f_{D_\alpha}}=u_\alpha$
otherwise.

Further, introducing a partition of unity
(see Section~\ref{subsubsec.partition-unity}),
we see that it suffices to study integrals of the form
\[ \int_{\prod F_\alpha\times F^d} \prod_{\alpha\in\mathscr A_\xi} \abs{x_\alpha}_{F_\alpha}^{s_\alpha-1}
\Phi((x_\alpha),y)\prod_{\alpha\in\mathscr A} u_\alpha^{s_\alpha-1}
\kappa((x_\alpha),y) \prod \mathrm dx_{\alpha}\, \mathrm dy, \]
where $\Phi$ is a smooth function 
with compact support on~$\prod F_\alpha\times F^r$.
The holomorphy of such integrals  is precisely the object
of Proposition~\ref{prop.mellin-n} above.

The case of the integrals $\mathscr I_A(\Phi;(s_\alpha))$ is analogous.
\end{proof}

By the same arguments, but using the meromorphic continuation
of Mellin transforms and the estimate of their growth in vertical strips, 
we obtain the following result.

\begin{prop}\label{prop.local-continuation}
The holomorphic function 
\[ s\mapsto \prod_{\alpha\in\mathscr A}\zeta_{F_\alpha}(s_\alpha)^{-1}
\mathscr I(\Phi;(s_\alpha)) \]
on $\Tube_{>0}^{\mathscr A}$ extends to a 
holomorphic function~$\mathscr M(\Phi;\cdot)$ on $\Tube_{>-1/2}^{\mathscr A}$.
Moreover, for any real numbers $a$ and~$b$ and for any
function~$\Phi$, there is a real number~$c$ such that 
\[ \mathscr M(\Phi;s)\leq c\prod_{\alpha\in\mathscr A}(1+\abs{s_\alpha})
\qquad\text{ for any $s\in\Tube_{(a,b)}^{\mathscr A}$.} \]
\end{prop}

\subsubsection{Leading terms}
We will need to understand the leading terms at the poles 
of these integrals after restricting the parameter $(s_\alpha)$ 
to an affine line in~$\C^{\mathscr A}$.
For $\alpha \in \mathscr A$, 
let $s\mapsto s_\alpha=-\rho_\alpha+\lambda_\alpha s$ 
be an increasing affine function with real coefficients.
To shorten notation, we  write $\mathscr I(\Phi;s)$
instead of $\mathscr I(\Phi;(s_\alpha))$, 
and similarly for the integrals~$\mathscr I_A$.
Define
\[ a(\lambda,\rho) = \max_{\alpha \in \mathscr A} 
                \frac{\rho_\alpha}{\lambda_\alpha}
 \]
and let~$\mathscr A(\lambda,\rho)$ denote the set of all $\alpha\in \mathscr A$
where the maximum is obtained,
\emph{i.e.}, such that $a(\lambda,\rho)=\rho_\alpha/\lambda_\alpha$.
By the preceding lemma, 
$\mathscr I(\Phi;s)=\mathscr I(\Phi,(-\rho_\alpha+\lambda_\alpha s))$
converges for $\Re(s)>a(\lambda,\rho)$ and defines a holomorphic
function there.
Similarly, 
letting
\[ a_A(\lambda,\rho)= \max_{\alpha\not\in A} \frac{\rho_\alpha}{\lambda_\alpha}, \]
the function
$s\mapsto \mathscr I_A(\Phi;s):=\mathscr I_A(\Phi,(-\rho_\alpha+\lambda_\alpha s))$
converges for $\Re(s)>a_A(\lambda,\rho)$.
Notice that $a_A(\lambda,\rho)\leq a(\lambda,\rho)$
in general, and that the inequality may be strict,
\emph{e.g.}, if $\mathscr A(\rho,{\lambda})\subseteq A$.

Let~$\Clan_{F,(\lambda,\rho)}(D)$ be the intersection
of the Clemens complex~$\Clan_F(D)$ with the simplicial subset
$\mathscr P^+(\mathscr A(\lambda,\rho))$
of~$\mathscr P^+(\mathscr A)$. In other words, we remove
from~$\Clan_F(D)$ all faces containing a vertex~$\alpha$ such
that $\rho_\alpha< a(\lambda,\rho) \lambda_\alpha$.

\begin{prop}
\label{prop.igusa-limit}
With the above notation, there exists a positive real number~$\delta$,
and, for any face~$A$ of~$\Clan_{F,(\lambda,\rho)}(D)$ of
maximal dimension, a holomorphic function~$\mathscr J_A$ defined
on~$\Tube_{>a(\lambda,\rho)-\delta}$ with polynomial growth
in vertical strips such that
\[ \mathscr J_A(\Phi;a(\lambda,\rho))
 = \mathscr  I_A(\Phi;(a(\lambda,\rho)\lambda_\alpha-\rho_\alpha)_{\alpha\not\in A}) \]
and such that
\[ \mathscr I(\Phi;s) = \sum_A \mathscr J_A(\Phi;s) \prod_{\alpha\in \mathscr A} \zeta_{F_\alpha}(s-a(\lambda,\rho)), \] 
where the sum is restricted to the faces~$A$ of~$\Clan_{F,(\lambda,\rho)}(D)$
of maximal dimension.
        In particular,
\[ \lim_{s\ra a(\lambda,\rho)}
  \mathscr I(\Phi;s) (s-a(\lambda,\rho))^{\dim\Clan_{F,(\lambda,\rho)}}
 = \sum_{A} 
\mathscr I_A (\Phi;(a(\lambda,\rho)\lambda_\alpha-\rho_\alpha)_{\alpha\not\in A})
\prod_{\alpha\in A} \frac{1}{\lambda_\alpha} . \]
\end{prop}
\begin{proof}
As in the proof of Lemma~\ref{lemm.local-convergence},
we shall use a partition of unity and local coordinates.

Let $\xi\in X(F)$, let~$\mathscr A_\xi$ be the set of~$\alpha\in\mathscr A$
such that $\xi\in D_\alpha$.
In a neighborhood~$\Omega$ of~$\xi$, we have local coordinates
$x_\alpha\in F_\alpha$, for $\alpha\in\mathscr A_\xi$,
and other coordinates $(y_1,\dots,y_r)$ for some integer~$r$.
There exist smooth functions~$u_\alpha$ on~$\Omega$ such that
$\norm{\mathsec f_{D_\alpha}}=\abs{x_\alpha}_{F_\alpha} u_\alpha$ 
if $\alpha\in\mathscr A_\xi$,
and $\norm{\mathsec f_{D_\alpha}}=u_\alpha$ otherwise.
We may assume, after shrinking~$\Omega$,
that $D_\alpha\cap\Omega=\emptyset$ if $\alpha\not\in\mathscr A_\xi$;
then, $u_\alpha$ does not vanish on~$\Omega$, for any~$\alpha\in\mathscr A$.
Finally, the restriction to~$\Omega$ of the measure $\tau_X$ can  be written
as $\kappa \mathrm dy_1\dots \mathrm dy_r \prod_{\alpha\in\mathscr A_x}\mathrm dx_\alpha $,
for some smooth function~$\kappa$ on~$\Omega$.
Let~$\theta_\xi$ be a smooth function with compact support in~$\Omega$.

The integral
\[ \mathscr I_\xi (s) = \int_\Omega  \prod_{\alpha\in\mathscr A}
\norm{\mathsec f_{D_\alpha}}(x)^{s_\alpha-1}\Phi(x) \theta_\xi(x) \,\mathrm d\tau_X(x)\]
can be rewritten as
\[ \mathscr I_\xi(s)=
\int_\Omega
 \prod_{\alpha\in\mathscr A_\xi}\abs{x_\alpha}_{F_\alpha}^{-\rho_\alpha+s\lambda_\alpha-1}
 \times \prod_{\alpha\in\mathscr A} u_\alpha^{s_\alpha-1}(x)
 \Phi(x)\theta_\xi(x)\kappa(x)\,
 \prod_{\alpha\in\mathscr A_\xi}
 \mathrm dx_\alpha \times \mathrm dy.
\]
Let $a=a(\lambda,\rho)$; when $s\ra a$, only the variables $s_\alpha$
such that $\rho_\alpha=a \lambda_\alpha$ and $\alpha\in\mathscr A_\xi$
contribute a pole.
Write $A =\mathscr A_\xi\cap\mathscr A(\lambda,\rho)$ for this subset.
Applying the regularization procedure which led to Prop.~\ref{prop.local-continuation},
but only to the variables~$x_\alpha$ with $\alpha\in A$,
furnishes an expression of the form
\[ \mathscr I_\xi (s)=\mathscr J_\xi(s) \prod_{\alpha\in A}\zeta_{F_\alpha}(-\rho_\alpha+s\lambda_\alpha) \]
for the integral~$\mathscr I_\xi(s)$,
where $\mathscr J_\xi$ is holomorphic for $\Re(s)>a(\lambda,\rho)-\delta$
and has polynomial growth in vertical strips.
Moreover,
\begin{align*} 
\lim_{s\ra a} \mathscr J_\xi(s) & = \lim_{s\ra a} \mathscr I_\xi(s)
\prod_{\alpha\in A}
      \zeta_{F_\alpha}(-\rho_\alpha+s\lambda_\alpha)^{-1}  \\
& = \int_{\Omega\cap D_A(F)} 
  \prod_{\alpha\in\mathscr A} u_\alpha^{-\rho_\alpha+a\lambda _\alpha-1}
 \prod_{\alpha\in\mathscr A_\xi\setminus A} \abs{x_\alpha}_{F_\alpha}^{-\rho_\alpha+a\lambda _\alpha-1}
 \theta_\xi(x)\Phi(x)\kappa(x)\,
\prod_{\alpha\in\mathscr A_\xi\setminus A}\mathrm dx_\alpha\, \mathrm dy .
\end{align*}
By the definition of the residue measure on~$D_A(F)$
(Section~\ref{subsubsec.residue-measures})
and its normalization used here, one thus has
\[\lim_{s\ra a}\mathscr I_\xi(s)
\prod_{\alpha\in  A}\zeta_{F_\alpha}(-\rho_\alpha+s\lambda_\alpha)^{-1}
= \prod_{\alpha\in A}\frac1{\mathrm c_{F_\alpha}} \int_{D_A(F)}
 \prod_{\alpha\not\in A} \norm{\mathsec f_{D_\alpha}}^{-\rho_\alpha+ a \lambda_\alpha-1}
 \theta_\xi(x)\Phi(x)\,\mathrm d\tau_{D_A}(x).\]
so that
\[ \lim_{s\ra a} (s-a)^{\Card A}  \mathscr I(\theta_\xi\Phi;s) 
=\mathscr I_A(\theta_\xi\Phi;(-\rho_\alpha+a\lambda_\alpha)) 
\prod_{\alpha\in A} \lambda_\alpha^{-1}.\]
Observe that $A$ is a maximal face of~$\Clan_{F,(\lambda,\rho)}$,
though maybe not one of maximal dimension.

Now choose the functions~$\theta_\xi$
so that they form a finite partition of unity, \emph{i.e.}, $\sum\theta_\xi=1$,
and only finitely many~$\theta_\xi$ are not zero.
Then $\mathscr I(\Phi;s)=\sum\mathscr I_\xi(s)$; regrouping the non-zero terms
according to the minimal face of the Clemens complex~$\Clan_{F,(\lambda,\rho)}$
to which they correspond furnishes 
the desired expression for~$\mathscr I(\Phi;s)$.

Let~$b$ be the dimension of this complex
and let $\Clanmax_{F,(\lambda,\rho)}$
be the set of its faces of dimension~$b$.
Granted the previous limits computations, one has
\begin{align*}
 \lim_{s\ra a} (s-a)^b\mathscr I(\Phi;s) 
& = \sum_\xi \lim_{s\ra a} (s-a)^b \mathscr I(\theta_\xi\Phi;s) \\
& = \sum_{A\in\Clanmax_{F,(\lambda,\rho)}}
 \prod_{\alpha\in A} \lambda_{\alpha}^{-1}
    \mathscr I_{A}(\Phi;(-\rho_\alpha+a\lambda_\alpha)),
\end{align*}
as claimed.
\end{proof}

\begin{coro}
If $\Phi\equiv 1$ or, more generally, if the
restriction of~$\Phi$ to $D(F)$ is not identically~$0$,
then the order of the pole of~$\mathscr I(\Phi;s)$ at $s=a(\lambda,\rho)$
is equal to
\[ 1+\dim \Clan_{F,(\lambda,\rho)}(D). \]
\end{coro}

%
%
%
%

%
%
%

\subsubsection{The case of good reduction}

We recast in this geometric context a formula of J.~Denef
(\cite[Theorem 3.1]{denef87}, see also~\cite[Theorem 9.1 
and Theorem 11.2]{chambert-loir-t2002}).

Assume that $F$ is non-archimedean and that our situation
comes from a smooth model $\mathscr X$ over $\mathfrak o_F$,
that the divisors $D_\alpha$ extend to divisors $\mathscr D_\alpha$
on~$\mathscr X$
whose sum becomes a relative Cartier divisor with strict normal crossings after
base change to a finite \'etale extension of~$\mathfrak o_F$, 
and that all metrics are defined by this model. 
The residue field of~$F$ is denoted by~$k$, its cardinality by~$q$.
For $\alpha\in\mathscr A$, the extension $F\subset F_\alpha$
is unramified by the good reduction hypothesis;
we denote by $f_\alpha$ its degree;
let $\mathfrak o_{F_\alpha}$ be the ring of integers of~$F_\alpha$
and $\mathfrak m_\alpha$ its maximal ideal.

The Zariski closure of the scheme $E_\alpha$ 
is a smooth subscheme $\mathscr E_\alpha$ of~$\mathscr X$,
of relative codimension $d_\alpha$, such that
$\mathscr E_\alpha(\mathfrak o_F)=\mathscr D_\alpha(\mathfrak o_F)$.
Similarly,  the Zariski closure of~$E_A$ is a smooth geometrically connected
subscheme~$\mathscr E_A$ of~$\mathscr X$ of codimension~$d_A$
and $\mathscr E_A(\mathfrak o_F)=\mathscr D_A(\mathfrak o_F)$.
For any subset $A\subset\mathscr A$,
we let $\tau_{D_A}$ denote the Tamagawa measure on $D_A(F)=E_A(F)$.
Since the extensions $F\subset F_\alpha$ are unramified,
one also has $\mathscr D_\alpha(k)=\mathscr E_\alpha(k)$
for any $\alpha\in\mathscr A$, and $\mathscr D_A(k)=\mathscr E_A(k)$
for any $A\subset\mathscr A$.

Assume also that the function~$\Phi$ is constant
on residue classes; the induced function on $\mathscr X(k)$ will
still be denoted by~$\Phi$.

\begin{prop}\label{prop.denef}
Under the above conditions, one has
\[ \mathscr I(\Phi;(s_\alpha))
 =   \sum_{A\subset \mathscr A} 
         \big(q^{-1}\mu(\mathfrak o_{F})\big)^{\dim X}
           \prod_{\alpha\in A} 
           \frac{q^{f_\alpha}-1}{q^{f_\alpha s_\alpha}-1}
           \big(\sum_{\tilde\xi \in \mathscr D_A^\circ(k)} \Phi(\tilde\xi)\big)
           . 
\]
In particular, for $\Phi=1$,
one has
\[ \mathscr I(1;(s_\alpha))
 =   \sum_{A\subset \mathscr A} 
         \big(q^{-1}\mu(\mathfrak o_F)\big)^{\dim X}
           \prod_{\alpha\in  A} 
           \frac{q^{f_\alpha}-1}{q^{f_\alpha s_\alpha}-1}
           \Card(\mathscr D_A^\circ(k)). 
\]
\end{prop}
\begin{proof}
Let $\tilde\xi\in\mathscr X(k)$, let $\mathscr A_{\tilde\xi}=\{\alpha\in\mathscr A\,;\, \tilde\xi\in \mathscr D_\alpha(k)\}$, 
so that $\tilde\xi$ belongs to the open
stratum~$\mathscr D_{\mathscr A_{\tilde\xi}}^\circ$.
By the good reduction hypothesis, we can introduce local (étale) 
coordinates $x_\alpha\in \mathfrak m_\alpha$ 
(for $\alpha\in \mathscr A_{\tilde\xi}$) and 
$y_\beta\in\mathfrak m$ (for $\beta$
in a set~$\mathscr B_{\tilde\xi}$ 
of cardinality $\dim X-\sum_{\alpha\in \mathscr A_{\tilde\xi}}f_\alpha$)
on the residue class $\Omega_{\tilde\xi}$ of~$\tilde\xi$,
such that $\mathscr D_\alpha$ is defined by the equation $x_\alpha=0$ 
on $\Omega_{\tilde\xi}$. Then the local Tamagawa measure
identifies with the measure $\prod_{\alpha\in \mathscr A_{\tilde\xi}} 
\mathrm dx_\alpha\times
\prod_{\beta\in \mathscr B_{\tilde\xi}}\mathrm dy_\beta$ 
on $\prod_{\alpha\in \mathscr A_{\tilde\xi}}\mathfrak m_\alpha
\times\prod_{\beta\in\mathscr B_{\tilde\xi}}\mathfrak m$.

Recall also  (see~\S\ref{subsubsec.local-zeta})
that for any ultrametric local field~$F$,
with ring of integers~$\mathfrak o_F$ and maximal ideal~$\mathfrak m$,
and any complex number $s$ such that $\Re(s)>0$,
one has
\[  q  \int_{\mathfrak m} \abs{x}_F^{s-1} \,\mathrm dx 
     =  \frac{q-1}{q^{s}-1} \mu(\mathfrak o_F), \]
where $q$ is the cardinality of the residue field.

These formulas, applied to the fields~$F_\alpha$, and
the decomposition of the integral $\mathscr I$ as a sum of similar
integrals over the residue classes $\tilde\xi\in\mathscr X(k)$, give us
\begin{align*}
\mathscr I(\Phi;(s_\alpha))
 &= 
\sum_{\tilde\xi \in \mathscr X(k) }
    \Phi(\tilde\xi) 
    \big(q^{-1}\mu(\mathfrak o_F)\big)^{\Card\mathscr B_{\tilde\xi}}
    \prod_{\alpha\in \mathscr A_{\tilde\xi} }
    \int_{\mathfrak m_\alpha} \abs{x_\alpha}_{F_\alpha}^{s_\alpha-1}\,\mathrm dx_\alpha \\
 &= \big(q^{-1} \mu(\mathfrak o_F)\big)^{\dim X}
     \sum_{\tilde\xi \in\mathscr X(k)}
      \Phi(\tilde\xi) \prod_{\alpha\in \mathscr A_{\tilde\xi}}
         \frac{q^{f_\alpha}-1}{q^{f_\alpha s_\alpha}-1}  ,
\end{align*}
since the residue field of~$F_\alpha$ has cardinality~$q^{f_\alpha}$
and $\mu(\mathfrak o_F)=\mu(\mathfrak o_{F_\alpha})$.
Let us interchange the order of summation: one gets
\[
\mathscr I(\Phi;(s_\alpha)) 
 = \big(q^{-1} \mu(\mathfrak o_F)\big)^{\dim X}
     \sum_{A\subset \mathscr A}
      \prod_{\alpha\in A} \frac{q^{f_\alpha}-1}{q^{f_\alpha s_\alpha}-1}
           \sum_{\tilde\xi \in \mathscr D_A^\circ(k)}  \Phi(\tilde\xi) .
\]
In particular, if $\Phi$ is the constant function~$1$, one has
\[ 
\mathscr I(1,(s_\alpha))
=   \sum_{A\subset \mathscr A} 
         \big(q^{-1}\mu(\mathfrak o_F)\big)^{\dim X}
           \prod_{\alpha\in A} 
           \frac{q^{f_\alpha}-1}{q^{f_\alpha s_\alpha}-1}
           \Card (\mathscr D_A^\circ(k)). 
\qedhere
\]
\end{proof}

By Weil's formula (\cf Equation~\eqref{eq.weil-formula}), one then has
\[ \tau_{D_A}(\mathscr E_A^\circ (\mathfrak o_F))
      = \big( q^{-1} \mu(\mathfrak o_F)\big)^{d_A} \big)
      \Card \mathscr E_A^\circ (k)
      = \big( q^{-1} \mu(\mathfrak o_F)\big)^{d_A} \big)
      \Card \mathscr D_A^\circ (k)
. \]
Since moreover $D_A(F)=E_A(F)$,
the last formula of the Proposition can be rewritten as
\[
 \mathscr I(1,(s_\alpha))
 =   \sum_{A\subset \mathscr A} 
         \big(q^{-1}\mu(\mathfrak o_F)\big)^{d_A}
           \prod_{\alpha\in A} 
           \frac{q^{f_\alpha}-1}{q^{f_\alpha s_\alpha}-1}
           \tau_{D_A} (\mathscr D_A^\circ(\mathfrak o_F)). \]

\subsection{Volume asymptotics over local fields}
\label{subsec.geometric-volume}

Let $X$ be a smooth projective variety over a local field~$F$.
Assume that $X$ is purely of dimension~$n$.

Let $D$ be an effective divisor in~$X$, denote by~$\mathscr A$
the set of its irreducible components, by $D_\alpha$ the
component corresponding to some $\alpha\in\mathscr A$ and
by $d_\alpha$ its multiplicity.
We have $d_\alpha>0$ for all~$\alpha$ and $D=\sum d_\alpha D_\alpha$.

Let $U=X\setminus D$, and assume that $\omega_X(D)$
is equipped with a metrization. 
Let~$\tau_{(X,D)}$ denote the corresponding measure on~$U(F)$
(see \S\ref{ss.metrics-measures}).
Endow the line bundles $\mathscr O_X(D)$
and $\mathscr O_{X}(D_\alpha)$, for $\alpha\in\mathcal A$,
with metrics, in such a way that the natural isomorphism
$\mathscr O_X(D)\simeq \bigotimes \mathscr O_X(D_\alpha)^{d_\alpha}$
is an isometry.
For $\alpha\in\mathcal A$, denote by~$\mathsec f_\alpha$
the canonical section of~$\mathscr O_{X}(D_\alpha)$;
denote by~$\mathsec f_D$ the canonical section of~$\mathscr O_X(D)$.
One has 
$\mathsec f_D=\prod_{\alpha\in\mathcal A}\mathsec f_{\alpha}^{d_\alpha}$.

We also endow $\omega_{X}$ with the metric which makes
the isomorphism $\omega_X(D)\simeq\omega_X\otimes\mathscr O_X(D)$
an isometry.
Letting~$\tau_{X}$ be the Tamagawa measure on~$X(F)$
defined by the metrized line bundle~$\omega_X$,
we have the following equalities:
\[ \mathrm d\tau_{(X,D)}(x) = \norm{\mathsec f_D(x)}^{-1}\,\mathrm d\tau_X(x)
=  \prod_{\alpha} \norm{\mathsec f_\alpha(x)}^{-d_\alpha} \, \mathrm d\tau_{X}(x). \]

Let $L$ be an effective divisor in~$X$ whose support contains the support 
of~$D$; assume that the corresponding line bundle $\mathscr O_X(L)$
is endowed with a metric. The norm of its canonical section~$\mathsec f_L$
vanishes on~$L$, hence on~$D$. Consequently, for any positive
real number~$B$,
the set of all $x\in X(F)$ such $\norm{\mathsec f_L(x)}\geq 1/B$
is a closed subset of~$X(F)$, which is contained in $U(F)$,
hence is compact in~$U(F)$. Consequently, its volume
with respect to the measure~$\tau_{(X,D)}$,
\begin{equation}
V(B) = 
 \int_{\norm{\mathsec f_L(x)}\geq 1/B}
 \mathrm d\tau_{(X,D)}(x) ,
\end{equation}
is finite for any~$B>0$.
We are interested in its asymptotic behavior 
when $B\ra\infty$.
Let us 
introduce the Mellin transform of the function~$\norm{\mathsec f_L}$
with respect to the measure~$\tau_{(X,D)}$, namely:
\begin{equation}
 Z(s) = \int_{U(F)} \norm{\mathsec f_L(x)}^s \mathrm d\tau_{(X,D)}(x).
\end{equation}

The analytic properties of $Z(s)$ and $V(B)$
strongly depend on the geometry of the pair~$(X,D)$.
We will assume throughout that over the algebraic closure~$\bar F$, 
the divisor~$D$ 
has strict normal crossings in~$X$; in that case we will see
that the answer can be stated in terms of the analytic
Clemens complex of~$D$.
In principle, using resolution of singularities, we can reduce
to this situation, even if it may be difficult in explicit examples
(see~\cite{hassett-tschinkel2003} for a specific computation related
to the asymptotic behavior of integral points of bounded height
established in~\cite{duke-r-s1993}).

For $\alpha\in\mathcal A$,
let $\lambda_\alpha$ be the multiplicity of~$D_\alpha$ in~$L$; 
the divisor $\Delta=L-\sum\lambda_\alpha D_\alpha$ is effective
and all of its irreducible components meet~$U$.
We denote by~$\mathsec f_\Delta$ the canonical section of
the line bundle~$\mathscr O_X(\Delta)$; we endow this line
bundle with a metric so that $\norm{\mathsec f_L}=\norm{\mathsec f_\Delta}
 \prod\norm{\mathsec f_\alpha}^{\lambda_\alpha}$.

Following the conventions of Section~\ref{subsubsec.Q-Cartier},
the results extend to the case where~$D$
and~$L$ are $\Q$-Cartier divisors; in that case,
the coefficients~$\lambda_\alpha$ and~$d_\alpha$ are rational numbers.

Let $\sigma=\max (d_\alpha-1)/\lambda_\alpha$, 
the maximum being over all $\alpha\in\mathcal A$ 
such that $D_\alpha(F)\neq\emptyset$.
If there is  no such~$\alpha$, we let $\sigma=-\infty$ by convention;
this means precisely that $U(F)$ is compact.
Only the case $\sigma\geq 0$ will really matter.
Indeed, as we shall see below, the condition $\sigma<0$ is equivalent
to the fact that $U(F)$ has finite volume with respect to~$\tau_{(X,D)}$.

Let $\Clan_{F,(L,D)}(D)$ be the subcomplex
of the analytic Clemens complex~$\Clan_F(D)$ consisting
of all non-empty subsets~$A\subset\bmA$ such that $E_A(F)\neq\emptyset$
and $d_\alpha=\lambda_\alpha\sigma+1$ for any $\alpha\in A$.

Let $b=\dim\Clan_{F,(L,D)}(D)$.
For any face~$A$ of maximal dimension~$b$ of~$\Clan_{F,(L,D)}(D)$,
let $D_A=\bigcap_{\alpha\in A}D_\alpha$ be the corresponding
stratum of $X$. The subset $D_A(F)$ carries a natural measure~$\mathrm d\tau_{D_A}$
and we define
\begin{equation}
 Z_A(s) = \int_{D_A(F)} \norm{\mathsec f_\Delta(x)}^s
 \prod_{\alpha\not\in A} \norm{\mathsec f_\alpha(x)}^{s\lambda_\alpha-d_\alpha} \, \mathrm d\tau_{A}(x).
\end{equation}

\begin{prop}\label{prop.igusa-volume}
Let $\sigma$, $\Clan_{F,(L,D)}(D)$ and~$b$ be defined as above. Then
the integral defining~$Z(s)$ converges for $\Re(s)>\sigma$
and defines a holomorphic function in that domain.

Assume that~$\sigma\neq-\infty$.
Then there is a positive real number~$\delta$
such that $Z$ has a meromorphic continuation to a half-plane 
$\Re(s)>\sigma-\delta$, with
a pole of order~$b=\dim \Clan_{F,(L,D)}(D)$ 
at $s=\sigma$ 
with leading coefficient
\[ \lim_{s\ra\sigma} (s-\sigma)^bZ(s) 
= \sum_{\substack{A\in \Clan_{F,(L,D)}(D) \\ \dim A=b}} Z_A (\sigma) \prod_{\alpha\in A}\frac1{\lambda_\alpha} \]
and moderate growth in vertical strips.

When $F=\R$ or~$\C$, $Z$ has no other pole
provided $\delta$ is chosen sufficiently small.

When $F$ is ultrametric, there is a positive integer~$f$
such that $(1-q^{f(\sigma-s)})^bZ(s)$ is holomorphic on the
half-plane $\{\Re(s)>\sigma-\delta\}$, again provided
$\delta$ is sufficiently small.
\end{prop}
\begin{proof}
By definition,
\[ Z(s)
 = \int_{X(F)} \norm{\mathsec f_\Delta}^s
\prod_\alpha \norm{\mathsec f_\alpha(x)}^{s \lambda _\alpha  -d_\alpha}
 \mathrm d\tau_{X}(x), \]
an integral of the type studied in~\S\ref{s.igusa}.
Precisely, using the notations introduced in~\S\ref{ss.geometric-igusa},
we have
$Z(s)=\mathscr I(\mathbf 1,(s+1;s\lambda_\alpha -d_\alpha+1))$, where
the first parameter~$s$ refers to the divisor~$\Delta$,
while for each $\alpha\in\mathcal A$,
 the parameter $s_\alpha=s\lambda_\alpha-d_\alpha$ corresponds
to the divisor~$D_\alpha$.
Similarly,
\[ Z_A(s)=\mathscr I_A(\mathbf 1,(s+1;s\lambda_\alpha-d_\alpha+1)).\]
By Lemma~\ref{lemm.local-convergence},
this integral converges and defines a holomorphic function
as long as $\Re(s)>0$ and $\Re(s)\lambda_\alpha>d_\alpha$.
This shows the holomorphy of~$Z(s)$ for $\Re(s)>\sigma$.

Assume that $\sigma\neq-\infty$.
By~Proposition~\ref{prop.local-continuation}, the function~$Z$
has a meromorphic continuation to the domain of~$\C$
defined by the inequalities~$\Re(s_\alpha)>-\frac12$ and $\Re(s)>0$,
hence to some domain of the form~$\Re(s)>\sigma-\delta$.

In the ultrametric case, the existence of a positive integer~$f$
such that $(1-q^{(\sigma-s)f})^b Z(s)$ has no pole
on such a half-plane also follows  directly
from Proposition~\ref{prop.local-continuation}
(one may take for~$f$ the l.c.m. of the~$f_\alpha$),
as well as the growth in vertical strips.

It remains to prove the asserted behavior at $s=\sigma$.
By Proposition~\ref{prop.igusa-limit},
\[ \lim_{s\ra \sigma} (s-\sigma)^b Z(s)=
\sum_{\substack{A\in\mathscr A_{L,D}^{\an} \\ \dim A=b}}
\mathscr I_A (\mathbf 1;
    (\sigma +1;\sigma \lambda_\alpha-d_\alpha+1)_{\alpha\not\in A})
    \prod_{\alpha\in A} \frac1{\lambda_\alpha}
= \sum_{\substack{A\in\mathscr A_{L,D}^{\an} \\ \dim A=b}}
          Z_A(\sigma) 
    \prod_{\alpha\in A} \frac1{\lambda_\alpha}.
\]
\end{proof}

Using the Tauberian theorem~\ref{theo.tauber} recalled in the appendix,
we obtain the following estimate for the volume~$V(B)$
in the archimedean case.

\begin{theo}
\label{theo.geometric-volume}
Assume that $F=\R$ or~$\C$ and that $\sigma\geq 0$. This implies that $b\geq 1$.

There exists a polynomial~$P$ 
and a positive real number~$\delta$ such that
\[ V(B) = B^\sigma P(\log B) + \mathrm O(B^{\sigma-\delta}). \]
Moreover, if $\sigma>0$, then $P$ has degree~$b-1$ and leading
coefficient
\[ \operatorname{lcoeff}(P) = \frac1{\sigma(b-1)!}
\sum_{\substack{A\subset  \Clan_{F,(L,D)}(D)  \\ \dim A=b}}
Z_A(\sigma) \prod_{\alpha\in A} \frac{1}{\lambda_\alpha};\]
otherwise, if $\sigma=0$, then $P$ has degree~$b$ and its leading
coefficient satisfies
\[ \operatorname{lcoeff}(P) = \frac1{b!}
\sum_{\substack{A\subset  \Clan_{F,(L,D)}(D)  \\ \dim A=b}}
Z_A(\sigma) \prod_{\alpha\in A} \frac{1}{\lambda_\alpha}.\]
\end{theo}

Considering integrals of the form
\[ Z_\Phi(s) = \int_{U(F)} \Phi(x) \norm{\mathsec f_L(x)}^s\, \mathrm d\tau_{(X,D)}(x), \]
or applying the abstract equidistribution theorem 
(Proposition~\ref{prop.abstract.equi}),
we deduce the following corollary (``equidistribution of height balls''):
\begin{coro}\label{coro.equidistribution.local}
Assume that $F=\R$ or~$\C$ and that~$\sigma>0$.
Then $b\geq 1$ and, when $B\ra+\infty$, the family of measures
\[   V(B)^{-1} \mathbf 1_{\{ \norm{\mathsec f_L(x)}\geq 1/B\}}
      \,\mathrm d\tau_{(X,D)}(x) \]
converges tightly to the unique probability measure which
is proportional to
\[ \sum_{\substack{A\subset \mathscr A_{L,D}^\an \\ \dim A=b}}
      \norm{\mathsec f_\Delta(x)}^\sigma 
       \prod_{\alpha\in A} \frac{1}{\lambda_\alpha}
       \prod_{\alpha\not\in A} \norm{\mathsec f_\alpha(x)}^{-1} \,\mathrm d\tau_{D_A}(x). \]
\end{coro}

In the case of ultrametric local fields, 
our analysis leads to a good understanding
of the Igusa zeta functions but the output
of the corresponding Tauberian theorem,
\emph{i.e.}, the discussion leading to 
Corollaries~\ref{coro.theo.tauber2}, \ref{coro.theo.tauber2b}
and \ref{coro.theo.tauber2c}
in the Appendix, is less convenient  since
one can only obtain an asymptotic expansions for $V(q^N)$,
when $N$ belongs to a fixed congruence class modulo some positive integer~$f$.
In fact, when all irreducible components of~$D$ are geometrically irreducible,
one may take $f=1$ in Proposition~\ref{prop.igusa-volume}.
In that case, Corollary~\ref{coro.theo.tauber2b} even leads
to a precise asymptotic expansion in that case.

We leave the detailed statement to the interested reader 
and we content ourselves with the following corollary.
\begin{coro}\label{coro.asymptotic.volume-ultrametric}
Assume that $F$ is ultrametric and that~$\sigma\geq 0$.
Let $b^*=b$ if $\sigma=0$, and $b^*=b-1$ if $\sigma>0$.
Then, when the integer~$N$ goes to infinity,
\[  0< \liminf \frac{V(q^N)}{q^{N\sigma} N^{b^*}} \leq 
   \limsup \frac{V(q^N)}{q^{N\sigma} N^{b^*}} <\infty. \]
\end{coro}
\begin{proof}
Changing metrics modifies the volume forms and the height functions
by a factor which is lower- and upper-bounded; this does not
affect the result of the corollary.
Consequently, we may assume that all metrics are smooth
and that the function~$\norm{\mathsec f_L}$ is $q^\Z$-valued.
The Igusa zeta function~$Z(s)$ is then~$2i\pi/\log q$-periodic
and has a meromorphic continuation of the form
$\sum_A \Phi_A(s)\prod_{\alpha\in A}(1-q^{f_\alpha\lambda_\alpha(\sigma-s)})$,
for some functions~$\Phi_A$ which are holomorphic on an open half-plane
containing the closed half-plane given by $\{\Re(s)\geq\sigma\}$.
Let $f$ be any positive integer such that $f$ is 
an integral multiple of $f_\alpha\lambda_\alpha$, 
for any~$\alpha\in \mathscr A$;
we see that $Z(s)(1-q^{f(\sigma-s)})^b$ extends holomorphically
to this open half-plane and 
it now suffices to apply Corollary~\ref{coro.theo.tauber2}.
\end{proof}

\subsection{Adelic Igusa integrals}
\label{sec.adelic-igusa}

\subsubsection{Geometric setup}
Let $F$  be a number field, let $\AD_F$ be the ring of adeles of~$F$.
More generally, if $S$ is a finite set of places of~$F$, 
let $\mathbf A_F^S$ be the restricted product of local fields~$F_v$, 
for $v\not\in S$.

We fix an algebraic closure~$\bar F$ of~$F$;
for each place~$v$, 
we also  fix a decomposition group $\Gamma_v$ at~$v$
in the Galois group $\Gamma_F=\Gal(\bar F/F)$.

Let $\bar X$ be a smooth projective variety over~$F$, 
let $(D_\alpha)_{\alpha\in \mathscr A}$
be a family of irreducible divisors in~$\bar X$ whose
sum $\Delta=\sum_{\alpha\in \mathscr A} D_\alpha$
is geometrically a divisor with strict normal crossings.
For $\alpha\in\mathscr A$,
let $\mathsf f_{D_\alpha}$ denote the canonical section
of the line bundle $\mathscr O_{\bar X}(D_\alpha)$
and let $F_\alpha$ be the algebraic closure of~$F$
in the function field of~$D_\alpha$.
Let $\bar{\mathscr A}$ be the set of irreducible
components of the divisor~$\Delta_{\bar F}$; it carries an action
of $\Gamma_F$, the set of orbits of which, $\bar{\mathscr A}/\Gamma_F$,
identifies canonically
with the set $\mathscr A$.

Endow all line bundles $\mathscr O_{\bar X}(D_\alpha)$, as well
as the canonical line bundle~$\omega_{\bar X}$, with adelic metrics.

Let $\mathscr B$ be any subset of~$\mathscr A$, corresponding
to a subset $\bar{\mathscr B}$ of~$\bar{\mathscr A}$
which is stable under the action of~$\Gamma_F$.
Let $Z=\bigcup_{\alpha\in\mathscr B} D_\alpha$
be the union of the corresponding divisors,
$D=\bigcup_{\alpha\not\in\mathscr B}D_\alpha$ the union
of the other divisors,
and let us define $X=\bar X\setminus Z$ and $U=X\setminus D$.

The local spaces $X(F_v)$, $U(F_v)$, for any place~$v$ of~$F$, 
and the adelic spaces $X(\AD_F)$, $U(\AD_F)$
are locally compact and carry Tamagawa measures
$\tau_{X,v}$, $\tau_{U,v}$, $\tau_X$ and~$\tau_U$
(see Definition~\ref{defi.tamagawa-measure})
which are Radon measures, \emph{i.e.}, finite on compact subsets.
The set of irreducible components $\mathscr A_v$ of the divisor~$\Delta_{F_v}$
is in natural bijection with the set of $\Gamma_v$-orbits
in~$\bar{\mathscr A}$.  For any such orbit~$\alpha$, we will 
denote by $D_\alpha$ the corresponding divisor on $X_{F_v}$.

Our aim here is to establish analytic properties
of the adelic integral
\[ \mathscr I (\Phi; (s_\alpha)) =
\int_{U(\AD_F)} \prod_{\alpha\in \mathscr A}
\prod_{v} \norm{\mathsf f_{D_\alpha}}_v^{s_\alpha-1} \Phi(x)
\, \mathrm d\tau_U, \]
when $\Phi$ is the restriction to~$U(\AD_F)$
of a smooth function with compact support on $X(\AD_F)$.

 
It will be convenient to view the map $\alpha\mapsto s_\alpha$
as a $\Gamma_F$-equivariant map from $\bar{\mathscr A}$ to~$\C$.
In other words, for each $\alpha\in\bar{\mathscr A}$, we 
let $s_\alpha=s_{[\alpha]}$, where $[\alpha]\in\mathscr A$ is
the orbit of~$\alpha$ under~$\Gamma_F$. More generally,
if $\beta$ is any subset of such an orbit, we define $s_\beta=s_\alpha$,
for any $\alpha\in\bar{\mathscr A}$ belonging to~$\beta$
(it does not depend on the choice of~$\alpha$).

\subsubsection{Convergence of local integrals}
Assume that $\Phi=\prod_v \Phi_v$ is a product
of smooth functions and define, for any $v\in\Val(F)$,
\[ \mathscr I_v (\Phi_v; (s_\alpha)) =
\int_{U(F_v)}
 \prod_{\alpha\in \mathscr A}
 \norm{\mathsf f_{D_\alpha}}_v^{s_\alpha-1} \Phi_v(x) \,
   \mathrm L_v(1,\EP (U))\mathrm d\tau_{X,v}. \]
When the local integrals $\mathscr I_v$ converge  absolutely,
as well as the infinite product 
$\prod_v \mathscr I_v(\Phi;(s_\alpha))$,
then the integral $\mathscr I(\Phi;(s_\alpha))$
exists and one has an equality
\[ \mathscr I(\Phi; (s_\alpha)) = \mathrm L_*(1,\EP (U))^{-1}
    \prod_v \mathscr I_v(\Phi_v;(s_\alpha)). \]

Let $\alpha\in\mathscr A$.
Let us decompose the $\Gamma_F$-orbit~$\alpha$ as a union of disjoint
$\Gamma_v$-orbits $\alpha_1,\dots,\alpha_r$.
Then $D_\alpha=\sum_{i=1}^r D_{\alpha_i}$ and $\mathsec f_{D_\alpha}
=\prod_{i=1}^r \mathsec f_{D_{\alpha_i}}$.
It follows that the integral~$\mathscr I_v$ can be rewritten as
\[ \mathscr I_v (\Phi_v; (s_\alpha)) =
\int_{U(F_v)}
 \prod_{\alpha\in \mathscr A_v}
 \norm{\mathsf f_{D_\alpha}}_v^{s_\alpha-1} \Phi_v(x) \,
   \mathrm L_v(1,\EP (U))\mathrm d\tau_{X,v}. \]
By Lemma~\ref{lemm.local-convergence}, $\mathscr I_v(\Phi_v;(s_\alpha))$
converges absolutely when $\Re(s_\alpha)>0$ for each $\alpha\in\mathscr A$.
If moreover $\Phi_v$ has compact support in~$X(F_v)$,
then the conditions $\Re(s_\alpha)>0$ for $\alpha\in\mathscr B$
are not necessary.

The decomposition $\alpha=\bigcup \alpha_i$ corresponds to
the decomposition  $F_\alpha\otimes_F F_v =\prod_{i=1}^r F_{\alpha_i}$.
Observe that the local factor of Dedekind's zeta function~$\zeta_{F_\alpha}$
at~$v$ is given by the formula
\[ \zeta_{F_\alpha,v}(s) = \prod_{i=1}^r (1-q_v^{-f_{\alpha,i}s})^{-1}, \]
where for each~$i$, $f_{\alpha,i}=[F_{\alpha,i}:F_v]$.
Let $\zeta_{F_\alpha}^*(1)$ be the residue at $s=1$ of this
zeta function.

\subsubsection{Convergence of an Euler product}
To study the convergence of the product, we may ignore a finite
set of places.

Let $S$ be a finite set of places containing the  archimedean places
so that, for all other places,
all metrics are defined by good integral models~$\overline{\mathscr X}$,
$\mathscr X$,  $\mathscr U$, \dots, over $\Spec\mathfrak o_{F,S}$.
Assume moreover that for any $v\not\in S$, $\Phi_v$ is the characteristic
function of~$\mathscr X(\mathfrak o_v)$.

By Denef's formula (Prop.~\ref{prop.denef}),
one has
\[ \mathscr I_v(\Phi_v;(s_\alpha))
= \sum_{A\subset \mathscr A_v\setminus\mathscr B_v}
     \big(q_v^{-1}\mu_v(\mathfrak o_v)\big)^{\dim X}
   \Card D_A^\circ (k_v)
   \prod_{\alpha\in A} \frac{q_v^{f_\alpha}-1}{q_v^{f_\alpha s_\alpha}-1}
\]
for any place $v\not\in S$.
Combined with the estimate of Theorem~\ref{prop.convergence},
this relation implies that 
\[  \mathscr I_v(\Phi_v;(s_\alpha))
  \prod_{\alpha\in\mathscr A\setminus \mathscr B}
\zeta_{F_\alpha,v}(s_\alpha)^{-1} = 1 + \mathrm O(q_v^{-1-\eps})
\]
provided $\Re(s_\alpha)>\frac12+\eps$ for each $\alpha\not\in\mathscr B$.
(See Prop.~9.5 in~\cite{chambert-loir-t2002} for a similar computation.)
This asymptotic expansion will imply the desired convergence of the infinite
product.

\begin{prop}\label{prop.igusa-adelic-merom}
Assume that $\Phi$ is a smooth function with compact support on $X(\AD_F)$.
Then the integral $\mathscr I(\Phi;(s_\alpha))$ converges
for $\Re(s_\alpha)>1$, for each $\alpha\not\in\mathscr B$,
and defines a holomorphic function in this domain. 
This function has a meromorphic continuation:
there is a holomorphic function~$\phi$ defined for
$\Re(s_\alpha)>\frac12$ if $\alpha\not\in\mathscr B$,
such that 
\[  \mathscr I(\Phi;(s_\alpha)) =\phi(s)\prod_{\alpha\not\in\mathscr B} \zeta_{F_\alpha}(s_\alpha). \]
Moreover,  if $s_\alpha=1$ for $\alpha\not\in\mathscr B$, then
\[ \phi(s) =  \prod_{\alpha\not\in\mathscr B}\zeta_{F_\alpha}^*(1)^{-1}
\int_{X(\AD_F)} \Phi(x)\prod_{\beta\in\mathscr B}\prod_v \norm{\mathsec f_{D_\beta}}_v^{s_\beta-1}\,\mathrm d\tau_X(x) .\]
\end{prop}
(Note that, in the last formula,
the function under the integration sign has compact
support on $X(\AD_F)$, while $\mathrm d\tau_X$ is a Radon
measure on that space.)
\begin{proof}
The first two parts of the theorem follow from the estimates
that we have just derived and the absolute convergence for $\Re(s)>1$
of the Euler product defining the Dedekind zeta function of
a number field.
 
Let $s\in\C^{\mathscr A}$ be such that $s_\alpha=1$ 
for $\alpha\not\in\mathscr B$. 
One has therefore
\[\mathscr I_v(\Phi_v;s) = \mathrm L_v(1,\EP (U))\int_{U(F_v)}
\prod_{\beta\in\mathscr B} \norm{\mathsec f_{D_\beta}}_v^{s_\beta-1}
\Phi_v(x)\,\mathrm d\tau_{X,v}(x). \]
Moreover, one has equalities of virtual representations of~$\Gamma_F$
(\cf Equation~\eqref{eq.virtual-eq}),
\begin{align*}
 \EP (U) & = \EP (\bar X) +\sum_{\alpha\in\mathscr A}\Ind_{\Gamma_{F_\alpha}}^{\Gamma_F}\mathbf 1\\
\EP (X) &= \EP (\bar X) +\sum_{\alpha\in\mathscr B}\Ind_{\Gamma_{F_\alpha}}^{\Gamma_F}\mathbf 1,
\end{align*}
from which it follows that
\[ \EP (X) = \EP (U) - \sum_{\alpha\not\in\mathscr B}\Ind_{\Gamma_{F_\alpha}}^{\Gamma_F}\mathbf 1. \]
In particular, for any finite place $v$ of~$F$,
\[ \mathrm L_v(1,\EP (X)) 
= \mathrm L_v(1,\EP (U)) \prod_{\alpha\not\in\mathscr B} \zeta_{F_\alpha,v}(1)^{-1} \]
and
\[ \mathrm L^*(1,\EP (X))
 = \mathrm L^*(1,\EP (U)) \prod_{\alpha\not\in\mathscr B}\zeta_{F_\alpha}^*(1)^{-1}. \]

If $S_\infty$ is the set of archimedean places of~$F$, then
(recall that $s_\alpha=1$ for $\alpha\not\in\mathscr B$)
\begin{align*}
 \phi(s) & =  \mathrm L^*(1,\EP (U))^{-1}
\prod_{v\in S_\infty}  \mathscr I_v(\Phi_v;s)
\prod_{v\not\in S_\infty} \left(\mathscr I_v(\Phi_v;s) 
\prod_{\alpha\not\in\mathscr B}\zeta_{F_\alpha,v}(1)^{-1} \right) \\
&= \mathrm L^*(1,\EP (U))^{-1} 
\prod_{v} \int_{X(F_v)} \Phi_v(x) \prod_{\beta\in\mathscr B}\norm{\mathsec f_{D_\beta}}_v^{s_\beta-1} \mathrm L_v(1,\EP(X)) \,\mathrm d\tau_{X,v}(x) \\
&=  \mathrm L^*(1,\EP (U))^{-1} 
\mathrm L^*(1,\EP (X))
\int_{X(\AD_F)} \Phi(x)\prod_{\beta\in\mathscr B}\prod_v
\norm{\mathsec f_{D_\beta}}_v^{s_\beta-1}\,\mathrm d\tau_X(x) \\
&= \prod_{\alpha\not\in\mathscr B}\zeta_{F_\alpha}^*(1)^{-1}
\int_{X(\AD_F)} \Phi(x)\prod_{\beta\in\mathscr B}\prod_v \norm{\mathsec f_{D_\beta}}_v^{s_\beta-1}\,\mathrm d\tau_X(x) .
   \end{align*}
\end{proof}

In fact, it is possible to establish a more general theorem.
Let $S$ be a finite set of places of~$F$ containing the archimedean places.
For any place $v\in S$, let $\Phi_v$ be a smooth bounded function on $X(F_v)$;
let also $\Phi^S$ be a smooth function with compact support 
on $X(\AD_F^S)$; let $\Phi$ be the function $\Phi^S\prod_{v\in S}\Phi_v$
on $X(\AD_F)$.

By the same arguments as above,
the integral $\mathscr I(\Phi;(s_\alpha))$ converges
provided $\Re(s_\alpha)>1$ for each $\alpha\not\in\mathscr B$
and
$\Re(s_\beta)>0$ for each $\beta\in\mathscr B$
and defines a holomorphic function in that domain. 

\begin{prop}
Let $\Omega\subset\C^{\mathscr A}$ be the set of $(s_\alpha)$
such that $\Re(s_\alpha)>\frac12$ for $\alpha\not\in\mathscr B$
and $\Re(s_\beta)>-\frac12$ for $\beta\in\mathscr B$.
The function $\mathscr I(\Phi;(s_\alpha))$
admits a meromorphic continuation  of the following form.
For any place $v\in S$ and any face $A$ of maximal
dimension of $\Clan_{F_v}(D)$, there is a holomorphic
function $\phi_A$ on~$\Omega$ such that
\[ \mathscr I(\Phi;(s_\alpha))=
\prod_{\alpha\not\in\mathscr B} \zeta_{F_\alpha}^S(s_\alpha) 
 \prod_{v\in S} \left( \sum_{A\in\Clanmax_{F_v}(D)} \phi_A(s)
        \prod_{\alpha\in A}\zeta_{F_\alpha,v}(s_\alpha)\right) . \]
Moreover, the functions $\phi_A$ have moderate growth in vertical strips
in the sense that for any compact subset~$K$ of $\R^{\mathscr A}\cap\Omega$,
there are real numbers~$c$ and~$\kappa$ such that
\[ \abs{\phi_A(s)} \leq c \prod_{\alpha\in\mathscr A}(1+\abs{s_\alpha})^\kappa,\]
for $s\in\C^{\mathscr A}$ such that $\Re(s)\in K$.
\end{prop}

\subsection{Volume asymptotics over the adeles}
\label{subsec.asymptotics.adelic}

In this Subsection we derive asymptotic estimates 
for volumes of height balls in adelic spaces,
similar to those we established above for height balls over local fields.

Let $F$ be a number field, 
$X$ a smooth projective variety over~$F$, purely of dimension~$n$.
Let $D$ be an effective divisor in~$X$ and  $\mathscr A$ its set of irreducible components; for $\alpha\in\mathscr A$, let $D_\alpha$ be the corresponding
component and $d_\alpha$ its multiplicity in~$D$. We have $D=\sum d_\alpha D_\alpha$.
For $\alpha\in\mathscr A$, let $\mathsec f_\alpha$ be the canonical
section of~$\mathscr O_X(D_\alpha)$; let $\mathsec f_D$
denote the canonical section of~$\mathscr O_X(D)$. We have
$\mathsec f_D=\prod \mathsec f_\alpha^{d_\alpha}$.

Let us endow these line bundles with adelic metrics, in such
a way that the isomorphism 
$\mathscr O_X(D)\simeq \bigotimes \mathscr O_X(D_\alpha)^{d_\alpha}$
is an isometry.

Let $U=X\setminus D$ ; let us endow $\omega_X$ and $\omega_X(D)$ 
with adelic metrics in such a way that the isomorphism
$\omega_X(D)\simeq \omega_X\otimes \mathscr O_X(D)$
is an isometry.
By the constructions  of Section~\ref{sec.general},
we obtain natural measures $\tau_X$, $\tau_{U}$
and $\tau_{(X,D)}$ on~$U(\AD_F)$ or $X(\AD_F)$, related
by the equalities
\[ \mathrm d\tau_{(X,D)}(x) = H_D(x) \mathrm d\tau_U(x), \]
where, for $x\in U(\AD_F)$, $H_D(x)=\prod_x \norm{\mathsec f_D(x)}_v^{-1}$.

Let $L$ be an effective divisor in~$X$ 
whose support is equal to the support of~$D$;
for $\alpha\in \mathscr A$, let $\lambda_\alpha$ be the multiplicity
of~$D_\alpha$ in~$L$, so that $L=\sum\lambda_\alpha D_\alpha$.
We endow the line bundle~$\mathscr O_X(L)$ 
with the natural adelic metric deduced from the
metrics of the line bundles~$\mathscr O_X(D_\alpha)$. 

The canonical section of~$\mathscr O_X(L)$ vanishes on~$L$, hence on~$D$.
Let $H_L$ denote the corresponding height function 
(denoted $H_{\mathscr O_X(L),\mathsec f_L}$ 
in Section~\ref{subsec.heights.adelic})
on the adelic space~$U(\AD_F)$; recall (Equation~\eqref{eq.height.f})
that it is defined by the formula
\[ H_L(\mathbf x) = \prod_{v\in\Val(F)} \norm{\mathsec f_L}(x_v)^{-1},
\qquad \text{for $\mathbf x = (x_v)_v$.} \]
By Lemma~\ref{lemm.northcott.adelic}, 
$H_L$ is bounded from below on~$U(\AD_F)$ and, for any real number~$B$,
the set of all $\mathbf x\in U(\AD_F)$ such that $H_L(\mathbf x)\leq B$
is compact in $(X\setminus L)(\AD_F)$.
Let $V(B)$ denote its volume with respect to the measure~$\tau_{(X,D)}$:
\begin{equation}
V(B) = \int_{\substack{U(\AD_F) \\ \{H_L(\mathbf x)\leq B\}}}
   \,\mathrm d\tau_{(X,D)}(\mathbf x);  
\end{equation}
it is a positive real number 
(for $B$ large enough)
and we want to understand its asymptotic behavior
when $B\rightarrow \infty$.
We are also interested in the asymptotic behavior of the probability
measures
\[ \frac1{V(B)}   \mathbf 1_{\{ H_L(\mathbf x)\leq B\}}
  \mathrm d\tau_{(X,D)}(\mathbf x).
\]
As in the case of local fields, we introduce a geometric Igusa
zeta function, namely
\begin{equation}
 Z(s) = \int_{U(\AD_F)} H_L(x)^{-s}\,\mathrm d\tau_{(X,D)}(x), \end{equation}
for any complex number~$s$ such that the integral converges
absolutely.
For any such~$s$, we thus have
\[ Z(s) = \int_{U(\AD_F)} \prod_v \norm{\mathsec f_\alpha(x_v)}^{s\lambda_\alpha-d\alpha} \, \mathrm d\tau_X(\mathbf x)=
 \mathscr I(\mathbf 1; (s\lambda_\alpha-d_\alpha+1))   \]
with the notation of Section~\ref{sec.adelic-igusa}.

Let $\sigma=\max_{\alpha\in\mathscr A}(d_\alpha/\lambda_\alpha)$
and let $\mathscr A_{L,D}$ be the set of $\alpha\in\mathscr A$
such that $d_\alpha=\lambda_\alpha\sigma$.

Then, the integral defining~$Z(s)$ converges for any 
complex number~$s$ such that $\Re(s)>\sigma$
(Proposition~\ref{prop.igusa-adelic-merom}).
Moreover, there exists a positive real number~$\delta$
such that function $s\mapsto Z(s)$ admits a meromorphic
continuation to the half-plane given by~$\Re(s)>\sigma-\delta$.
Namely, by Proposition~\ref{prop.igusa-adelic-merom} again,
there exists a holomorphic function~$\phi$ defined for $\Re(s)>\sigma-\delta$
such that
\[ Z(s) = \phi(s) \prod_{\alpha\in\mathscr A_{L,D}}
    \zeta_{F_\alpha}(s\lambda_\alpha-d_\alpha+1) \]
and 
\[ \phi(1) = \prod_{\alpha\in\mathscr A_{L,D}}
        \zeta_{F_\alpha}^*(1)
    \int_{X(\AD_F)} \prod_{\alpha\not\in\mathscr A_{L,D}} 
           H_{D_\alpha}(x)^{d_\alpha-\sigma\lambda_\alpha}
        \,\mathrm d\tau_X(x). \]

In particular, the function~$Z$ has a pole at $s=1$
of order~$\Card(\mathscr A_{L,D})$ and satisfies:
\[  \lim_{s\ra 1} (s-1)^{\Card(\mathscr A_{L,D})} Z(z)
     = \phi(1) \prod_{\alpha\in\mathscr A} \frac {\zeta_{F_\alpha}^*(1)}{\lambda_\alpha}
= \prod_{\alpha\in\mathscr A_{L,D}} \lambda_\alpha^{-1}
    \int_{X(\AD_F)} \prod_{\alpha\not\in\mathscr A_{L,D}} 
           H_{D_\alpha}(x)^{d_\alpha-\sigma\lambda_\alpha}
        \,\mathrm d\tau_X(x). \]
Let $E$ denote the $\Q$-divisor $\sigma L-D$.
We have
\[ E=\sigma L-D= \sum_{\alpha\in\mathscr A} (\sigma\lambda_\alpha-d_\alpha)D_\alpha = \sum_{\alpha\not\in\mathscr A_{L,D}} (\sigma\lambda_\alpha-d_\alpha)D_\alpha .\]
Consequently,
\begin{equation}\label{eqn.Z*(1)-adelic}
  \lim_{s\ra \sigma} (s-\sigma)^{\Card(\mathscr A_{L,D})} Z(\sigma)
= \prod_{\alpha\in\mathscr A_{L,D}} \lambda_\alpha^{-1}
 \int_{X(\AD_F)} H_E(x)^{-1}\,\mathrm d\tau_X(x)
       = \int_{X(\AD_F)} \,\mathrm d\tau_{(X,E)}(x).
\end{equation}
We summarize the results obtained in the following Proposition:
\begin{prop}\label{prop.zeta.global}
Let $\sigma$ and $\mathscr A_{L,D}$ be defined as above.
Then the integral defining $Z(s)$ converges for $\Re(s)>\sigma$
and defines a holomorphic function in that domain.
Moreover, there is a positive real number~$\delta$
such that $Z$ has a continuation to the half-plane $\Re(s)>\sigma-\delta$
as a meromorphic function with moderate growth in vertical strips, 
whose only pole is at $s=\sigma$, 
with order $b=\Card(\mathscr A_{L,D})$ 
and leading coefficient given by Equation~\eqref{eqn.Z*(1)-adelic}.
\end{prop}

Similarly to what we did in the local case,
using the Tauberian theorem~\ref{theo.tauber} and the abstract
equidistribution theorem (Proposition~\ref{prop.abstract.equi}),
we obtain the following result.

\begin{theo}\label{theo.volume.global}
There exists a monic polynomial~$P$ of degree~$b$ 
and a positive real number~$\delta$ such that, when
$B\ra\infty$, 
\[ V(B) = 
\frac{\prod_{\alpha\in\mathscr A_{L,D}}\lambda_\alpha^{-1}}
     {\sigma (b-1)!} 
B^\sigma P(\log B) 
\int_{X(\AD_F)} H_E(x)^{-1}\,\mathrm d\tau_X(x)
+ \mathrm O(B^{\sigma-\delta}). \]
Moreover, we have the tight convergence of probability measures 
\[ \frac1{V(B)} \mathbf 1_{\{H_L(x)\leq B\}} \mathrm d\tau_{(X,D)}(x)
 \ra \frac1{\int_{X(\AD_F)}H_E^{-1} \tau_X} H_E(x)^{-1} \,\mathrm d\tau_{X}(x). \]
\end{theo}

\begin{rema}
Let $S$ be a finite set of places of~$F$.
At least two variants of the preceding results
may be useful in $S$-integral contexts.

Let $\AD_F^S$ be the ring of adeles ``outside~$S$''.
For the first variant, we consider points of $U(\AD_F^S)$ 
of bounded height, and their volume with respect to the Tamagawa
measure $\tau_{(X,D)}^S$ in which the local factors in~$S$.
In that context, all infinite products of the set~$\Val(F)$
of places of~$F$ are replaced by the products over the set
$\Val(F)\setminus S$ of places which do not belong to~$S$.  
The modifications to be made to the statement and proof
of Theorem~\ref{theo.volume.global} are obvious
and lead to an asymptotic expansion for the
volumes of height balls in~$U(\AD_F^S)$.

A second variant is also possible which restricts
to the subset~$\Omega$ of~$U(\AD_F)$ consisting of
points $(x_v)_{v\in\Val(F)}$ which are ``integral'' 
at each place~$v\not\in S$. We do not need to be
specific about that condition; for all that matters
in our analysis, we understand that this subset
is relatively compact in~$U(\AD_F)$ and has non-empty
interior. A scheme-theoretic definition
would ask that we are given a projective and flat model of~$X$
over~$\mathfrak o_F$; then a point $(x_v)\in U(\AD_F)$
belongs to~$\Omega$ if and only if, for any finite 
place $v\not\in S$, the reduction
mod~$v$ of~$x_v$ does not belong to $D(\F_v)$.
An adelic definition would consider $\Omega$ to be
defined by a condition of the
form $\prod_{v\not\in S} \norm{\mathsec f_L}_v(x)\leq B_0$,
for some real number~$B_0$.

In that case, the analytic study of the adelic zeta
function involved is straightforward from what
has been done in Section~\ref{subsec.geometric-volume}.
Namely, it decomposes as an infinite  product over
all places~$v\in\Val(F)$ of $v$-adic zeta functions.
The subproduct corresponding
to places~$v\not\in S$ extends holomorphically
to the whole complex plane, while each factor attached
to a place~$v\in S$ is the source of zeroes and poles
described by the $v$-adic analytic Clemens complex
as in Prop.~\ref{prop.igusa-volume}.

Although the procedure should be quite clear to the reader, 
in any specific example,
a general statement would certainly be
too obfuscating to be of any help.
We would like the reader to remark than when all abscissae of convergence
at places in~$S$ are equal to a common real number~$\sigma$, 
the order of the pole at~$\sigma$ will be the
sum  of the orders~$b_v$,
and the leading coefficient the product of those
computed in Prop.~\ref{prop.igusa-volume}.
Moreover, if $S$ contains archimedean places,
the order of the pole at $\sigma$ will be
\emph{strictly} greater than the order of the other
possible poles on the line $\Re(s)=\sigma$. 
In that case, the Tauberian theorem~\ref{theo.tauber.S}
gives a simple asymptotic formula for the volume
of points of bounded height. 
With the notation of that Theorem, the $q_j$ are powers
of prime numbers and the non-Liouville property
of the quotients $\log q_j/\log q_{j'}$ follows
from Baker's theorem on linear forms in logarithms,
\cite[Theorem~3.1]{baker1975}.

We expect that these asymptotic expansions for volumes
can serve as a guide to understand the asymptotic number
of $S$-integral solutions of polynomial equations,
\emph{e.g.}, to a $S$-integral generalization of the circle method.
\end{rema}

\clearpage
\newcommand{\legendre}{\genfrac{(}{)}{}1}
\def\opp{{\mathrm {opp}}}
\section{Examples}
\label{sec.examples}

\subsection{Clemens complexes of toric varieties}

Let $T$ be an algebraic torus over a field~$F$.
Let $M$ and~$N$ denote the groups of characters and of cocharacters
of the torus~$T_{\bar F}$, endowed with the action of
the Galois group~$\Gamma=\Gal(\bar F/F)$
and the natural duality pairing $\langle\cdot,\cdot\rangle$.

Let $X$ be a smooth proper $F$-scheme 
which is an equivariant compactification of~$T$.
By the theory of toric varieties, $X$ corresponds
to a fan~$\Sigma$ in the space~$N_\R$ which is invariant
under the action of~$\Gamma$. By definition, $\Sigma$
is a set of convex polyhedral rational cones  in~$N_\R$
satisfying the following properties:
\begin{itemize}
\item for any cones $\sigma,\sigma'$ in~$\Sigma$, their intersection
is a face of both~$\sigma$ and~$\sigma'$;
\item all faces of each cone in~$\Sigma$ belong to~$\Sigma$;
\item for any $\gamma\in\Gamma_F$ and any cone~$\sigma\in\Sigma$,
$\gamma(\sigma)\in\Sigma$.
\end{itemize}

Assume for the moment that~$T$
is split, meaning that $\Gamma$ acts trivially on~$M$.
To each cone~$\sigma$ of~$\Sigma$ corresponds 
a $T$-stable affine open subset $X_\sigma=\Spec F[\sigma^\vee \cap M]$
in~$X$,
where $\sigma^\vee$ is the cone in~$M_\R$ dual to~$\sigma$.
The variety~$X$ is glued from these affine charts~$X_\sigma$,
along the natural open immersions $X_\sigma\ra X_\tau$, for any two cones
$\sigma$ and~$\tau$ in~$\Sigma$ such that $\tau\supset \sigma$.
The zero-dimensional cone~$\{0\}$ corresponds to~$T$;
in particular, each $X_\sigma$ contains~$T$ and carries  a
compatible action of~$T$.

Let $t$ denote the dimension of~$T$ and
let $\sigma$ be a cone of maximal dimension of~$\Sigma$.
By the theory of toric varieties, the smoothness assumption
on~$X$ implies that there exists a basis
of~$N$ which generates~$\sigma$ as a cone. 
Consequently, the torus embedding $(T,X_\sigma)$
is isomorphic to $(\gm^t,\Aff^t)$.
The irreducible components of $X_\sigma\subset T$ then
correspond to the hyperplane coordinates in~$\Aff^t$. 
In particular, $X_\sigma\setminus T$ is a divisor
with strict normal crossings, 
all of whose components of $X_\sigma\setminus T$
meet at the origin.
This shows that a set of components of~$X\setminus T$
has a non-trivial intersection whenever the corresponding
rays belong to a common cone in~$\Sigma$.
The converse holds since each point of~$X$ admits a $T$-invariant
affine neighborhood, hence of the form~$X_\sigma$.

In particular, what precedes holds over~$\bar F$
and implies the  following description
of the geometric Clemens complex $\Cl(X,D)$ of~$(X,D)$.
There is a bijection between the set of orbits 
of~$T_{\bar F}$ in~$X_{\bar F}$ and the set~$\Sigma$.
The irreducible components of~$D_{\bar F}$ are in bijection
with the set of one-dimensional cones (\emph{rays}) of~$\Sigma$
and a set of components has  a non-trivial intersection
if and only if the corresponding rays
belong to a common cone in~$\Sigma$.
The cone~$\{0\}$ belongs to~$\Sigma$ and corresponds
to the open orbit~$T$.
Consequently, the Clemens complex $\Cl(X,D)$ is equal
to the set~$\Sigma\setminus\{\{0\}\}$, partially ordered  by inclusion,
with the obvious action of the Galois group~$\Gamma_F$.

If $T$ is split, then $\Gamma_F$ acts trivially
on~$\Cl(X,D)$, hence $\Cl_F(X,D)=\Cl(X,D)$.
We have also observed in that case that the various strata of~$X$
even had $F$-rational points.
In other words, the $F$-analytic Clemens complex~$\Clan_F(X,D)$
is also equal to $\Cl(X,D)$.

Let us now treat the general case by proving
that $\Clan_F(X,D)=\Cl_F(X,D)$.   By definition,
the $F$-rational Clemens complex
$\Cl_F(X,D)$ is the subcomplex of~$\Cl(X,D)$
consisting of $\Gamma_F$-invariant faces;
in other words, it corresponds to $\Gamma_F$-invariant
cones of positive dimension. We need to prove
that for any such cone~$\sigma$, the closure of
the corresponding orbit~$\mathscr O_\sigma$ in~$X$, 
which is defined over~$F$, has an $F$-rational point.
For each ray~$r$ of~$\sigma$, let $n_r$ be the generator of $N\cap r$;
the group $\Gamma_F$ acts on the set of these~$n_r$,
and their sum~$n$ is fixed by $\Gamma_F$.
It corresponds to a cocharacter 
$c_n\colon \gm\ra T$ whose limit at~$0$ is an $F$-rational point
and this point belongs to~$\overline{\mathscr O_\sigma}$.
In fact, it even belongs to~$\mathscr O_\sigma$ because
the point~$n$ belongs to the relative interior of the cone~$\sigma$.

\begin{lemm}
Let $T_0$ be the maximal $F$-split torus in~$T$;
the group~$N_0$ of cocharacters of~$T_0$ is equal
to the subspace of~$N$ fixed by~$\Gamma_F$.
Let $X_0$ be the Zariski closure of~$T_0$ in~$X$.
It is a smooth equivariant compactification of~$T_0$
and its fan~$\Sigma_0$ in~$(N_0)_\R$
has as cones the intersections~$\sigma\cap (N_0)_\R$,
for all cones~$\sigma\in\Sigma^{\Gamma_F}$ which are invariant under~$\Gamma_F$.
\end{lemm}
\begin{proof}
It suffices to prove the desired result in each affine chart of~$X$. 
Consider a cone~$\sigma\in\Sigma$ and
let $X_\sigma=\Spec \bar F[M\cap\sigma^\vee]$ be the corresponding
affine open subset of~$X_{\bar F}$. 
A vector in~$M_0\cap \sigma$, being fixed under the action of~$\Gamma_F$,
belongs to all of the cones~$\gamma(\sigma)$, for $\gamma\in\Gamma_F$,
hence belongs to their intersection~$\tau$ which, by the definition
of a fan, is a cone in~$\Sigma$, obviously $\Gamma_F$-invariant.
As a consequence, the closure of~$T_0$ in~$X_\sigma$ is contained
in the  toric open subvariety~$X_\tau$. 

We thus can assume that $\sigma$ is a $\Gamma_F$-invariant cone,
and we may moreover assume that it is of maximal dimension.
Since $X$ is smooth, there is a basis $(e_1,\dots,e_d)$ of~$N$
and an integer~$s\in\{1,\dots,d\}$ such that $\sigma$
is generated by~$S=\{e_1,\dots,e_s\}$. We can also
assume that $e_{s+1},\dots,e_r$ belong to~$M_0$
and generate  a complement to 
the (saturated) subgroup generated by~$M_0\cap\sigma$.
%
Over~$\bar F$, this identifies
$T$ with~$\gm^d$ and $X_\sigma$ with $\Aff^s\times\gm^{d-s}$. 
Moreover, $\Gamma_F$ acts by permutations on~$S$.  This implies that
$M_0\cap\sigma$ is generated by the vectors $\sum_{i\in O} e_i$, 
where $O$ runs over the set $S/\Gamma_F$ of $\Gamma_F$-orbits in~$S$.
Consequently, $T_0\cap X_\sigma$ is the set of elements
$(x_1,\dots,x_d)$ in~$\gm^d$ such that $x_i$ is constant on
each orbit~$O$, and $x_i=1$ for $i>s$. 
Its closure  is the set of such elements
$(x_1,\dots,x_d)\in\Aff^r\times\gm^{d-r}$ 
satisfying the same relations. This shows
that the closure of~$T_0$ in~$X_\sigma$ is 
a isomorphic to the toric embedding of $\gm^{\dim M_0}$
in $\Aff^{\Card (S/\Gamma_F)}\times\gm^{\dim M_0-\Card (S/\Gamma_F)}$.
\end{proof}

\begin{coro}
By associating to a $T_0$-orbit in~$X_0$ the corresponding
$T$-orbit in~$X$, we obtain a bijection between the Clemens
complex of~$(X_0,X_0\setminus T_0)$ and the $F$-analytic
Clemens complex of~$(X,X\setminus T)$.
\end{coro}

\subsection{Clemens complexes of wonderful compactifications}

Let $G$ be a connected reductive group over a field~$F$. 
Let $T_0$ denote a maximal split torus of~$G$ and
$T$ a maximal torus of~$G$ containing~$T_0$.

Let $X$ be a smooth proper~$F$-scheme which is 
a biequivariant compactification of~$G$.
We denote by $Y$ and~$Y_0$ the closures
of the tori~$T$ and~$T_0$ in~$X$; these are toric varieties
defined over~$F$.
We assume that over a separable closure~$\bar F$ of~$F$,
the divisor $D=X\setminus G$
has strict normal crossings.
Observe that this assumption is satisfied if $X$ is 
a wonderful compactification  of~$G$.

\begin{prop}\label{prop.cartan}
With this notation,
one has: $X(F)=G(F)Y(F)G(F)$ and $Y(F)=T(F)  Y_0(F)$;
in particular, $X(F)=G(F)Y_0(F)G(F)$.
%
\end{prop}
\begin{proof}
Let us first prove that $X(F)=G(F)Y(F)G(F)$.
It suffices to prove that $X(F)$ is contained in the latter set.
We use the arguments of~\cite{brion-kumar2005}, Lemma~6.1.4.
Let $x_0\in X(F)$;  since $X$ is assumed to be smooth,
we can  choose an arc $x\in X(F\lbra t\rbra )$ such that $x(0)=x_0$ and $x$
is not contained in $X\setminus G$.
We interpret $x$ as a $F\lpar t\rpar$-point of~$G$;
by the Cartan decomposition
(\cite{bruhat-tits1972}, Proposition~4.4.3. 
The fact that the group $G (F\lbra t\rbra ) $ is good follows from the
property that the $F\lpar t\rpar$-algebraic group $G_{F\lpar t\rpar}$ 
is split over an unramified extension, namely $\bar F\lpar t\rpar$,
and descent properties of the building)
\[ G( F\lpar t\rpar ) = G( F\lbra t\rbra ) T (F\lpar t\rpar) G (F\lbra t \rbra),
\]
we may write $x = g_1 y g_2$, where $g_1,g_2\in G(F\lbra t\rbra)$
and $y\in T(F\lpar t\rpar)$. Writing $y=g_1^{-1} x g_2^{-1}$,
we see that $y\in Y(F\lbra t\rbra)$.
Specializing at $t=0$, we now obtain $x_0=g_1(0) y(0) g_2(0)$,
as wanted.

Let us now prove the second equality, namely $Y(F)=T(F)Y_0(F)$.
Again, it suffices to prove that $Y(F)$ is contained in~$T(F)Y_0(F)$.
Let $x_0\in Y(F)$; as above, there
is an arc $x\in X(F\lbra t\rbra )\cap G( F\lpar t\rpar)$ 
such that $x(0)=x_0$. Using the Cartan decomposition,
we may replace $x$ by an arc of the form $y=g_1^{-1} x g_2$,
with $g_1,g_2\in G( F\lbra t\rbra)$ satisfying
$y\in X(F\lbra t\rbra) \cap T(F \lpar t\rpar)$. 
Moreover, we may also assume that $g_1(0)=g_2(0)$ is
the neutral element of~$G(F)$. In particular, 
$y\in Y(F\lbra t\rbra)$ and $y(0)=x_0$.

Let $S$ denote the anisotropic torus $T/T_0$
and $\pi\colon T\ra S$ the quotient map.
One has $\pi(y)\in S(F\lpar t\rpar)$.
Since $S$ is anisotropic,
one has $S(F\lbra t\rbra )=S(F\lpar t\rpar )$
(Lemma~\ref{lemm.tore-anisotrope} below), hence $\pi(y)\in S(F\lbra t\rbra )$.
Looking at the exact sequence of tori 
$$
1\ra T_0\ra T\ra S\ra 1
$$
and applying Hilbert's theorem~90 over the discrete valuation
ring $F\lbra t\rbra $, it follows that there exists $z\in T(F\lbra t\rbra )$
such that $\pi(z)=\pi(y)$, hence $z^{-1}y\in T_0(F\lpar t\rpar )$.
Moreover, $z^{-1}y\in Y(F\lbra t\rbra )$.
Specializing~$t$ to~$0$, we see that $z(0)^{-1} y(0)\in Y_0(F)$,
that is $y(0)=x_0\in T(F) Y_0(F)$, as claimed.

The last equality follows immediately.
\end{proof}

The following lemma is well-known; we include a proof for the convenience
of the reader.
\begin{lemm}\label{lemm.tore-anisotrope}
Let $S$ be an anisotropic torus over a field~$F$.
Then, $S(F\lpar t\rpar )=S(F\lbra t\rbra )$.
\end{lemm}
\begin{proof}
Let $F'$ be a finite extension of~$F$ which splits~$S$; then,
for any $F$-algebra~$E$, $S(E)$ is the set
of morphisms from~$\mathfrak X^*(S)$ to~$E\otimes_F F'$
which commute with the actions of~$\Gamma_{F'/F}$ on both
sides. It follows that
a  point~$P$ in $S(F\lpar t\rpar )$ corresponds to  a
$\Gamma_F$-equivariant group morphisms~$\phi$
from $\mathfrak X^*(S)$ to~$F'\lpar t\rpar ^*$.
 
By composing~$\phi$ with the order map $F'\lpar t\rpar ^*\ra \Z$,
we obtain a $\Gamma_{F'/F}$-invariant morphism
$\ord\circ \phi\colon \mathfrak X^*(S)\ra\Z$,
which is necessarily~$0$ since $S$ is anisotropic.
Consequently, $\phi(c)=0$ for any $c\in\mathfrak X^*(S)$,
which means that $\phi(c)\in F'\lbra t\rbra ^*$.
In other words, $\phi$ is a $\Gamma_{F'/F}$-equivariant
morphism from $\mathfrak X^*(S)$ to~$F'\lbra t\rbra ^*$,
and the corresponding point~$P$ belongs
to $S(F'\lbra t\rbra )$.
\end{proof}

Our goal now is to describe the various Clemens complexes 
attached to the pair~$(X,D)$.
Let $W=\mathrm N_G(T)/T$ and $W_0 =\mathrm N_G(T_0)/\mathrm Z_G(T_0)$
be the Weyl groups of~$G$ relative to the tori~$T$ and~$T_0$.

Since the Weyl group~$W$ acts on~$G$ via the conjugation by an element
of~$G$, this extends to an action on~$X$. This action 
induces the trivial action on~$\Cl(X,D)$.
Indeed, the group~$G$ being connected, each irreducible
component of $D_{\bar F}$ is preserved by
the actions of~$G_{\bar F}$.

\begin{prop} Over~$\bar F$, the open strata of the stratification of~$X_{\bar F}$ deduced from~$D_{\bar F}$ are exactly 
the orbits of~$(G\times G)_{\bar F}$ in~$X_{\bar F}$.
Moreover, associating to an orbit its closure
defines a $\Gamma_F$-equivariant bijection from
$G(\bar F)\backslash X(\bar F)/G(\bar F)$ to~$\Cl(X,D)$.
\end{prop}
\begin{proof}
We may assume that the field~$F$ is separably closed.
We then have $T=T_0$.

By Proposition~\ref{prop.cartan}, the map
\[ Y(F)/T(F) \ra G(F)\backslash X(F)/ G(F) \]
which associates to the $T(F)$-orbit of a point~$x$ in~$Y(F)$
the orbit~$G(F)xG(F)$ in~$X(F)$ 
is surjective. By the theory of toric varieties,
$T(F)$ has only finitely many orbits in~$Y(F)$.
Consequently, $G\times G^\opp$ has only finitely many orbits in~$X$. 

Let $Z$ be an element of~$\Cl(X,D)$, viewed as
an irreducible closed subvariety of~$D$. It is smooth, and possesses
a $G\times G^\opp$-action.
The group $G\times G^\opp$ acts on~$Z$ with only finitely
many orbits; necessarily, one of these orbits, say~$Z_0$,
is open in~$Z$. Since $Z$ is irreducible, it follows
that  $Z=\overline{ Z_0}$,
and $Z$ is the closure of a $G\times G^\opp$-orbit.
By Theorem~2.1 of~\cite{knop1991}, there exists
an open affine subset~$X_0$ of~$X$
such that $Z_0$ is the unique closed orbit
of~$G\times G^\opp$ in~$GX_0G$.

Let us show that $Z\setminus Z_0$ is contained in $X\setminus GX_0G$
or, equivalently, that $Z\cap GX_0G=Z_0$. 
Indeed, let $O$ be any orbit in~$Z$
distinct from~$Z_0$; we need to prove that $O\cap GX_0G=\emptyset$.
Let us denote by~$\overline O$ the closure of~$O$ in~$X$;
since $Z_0$ is open in~$Z$, $\overline O$ has empty interior in~$Z$,
hence is an irreducible subset of~$X$ whose dimension satisfies
$\dim \overline O < \dim Z$. However, any action of an algebraic
group on an algebraic variety has a closed orbit; 
in particular $\overline O\cap GX_0G$ contains a closed orbit,
which implies that $\overline O\cap GX_0G\supset Z_0$,
and this contradicts the assumption $\dim\overline O<\dim Z_0$.

Since $X_0$ is open in~$X$, it contains at least one point of~$G$,
hence $GX_0G$ contains~$G$. Consequently, $X\setminus GX_0G$
is contained in~$X\setminus G$.
Since $GX_0G$ is an $G\times G^\opp$-invariant open subset of~$X$,
each irreducible component of~$Z\setminus Z_0$ is contained
in some irreducible component of~$X\setminus GX_0G$ which does not meet~$Z_0$.

Consequently, the orbit~$Z_0$ is an open stratum of the 
stratification defined by the boundary divisor~$D$.

Conversely, let $O$ be an orbit of $G\times G^\opp$ in~$X$
and let $Z$ be the stratum of minimal dimension in~$\Cl(X,D)$
which contains~$O$.
By the preceding argument, $Z$ is the closure of
an orbit~$Z_0$ whose complement $Z\setminus Z_0$
is a union of lower dimensional strata.
By the minimality assumption on~$Z$, $Z_0=O$,
which proves that $O$ is an open stratum.

The rest of the proposition follows at once.
\end{proof}



The group of characters~$M_0$ of~$T_0$ is the group of coinvariants
of~$\Gamma$ in~$M$ (it is a quotient of~$M$),
and its group of cocharacters~$N_0$ is the group of invariants
of~$\Gamma$ in~$N$ (a subgroup of~$N$).
Let $t_0=\dim T_0$.
Let $\Sigma_0$ denote the fan of~$N_0$ induced by~$\Sigma$;
it is simplicial albeit not obviously smooth in general.
This fan defines a toric variety~$X_0$ which in fact is isomorphic 
to the Zariski closure of~$T_0$ in~$X$.

\subsection{Volume estimates for compactifications of semi-simple groups}

\label{subsec.wonderful}

\subsubsection{Introduction}
Let $G$ be a semisimple algebraic group of adjoint type
over a field~$F$ of characteristic~$0$.
Let $\iota\colon G\ra \GL(V)$ be a
faithful algebraic representation of~$G$ in a
finite dimensional~$F$-vector space~$V$. 
The natural map 
\[
\GL(V)\ra\End(V)\setminus\{0\}\ra\P\End(V)
\]
induces a map $\bar\iota\colon G\ra\P\End(V)$. Let $X_\iota$ be
the Zariski closure of its image;
it is a bi-equivariant compactification
of~$G$. Let $\partial X_\iota$ be the complement to~$G$
in~$X_\iota$.

When $\iota$ is irreducible 
with regular highest weight, 
$X_\iota$ is the wonderful compactification defined 
by De Concini and Procesi in~\cite[3.4]{deconcini-p83}.
In that case, $X_\iota$ is smooth and $\partial X_\iota$
is a divisor with strict normal crossings.

When $F$ is a local field, we endow 
the vector space~$\End(V)$ with a norm~$\norm{\cdot}$.
When $F$ is a number field, let us choose, for any place~$v$ of~$F$,
a $v$-adic norm on~$\End(V)\otimes F_v$, so that 
there is an $\mathfrak o_F$-lattice in~$\End(V)$
inducing these norms for almost all finite places.
As it was explained in Sections~\ref{exam.projective-space.metric}
and~\ref{exam.projective-space.adelic.metric},
such choices give rise to a metric on the line bundle~$\mathscr O(1)$
on the projective space~$\P\End(V)$, resp. an adelic metric,
when $F$ is a number field.

We claim that the line bundle~$\mathscr O(1)$ on~$X_\iota$
has a nonzero global section~$s_\iota$ which is invariant under~$G$. 
Such a section~$s_\iota$ 
is unique up to multiplication by a scalar,
since the quotient of two of them is a $G$-invariant rational function
on~$X_\iota$.  (See also~\cite{deconcini-p83}, 1.7, p.~9, proposition.)

To prove the existence, let us first observe that 
the line in~$\End(V)$ generated by $\id_V$ is $G$-invariant.
By semi-simplicity of~$G$, there exists a linear form~$\ell_\iota$
on~$\End(V)$ which is invariant under~$G$ and maps~$\id_V$ to~$1$.
%
Let $s_\iota$ be the restriction
to~$G$ of the global section of~$\mathscr O(1)$ over~$\P\End(V)$
defined by~$\ell_\iota$.
Now, we have
\[ \norm{s_\iota}(\bar\iota(g))=\frac{\abs{\ell_\iota(g)}}{\norm{\iota(g)}} 
 = \frac{\abs{\ell_\iota^g(e)}}{\norm{\iota(g))}}  = {\norm{\iota(g)}}^{-1}, \]
for any $g\in G(F)$ when $F$ is a local field,
and similar equalities at all places of~$F$
in the number field case.

Consequently, for $F=\R$ or~$\C$,
the results of Section~\ref{subsec.geometric-volume}
imply, as a particular case, an asymptotic formula for the volume of
sets of $g\in G(F)$ with $\norm{\iota(g)}\leq B$, when $B\ra\infty$,
as well as similar (but weaker) estimates when $F$ is a $p$-adic field.
Similarly, when $F$ is a number field, 
the volume estimates established in 
Section~\ref{subsec.asymptotics.adelic} imply 
estimates for the volume of adelic sets in~$G(\AD_F)$
consisting of adelic points of bounded height.

In the case of local fields,
it requires further computations for making these estimates explicit, in terms of the representation~$\iota$.
In particular, we will need to describe the analytic Clemens complex 
of~$\partial X_\iota$

\subsubsection{The wonderful compactification of De Concini and Procesi}
For simplicity, we assume for the moment that $G$ is split.
We fix a maximal torus $T$ of~$G$ which is split,
as well as a Borel subgroup~$B$ of~$G$ containing~$T$.
We identify the groups of characters $\mathfrak X^*(T)$ and~$\mathfrak X^*(B)$,
as well as the groups of cocharacters $\mathfrak X_*(T)$ and~$\mathfrak X_*(B)$;
we also let $\mathfrak a=\mathfrak X_*(T)\otimes_\Z\R$ and 
$\mathfrak a^*=\mathfrak X^*(T)\otimes_\Z\R$.

Let~$\Phi$, resp. $\Phi^+$, be the set of roots of~$G$, resp. of positive roots
in~$\mathfrak X^*(T)$; let $\beta$ denote the sum of all positive roots.
Let $\Delta\subset\Phi^+$ be the set of simple roots; they form a basis
of the real cone in~$\mathfrak a^*$ 
generated by positive roots, 
hence we may write $\beta=\sum_{\alpha\in\Delta}m_\alpha\alpha$
for some positive integers~$m_\alpha$.

Let $\iota\colon G\ra\GL(V)$ be a representation as above;
we assume here that $\iota$ is irreducible 
and that its highest weight~$\lambda$ is regular,
\emph{i.e.}, can be written as $\lambda=\sum_{\alpha\in\Delta} d_\alpha \alpha$,
for some positive integers~$d_\alpha$. 
In that case, as recalled above, $X_\iota$
is the ``wonderful compactification of~$G$'' defined by De Concini and Procesi.
The variety~$X_\iota$ does not depend on the actual choice of~$\iota$, 
but the projective embedding does.

The irreducible components of $X_\iota\setminus G$
are naturally indexed by the set~$\Delta$; we will write $D_\alpha$
for the divisor corresponding to a simple root~$\alpha$.
Let $D=\sum D_\alpha$.
The Clemens complex~$\Cl(\partial X_\iota)$ is simplicial,
and coincides with the $F$-analytic Clemens complex~$\Clan_F(\partial X_\iota)$
since $G$ is assumed to be split.

The line bundles~$\mathscr O(D_\alpha)$ form a basis of~$\Pic(X_\iota)$,
as well as generators of the cone of effective divisors 
(which is therefore simplicial).
The restriction of~$\mathscr O_{\P}(1)$ to~$X$ 
corresponds precisely to~$\lambda$,
and there is a canonical isomorphism
\[ \mathscr O_{\P\End(V)} (1)|_{X_\iota} \simeq \mathscr O(\sum_{\alpha\in\Delta} d_\alpha D_\alpha). \]
According to~\cite{deconcini-p83}, the anticanonical line bundle of~$X_\iota$
is given by
\[ K_{X_\iota}^{-1} \simeq \mathscr O(\sum_{\alpha\in\Delta} (m_\alpha+1) D_\alpha).\]

Let $\mathscr C\subset\mathfrak a^*$ be the convex hull of the characters
of~$T$ appearing in the representation~$\iota$. This
is also the convex hull of the images of the highest weight~$\lambda$
under the action of the Weyl group. This is a convex
and compact polytope which contains~$0$ in its interior
by~\cite{maucourant2007}, Lemma 2.1.

\begin{lemm}
Let $\sigma=\max_{\alpha\in\Delta} (m_\alpha/d_\alpha)$, 
let $t$ be the number of elements $\alpha\in\Delta$ 
where equality holds.
Then $\sigma$ is the smallest positive real number such
that $\beta/\sigma \in\mathscr C$ and
$t$ is the maximal codimension of a face of~$\mathscr C$
containing~$\beta/\sigma$.
\end{lemm}
\begin{proof}
For any simple root~$\alpha$, let~$\alpha^\vee$ be the corresponding
coroot, so that the simple reflexion~$s_\alpha$ associated to~$\alpha$
is given by $s_\alpha(x)=x-2\langle \alpha^\vee,x\rangle \alpha$,
for $x\in\mathfrak a^*$. We have $\langle\alpha^\vee,\alpha\rangle=1$
while $\langle\alpha^\vee,\alpha'\rangle=0$ for any other
simple root~$\alpha'$.

Within the Weyl chamber of~$\lambda$, the polytope~$\mathscr C$
is bounded by the affine hyperplanes orthogonal to the
simple roots~$\alpha$ and passing through~$\lambda$.
Consequently,   a point~$x$ in this chamber
belongs to~$\mathscr C$
if and only $\langle \alpha^\vee,x\rangle\leq\langle\alpha^\vee,\lambda\rangle$.
Moreover, such a point~$x$ belongs to the boundary of~$\mathscr C$
if and only if equality is achieved for some simple root~$\alpha$;
then, the number of such~$\alpha$ is 
the maximal codimension  of a face of~$\mathscr C$ containing~$x$.
For $x=\beta/\sigma=\sum_{\alpha\in\Delta} (m_\alpha/\sigma) \alpha$,
we find 
\[ \langle\alpha^\vee,x-\lambda\rangle=\frac{m_\alpha}\sigma-\lambda_\alpha,\]
hence the Lemma.
\end{proof}

Consequently, our geometric estimates (Theorem~\ref{theo.geometric-volume})
imply the following result of Maucourant~\cite{maucourant2007}
under the assumption that~$G$ is split and~$V$ has  a unique
highest weight.
\begin{coro}\label{coro.maucourant}
Define $\sigma=\max_{\alpha\in\Delta} (m_\alpha/d_\alpha)$
and let $t$ be 
be the number of~$\alpha\in\Delta$ 
where the equality holds.

Assume that $F=\R$ or~$\C$.
When $B\ra\infty$, the volume $V(B)$ of all $g\in G(F)$ 
such that $\norm{\iota(g)}\leq B$ satisfies an asymptotic formula
of the form:
\[ V(B)\sim c B^{\sigma}(\log B)^{t-1}. \]
\end{coro}
According to our theorem, 
the positive constant~$c$ can be written as a product of
the combinatorial factor
\[ \frac1{\sigma(t-1)!} \prod_{\alpha\in A}\frac{1}{\lambda_\alpha}
\]
and an explicit integral on the stratum of~$X_\iota(F)$ defined by~$A$,
with respect to its normalized residue measure~$\tau_{D_A}$.

In the $p$-adic case, Corollary~\ref{coro.asymptotic.volume-ultrametric} similarly implies
the following result:
\begin{coro}\label{coro.maucourant.ultrametric}
Keep the same notation, assuming that $F$ is a $p$-adic field.
Then, when $B\ra\infty$,
\[ 0<\liminf \frac{V(B)}{ B^{\sigma}(\log B)^{t-1}} \leq 
\limsup \frac{V(B)}{ B^{\sigma}(\log B)^{t-1}} <+ \infty. \]
\end{coro}
\subsubsection{The case of a general representations (split group)}
We now explain how to treat non-irreducible representations,
still assuming that the group~$G$ is split.

By Proposition~6.2.5 of~\cite{brion-kumar2005},
there is a diagram of equivariant compactifications
\[ \xymatrix{   & {\tilde X} \ar[ld]_{\pi} \ar[rd] \\
    X_\iota && X_w } \]
where $X_w$ is the  wonderful compactification previously studied
and $\tilde X$ is smooth. Since the boundary divisor of a
smooth toric variety is  a divisor with strict normal crossings,
it then follows from~\cite{brion-kumar2005}, Proposition~6.2.3,
that the boundary~$\partial\tilde X$ of~$\tilde X$    is a divisor with
strict normal crossings in~$\tilde X$.
Moreover, the proof of this proposition and the local
description of toroidal $G$-embeddings in \emph{loc. cit.}, \S6.2,
show that $\tilde X$ is obtained from~$X_w$ by a sequence
of $T$-equivariant blow-ups.
Therefore, the boundary~$\partial\tilde X$
consists of the strict transforms~$\tilde D_\alpha$ of the divisors~$D_\alpha$,
indexed by the simple roots, and of the exceptional divisors~$E_i$
(for $i$ in some finite index set~$I$). These divisors form
a basis of the effective cone in the Picard group.
The anticanonical line bundle of~$\tilde X$ decomposes   as a sum
\[ \sum_{\alpha\in\mathscr A} (m_\alpha+1) D_\alpha + \sum_{i\in I} E_i.\]

Let $\tilde L$ be the line bundle $\pi^*\mathscr O(1)$ on~$\tilde X$;
let us write it as $\tilde L=\sum \tilde\lambda_\alpha D_\alpha + \sum\tilde\lambda_i E_i$
in the above basis.
Since we can compute volumes on~$\tilde X$, our estimates
imply that $V(B)$ has an asymptotic expansion
of the form given in Corollary~\ref{coro.maucourant}, 
$\sigma$ being given by
\[ \sigma= \max \left( \max_{\alpha\in\mathscr A} \frac{m_\alpha}{\tilde\lambda_\alpha},
 \max_{i\in I} \frac{0}{\tilde\lambda_i} \right) 
 = \max_{\alpha\in\mathscr A} \frac{m_\alpha}{\tilde\lambda_\alpha},
\]
and $t$ is again the number of indices~$\alpha\in\mathscr A$
where equality holds.

Observe that the definition of~$\sigma$
precisely means that the line bundle 
$\sigma\pi^* L-(K_{\tilde X}+\partial \tilde X)$ belongs to the boundary 
of the effective cone of~$\tilde X$; the integer~$t$
is then the codimension of the face of minimal dimension
containing that line bundle.
These properties can be checked on restriction to the toric
variety~$\tilde Y$ given by the closure of~$T$ in~$\tilde X$;
indeed the restriction map $\Pic(\tilde X)\ra\Pic(\tilde Y)$
is an isomorphism onto the part of~$\Pic(\tilde Y)$ invariant
under the Weyl group,  and similarly for the effective cone.
To determine~$\sigma$, 
it now suffices to test for the positivity of  the piecewise linear
(\pl) function on the vector space~$\mathfrak a^*$ associated to this
divisor.

The above formula for the anticanonical line bundle
of~$\tilde X$ implies that the \pl function 
for~$K_{\tilde X}+\partial\tilde X$
is that of $K_{X_w}+\partial {X_w}$.

On the other hand, the theory of heights on toric varieties
relates the \pl function corresponding
to a divisor to the  normalized local height function
(see~\cite{batyrev-t95b}).
The formula is as follows. Let $D$ be an effective $T$-invariant divisor,
the corresponding $T$-linearized line bundle has a canonical 
$T$-invariant nonzero global section~$s_D$. For $t\in T(F)$,
let $\ell(t)$ be the linear form on~$\mathfrak a$
defined by $\chi\mapsto \log\abs{\chi(t)}$; dually, one
has $\ell(t)(\chi)=\langle\chi,\log(t)\rangle$.
Then, the norm of~$s_D$ with respect to the canonical normalization 
is given by
\[ \log\norm{s_D(t)}^{-1} = \phi_D(\ell(t)). \]
If the local height function is not normalized, the previous
equality only holds up to the addition of a bounded term.

In our case, let $\Phi$ 
be the set of weights of~$T$ in the representation~$\iota$. Then,
for $t\in T(F)$, 
\[ \norm{\iota(t)} \approx \max_{\chi\in\Phi} \abs{\chi(t)}, \]
hence
\[ \log\norm{s_\iota(t)}^{-1} = \log \norm{\iota(t)}
 = \max_{\chi\in\Phi} \log\abs{\chi(t)} + \mathrm O(1). \]
In other words, the line bundle~$\pi^*L$
corresponds to the \pl function 
$\phi_\iota = \max_{\chi\in\Phi} \langle\chi, \cdot\rangle$
on~$\mathfrak a^*$.

As explained above, $\sigma$ is the least positive real
number such that
the \pl function $ s\phi_\iota $ is greater than
the \pl function~$\phi_\beta$ associated to~$K_{X_w}^{-1}(-\partial X)$.
Since the Weyl group acts trivially on the Picard group
of~$\tilde X$, these \pl functions are invariant under the action
of the Weyl group and it is sufficient to test the inequality 
on the positive Weyl chamber~$C$ in~$\mathfrak a$.

Let $(\varpi_\alpha)$ denote the basis of~$\mathfrak a^*$
dual to the basis~$(\alpha)$ --- up to the usual identification
of $\mathfrak a$ with its dual given by the Killing
form of~$G$, this is the basis of fundamental weights.
One has $C=\sum_\alpha\R_+\varpi_\alpha$.

Recall also that $\beta$ is the sum of the fundamental weights of~$G$,
hence belongs to the positive Weyl chamber in~$\mathfrak a^*$.
It follows that $w\beta\geq\beta$ for any element~$w$ in the Weyl group~$W$
(Bourbaki, LIE VI, \S1, Proposition~18, p.~158). It follows that
$\langle w\beta,\varpi_\alpha\rangle\leq\langle \beta,\varpi_\alpha\rangle$
for any $w\in W$. Consequently,
\[ \phi_\beta(y) = \max\langle w\beta,y\rangle = \langle\beta,y\rangle \]
for any $y\in C$.
Moreover, if $y=\sum y_\alpha\varpi_\alpha\in\mathfrak a$,
then 
\[ \langle \beta,y\rangle=\langle \sum m_\alpha \alpha,\sum y_\alpha\varpi_\alpha\rangle = \sum m_\alpha y_\alpha. \]

Let $\Lambda$ be the set of dominant weights of~$\iota$
with respect to~$C$. For any $\lambda\in\Lambda$,
let us write $\lambda=\sum \lambda_\alpha \alpha$, for some
nonnegative $\lambda_\alpha\in\Z$. 
Consequently, for any $y=\sum y_\alpha\varpi_\alpha\in C$,
with $y_\alpha\geq 0$ for all~$\alpha$, one has
\[ \phi_\iota(y) = \max_{\chi\in\Phi}\langle\chi,y\rangle
= \max_{\lambda\in\Lambda} \langle\lambda,y\rangle
 = \max_{\lambda\in\Lambda} \left(\sum_\alpha \lambda_\alpha y_\alpha\right).
\]
The condition that $s\phi_\iota(y)\geq \phi_\beta$ on~$C$
therefore means that $s\max_{\lambda\in\Lambda}\lambda_\alpha\geq m_\alpha$
for any $\alpha\in\Phi$. In other words, 
$ \sigma = \max (m_\alpha/\tilde\lambda_\alpha)$,
where we have set $\tilde\lambda_\alpha= \max_{\lambda\in\Lambda}\lambda_\alpha$.
Moreover, $t$ is the number of simple roots~$\alpha$ such that
$\tilde\lambda_\alpha \sigma=m_\alpha$.

As in~\cite{maucourant2007},
let $\mathscr C$ be the convex hull of~$\Phi$ in~$\mathfrak a^*$;
it is a compact polytope whose dual~$\mathscr C^*$
is defined by the inequality $\phi_\iota(\cdot)\leq 1$
in~$\mathfrak a^*$. 
Let $s$ be a positive real number. By definition, $\beta/s$
belongs to~$\mathscr C$ if and only
if $\langle \beta/s, y\rangle\leq 1$ for any $y\in\mathfrak a^*$
such that $\phi_\iota(\mathfrak a)\leq 1$. This is precisely
equivalent to the fact that $s\phi_\iota-\langle\beta,\cdot\rangle$
is nonnegative on~$\mathfrak a^*$. 

Let $y\in \mathfrak a^*$ and let $w\in W$ be such that $wy\in C$.
Since $\phi_\iota$ is invariant under~$W$,
\[ s\phi_\iota(y)-\langle\beta,y\rangle = s\phi_\iota(wy)-\langle w^{-1}\beta,wy\rangle
\leq  s\phi_\iota(wy)-\langle \beta,wy \rangle, \]
with equality if and only if $w=e$.
We have thus shown that $\beta/s\in\mathscr C$  if and only
if the \pl function $s\phi_\iota-\phi_\beta$ is nonnegative.
In other words, $\sigma$ is the least positive real number
such that $\beta/\sigma$ belongs to~$\mathscr C$,
as claimed by Maucourant in~\cite{maucourant2007}.

Let $\mathscr F$ be the face of~$\mathscr C$ containing~$\beta/\sigma$.
It is also explained in~\cite{maucourant2007} 
that the dual face of~$\mathscr F$
is contained in the positive Weyl chamber~$C$ (Lemma~2.3)
and is given by
\[ \mathscr F^*=\{ y\in \mathscr C^* \,;\, \langle \beta,y\rangle=\sigma\}.\]
If again we decompose $y$ as $\sum y_\alpha\varpi_\alpha$,
it follows that $\mathscr F^*$ identifies as the set
of nonnegative $(y_\alpha)$ such that $\phi_\iota(y)\leq 1$
and $\langle \beta,y\rangle=\sigma$. 
The first condition gives us $\sum\tilde\lambda_\alpha y_\alpha\leq 1$
and the second is equivalent to $\sum m_\alpha y_\alpha =\sigma$.
This implies $\sum y_\alpha=1$ 
and $y_\alpha=0$ if $\sigma \neq m_\alpha/\tilde\lambda_\alpha$.
Consequently, $\codim\mathscr F=\dim \mathscr F^*=t$,
showing the agreement of our general theorem with
the result obtained by Maucourant in~\cite{maucourant2007},
under the assumption that $G$ is split.

\subsubsection{The general case}
We now treat the general case of a possibly non-split group. 
Let $T$ be be a maximal split torus in~$G$, so that
its Lie algebra~$\mathfrak a_\R$ is a Cartan subalgebra of
$\Lie(G)$.
We have already explained how the $F$-analytic Clemens complex
of~$X$ is related to the toric variety~$Y$ given by the closure of~$T$
in~$X$. As a consequence, all positivity conditions and dimensions
of faces which intervene in our geometric result rely only
on the divisors which are ``detected''  by~$Y$,
hence are expressed in terms of \pl functions in~$\mathfrak a^*$.
The previous analysis now applies verbatim and allow us again
to recover Maucourant's theorem.

\subsubsection{Adelic volumes}
Let us now assume that $F$ is a number field.
Theorem~\ref{theo.volume.global} 
describes the analytic behavior of the volume --- with
respect to the Haar measure --- of adelic
points in~$G(\AD_F)$ of height~$\leq B$, when $B\ra\infty$.
This allows to recover Theorem 4.13 in~\cite{gorodnik-m-o2009}.
In fact, that Theorem itself is proved as a corollary
of the analytic behavior of the associated Mellin transform
which had been previously shown in~\cite[Theorem 7.1]{shalika-tb-t2007}.
Similarly, this analytic behavior is a particular case of
our Proposition~\ref{prop.zeta.global}.

\subsection{Relation with the output of the circle method}

Let $X$ be a non-singular, geometrically irreducible, closed
subvariety of an affine space~$V$ of dimension~$n$ over a number field~$F$.
When $F=\Q$, $W=\Aff^n$ and $X$ is defined as the proper intersection of
$r$ hypersurfaces defined 
Let $n$ be a positive integer and let $f_1,\dots,f_r\in\Z[X_1,\dots,X_n]$
be polynomials with integer coefficients.

The circle method furnishes eventually an estimate, when 
the real number~$B$ grows to~$\infty$, of the number~$N(X,B)$
of solutions $x\in\Z^n$ of the system~$X$ given by $f_1(x)=\dots=f_r(x)=0$
whose ``height'' $\norm x$ is bounded by~$B$.
A set of conditions under which the circle method applies
is given by~\cite{birch62}; it suffices that 
the hypersurfaces~$\{f_i=0\}$ meet properly and define
a non-singular codimension~$r$ subvariety of~$\Aff^n$, and that
$n$ is very large in comparison to the degrees of the~$f_i$.

Let us assume that $X=V(f_1,\dots,f_r)$ is smooth and has codimension~$r$.
By the circle method, an approximation for~$N(B)$ is given
by a product of ``local densities'': for any prime number~$p$,
let 
\[ \mu_p(X) = \lim_{k\ra\infty } \frac{\Card X(\Z/p^k\Z)}{p^{k\dim X}};\]
for the infinite prime, let 
\[ \mu_\infty (X,B)= \lim_{\eps\ra 0 } \eps^{-r} \vol \{x\in\R^n\,;\, \norm x\leq B, \ \abs{f_i(x)}<\eps/2\}, \]
where $\vol$ refers to the Euclidean volume in~$\R^n$.
In some cases, one can indeed 
prove that $N(X,B)\sim \mu_\infty(X,B)\prod_{p<\infty}\mu_p(X)$
when $B\ra\infty$.

As already observed by~\cite{borovoi-r1995}  (Section~1.8)
we first want to recall that the right-hand-side $V(X,B)$ 
of this asymptotic expansion is really a volume.
Under the transversality assumption we have made on
the hypersurfaces defined by the~$f_i$,
there exists a differential form~$\tilde\omega$ on
a neighborhood of~$X$ in~$\Aff^n$
such that 
\[ \tilde\omega\wedge\mathrm df_1\wedge\dots\wedge\mathrm df_r 
=\mathrm dx_1\wedge\dots\wedge\mathrm dx_n. \]
The restriction of~$\tilde\omega$ to~$X$ is a well-defined
gauge form~$\omega$ on~$X$ and
\[ \mu_\infty(X,B)=\int_{\substack{ X(\R) \\ \norm x\leq B}} \mathrm d\abs\omega_\infty (x),
\qquad 
  \mu_p(X) = \int_{X(\Z_p)} \abs\omega_p\quad\text{($p$ prime).} \]

Let us write $[t:x]$ 
for the homogeneous coordinates of a point in~$\mathbf P^n(K)$, so that $t\in K$
and $x=(x_1,\dots,x_n)\in K^n$.
One then defines an metric on $\mathscr O(1)$ by the formula
\[ \norm{s_F}([t:x])=\frac{\abs{F(t,x)}}{\max(\abs t,\norm x)^{\deg F}} \]
for 
any point $[t:x]\in\P^n(\R)$
and 
any homogeneous polynomial $F$ in $\R[t,x]$,
$s_F$ being the section of~$\mathscr O(\deg F)$ attached to~$F$.
Let $\mathsec f_0$ be the section corresponding to $F=t$;
for $B\geq 1$, the condition $\norm x\leq B$ can thus be
translated to $\norm{\mathsec f_0}([1:x])\geq 1/B$.

First assume that the Zariski closure of~$X$ in~$\mathbf P^n$
is smooth; let $D$ be the divisor $\mathbf V(\mathsec f_0)$ in~$X$. 
Letting $d=\sum\deg(f_i)$, the divisor of~$\omega$ is equal to $d-n-1$.
In that case, the measure $\abs\omega$ coincides
with the measure $\tau_{(\bar X,(n+1-d)D}$ on~$X(\R)$, so
that $\mu_\infty(X,B)$ is an integral of the form studied
in this paper, namely
\[ \mu_\infty(X,B) = \tau_{(X,(n+1-d)D)}( \norm{\mathsec f_0}\geq 1/B ). \]
Assume $n> d$.
Then, when $B$ converges to infinity, Theorem~\ref{theo.geometric-volume}
implies that $\mu_\infty(X,B)\approx B^{n-d} (\log B)^{b-1}$,
where $b$ is the dimension of the Clemens complex of
the divisor~$D(\R)$, that is the maximal number of components
of $D(\R)$ that have a common intersection point.
(When $n=d$, a similar result holds except that $b-1$ has to be
replaced by~$b$.)

The paper \cite{duke-r-s1993}
showed that this asymptotic does not hold 
for the specific example of a generic $\operatorname{SL}(2)$-orbit
of degree~$d$ binary forms. For example, when $d=3$,
this amounts to counting points on an hypersurface of degree~$4$
in~$\P^4$ and the expected exponent~$0=4-4$ is replaced by~$2/3$.
This discrepancy is explained by~\cite{hassett-tschinkel2003}:
observing that $(\bar X,D)$ is not smooth in that case,
they computed a log-desingularization of $(\bar X,D)$
on which the behaviour of the integral $\mu_\infty(X,B)$
can be predicted.

In the general case,
let us thus consider a projective smooth compactification $\bar Y$ of~$X$
together with a projective, generically finite,
morphism $\pi\colon\bar Y\ra\bar X $
such that, over~$\bar\Q$, the  complementary divisor~$E$ to~$X$ in~$\bar Y$
has strict normal crossings. Let $E_\alpha$ be the irreducible
components of~$E$.  
It is customary to assume that $\pi$ induces an isomorphism over~$X$,
but this is not necessary in the following analysis;
it is in fact sufficient to assume that the degree of~$\pi$
is constant on the complement to a null set.

Let us pose $\eta=\pi^*\omega$; then, the divisor of~$\eta$
has the form $\div(\eta)=-\sum\rho_\alpha E_\alpha=-E'$, so that
\[ \mu_\infty(X,B) = \int_{\substack{\bar Y(\R) \\\norm{ \pi^*\mathsec f_0}\geq 1/B}}  \abs\eta = \tau_{(\bar Y,E')} (\{ \norm{\pi^*\mathsec f_0}\geq 1/B\}).
\]
Write also $\pi^*D=\sum\lambda_\alpha E_\alpha$.
According to Theorem~\ref{theo.geometric-volume},
$\mu_\infty(X,B)\approx B^a(\log B)^{b-1}$ where now,
$ a= \max (\rho_\alpha-1)/\lambda_\alpha$, the minimum
being restricted to those $\alpha$ such that $E_\alpha(\R)\neq\emptyset$;
the integer~$b$ is the dimension of the subcomplex
of the analytic Clemens complex
consisting of those~$E_\alpha$ achieving this minimum.
(When $a=0$, $b-1$ is replaced by~$b$.)

Let us give the specific example of an $\SL(2)$-orbit 
in the affine space of binary forms of degree~$d$,
as treated in~\cite{hassett-tschinkel2003}.
That paper constructs a pair $(\bar Y,E)$ with an action of~$\mathfrak S_n$
such that $E=E_1+E_2$ has two irreducible components, 
the divisors $E_1$ and~$E_2$ (denoted $A[n-1]$ and~$A[n]$ in that paper)
forming a basis of the $\mathfrak S_n$-invariant part of $\Pic(\bar Y)$, 
Moreover, $K_{\bar Y}^{-1}=E_1+2E_2$.
They compute the inverse image of~$\mathscr O(1)$
and obtain $\frac{n-2}2 E_1+\frac n2 E_2$ (\loccit, Lemma 3.3).
Consequently, $(\rho_1-1)/\lambda_1=0$ and
$(\rho_2-1)/\lambda_2=2/n$. Finally, $a=2/n$ and $b=1$,
implying that $\mu_\infty(X,B)\approx B^{2/n}$.

Let us return to the general case,
and assume that  $K_{\bar X}$ is $\Q$-Cartier.
Then, the \emph{log-discrepancies} $(\eps_\alpha)$ are defined by the formula
\[ K_{\bar Y}(E) = \pi^*(\bar K_{\bar X}(D)) + \sum\eps_\alpha E_\alpha, \]
so that $1-\rho_\alpha=(d-n)\lambda_\alpha+\eps_\alpha$ for any~$\alpha$.
Then,
\[ a= n-d + \min_\alpha \frac{-\eps_\alpha}{\lambda_\alpha}, \]
where, again, the minimum is restricted to those~$\alpha$
such that $E_\alpha(\R)\neq\emptyset$.
When $(\bar X,D)$  has log-canonical singularities, $\eps_\alpha\geq 0$
for each~$\alpha$
and one obtains again $a\leq n-d$ since the log-discrepancy corresponding
to the strict transform of the components of~$D$ is~$0$.

However, in the case treated by~\cite{hassett-tschinkel2003},
$\eps_1=0$ and $\eps_2=-1$, leading to the
opposite inequality $a=\frac2n>0$.

\subsection{Matrices with given characteristic polynomial}

\subsubsection{}

Let $V_P$ be the $\Z$-scheme of matrices in the affine $4$-space
whose characteristic polynomial is $X^2+1$,
and let $B_T$ be the euclidean ball of radius~$T$ in $\R^4$.
Shah gives in~\cite{shah2000} the following
asymptotic expansion:
\[ \# (V_P(\Z)\cap B_T) \sim T \zeta_K^*(1) \dfrac{\pi^{1/2}}{\Gamma(3/2)}
\dfrac{\pi}{\Gamma(2/2)\zeta(2)}
   = T \zeta_K^*(1) \frac{2\pi}{\zeta(2)} = C_{X^2+1}\, T  , \]
where $\zeta_{\Q(i)}^*$ is the leading term at~$1$ of Dedekind's
zeta function relative to the number field $\Q(i)$.

A matrix in~$V_P$ has the form
$\left(\begin{smallmatrix} x & z \\ y & -x \end{smallmatrix}\right)$
with $x^2+yz+1=0$.
Let $X$ denote the subvariety of $\P^3$ defined by $X^2+YZ+T^2=0$. It
is smooth of dimension~$2$ over $\Z[1/2]$.
The scheme $V_P$ is exactly the open subset $U\subset X$ defined by $T\neq 0$.
The divisor at infinity $D=X\setminus U$ is defined by $T=0$;
it is smooth as well (still over $\Z[1/2]$).

\subsubsection{}
Let $p$ be an odd prime number.
One has $\# U(\F_p)=p^2+\legendre{-1}{p} p$.
Indeed, $z\neq 0$ gives $p(p-1)$ points. If $z=0$, $y$ may be
arbitrary and $x$ has to be $\pm\sqrt{-1}$, hence
$2p$ points if $-1$ is a square and $0$ else.
Finally,
\[ \vol U(\Z_p) = p^{-2} \# U(\F_p) = 1 + \legendre{-1}{p} \frac1p. \]

For $p=2$, we split $U(\Z_2)$ in two parts:
\begin{itemize}
\item $y$ odd, so that $z=-(x^2+1)/y$. This has measure $1/2$.
\item $y=2y'$ even, which implies that $y'$ and $x=2x'+1$ odd, so that
\[ z=-(x^2+1)/y= - (1+2x'+2(x')^2)/y'. \]
This has measure $1/4$. Indeed, it parametrizes as
\[ 
\left\{
\begin{aligned}
 x& =1+2u \\ y& = \frac{2}{1+2v} \\ z& = (1+2u+2u^2)(1+2v) =
1+2u+2v+2u^2(1+2v)+4uv, \end{aligned}\right.
\]
with $(u,v)\in(\Z_2)^2$.
The differential of this map is given by the matrix
\[ \begin{pmatrix} 2  & 0 \\
       0 & -4(1+2v)^{-2} \\
 2 + 4u(1+3v) & 2 + (4u^2+4u) \end{pmatrix}
 = 2 \begin{pmatrix} 1 & 0 \\ 0&  -2(1+2v)^{-2} \\ 1+2u(1+3v) & 1+2u(1+u)
\end{pmatrix}\]
and is twice a $2\times 3$ matrix with coefficients in~$\Z_2$,
of which one of the $2\times 2$ is invertible.
Therefore, the measure
of its image is $\abs{2}_2^2=1/4$.
\end{itemize}
Finally, $\vol U(\Z_2)=1/2+1/4=3/4$.

\subsubsection{}
The virtual character $\EP (U)$ is $-[\chi]$,
for $\chi=(-1/\cdot)$ the quadratic character corresponding to $\Q(i)$.
We will prove this in general below, however, this can be seen
directly as follows. 

First, $U$ is an hypersurface of the affine~$3$-space,
so it has no non-constant invertible function.
Let us now study its Picard group.
Over $\Q(i)$, $X$ is the hypersurface of~$\P^3$ defined
by the equation~$(x+it)(-x+it)=yz$, so is isomorphic to~$\P^1\times\P^1$,
the two factors being interchanges by the complex conjugation.
Consequently, $\Pic(X_{\bar F})$, being generated by these two lines,
is the Galois module $\mathbf 1\oplus[\chi]$.
The result then follows from the fact that  
$\Pic(X_{\bar F})$ maps surjectively to~$ \Pic(U_{\bar F})$,
its kernel being the submodule generated by the class
of the hyperplane of equation~$t=0$ in~$\P^3$.

The ${\rm L}$-function of~$\EP (U)$ has local factors
${\rm L}_p(\EP (U),s)=1-\legendre{-1}p p^{-s}$.
The product ${\rm L}(\EP (U),s)$ of all ${\rm L}_p$ for $p>2$
is exactly ${\rm L}(\chi,s)^{-1}$. At $s=1$, it has neither a pole, nor a zero.

\subsubsection{}
One has
\[ \vol U(\Z_2) \prod_{p>2} \vol U(\Z_p) \mathrm L_p(1)
     = \frac{3}4 \prod_{p>2} \left( 1 - p^{-2} \right)
   =  \frac {6}{\pi^2}
  = \zeta(2)^{-1}.
 \]
As ${\rm L}(\chi,s)\zeta(s)= \zeta_{\Q(i)}(s)$,
one has
${\rm L}(\chi,1) = \zeta_{\Q(i)}^*(1)$ and the normalized volume
of $U(\AD_f)$ is equal to
\[   \zeta_{\Q(i)}^*(1) \frac1{\zeta(2)}. \]

\subsubsection{}
Now we have to compute the volume of the divisor at infinity.
Due to the way matrices are counted in~\cite{shah2000}
(one wants $2x^2+y^2+z^2\leq T^2$)
the natural metric on the tautological bundle of $\P^3$
is given by the formula
\[  \norm{s_P}(x:y:z:t) = \frac{P(x,y,z,t)}{ \max\big(2\abs x^2+\abs y^2 + \abs
z^2 , \abs t^2\big)^{\deg P /2}}, \]
where $P$ is a homogeneous polynomial in $4$~variables
and $s_P$ the corresponding section of $\mathcal O(\deg P)$.

The divisor $D$ has equations $s_1=X^2+YZ+T^2=0$ and $s_2=T=0$ in~$\P^3$.
We'll compute its volume using affine coordinates $(y,z,t)$ for $\P^3$
and $x$ for~$D$.
Affine equations of~$D$ corresponding to the two previous sections
are $1+yz+t^2=0$ and $t=0$.
One has
\[ \mathrm d(1+yz+t^2)\wedge \mathrm dz \wedge \mathrm dt
   = z \mathrm dy\wedge\mathrm dz\wedge\mathrm dt, \]
and by definition, 
\[ 
\norm{\mathrm dy\wedge\mathrm dz\wedge\mathrm dt}
 = \max\big( 2+y^2+z^2,t^2\big)^2 . \]
Consequently, on~$D$, one obtains the following equalities
\[ \lim \frac{\norm{s_1}}{\abs{1+yz+t^2}}= \frac{1}{\max(2+y^2+z^2,t^2)}
=\frac1{2+y^2+z^2}  \]
(since $t=0$ on~$D$) and 
\[ \lim \frac{\norm{s_2}}{\abs t}= \frac{1}{\max(2+y^2+z^2,t^2)^{1/2}}
= \frac1{(2+y^2+z^2)^{1/2}}. \]
Finally, 
\begin{align*}
 \norm{\mathrm dz}
& = \norm{z \mathrm dy\wedge\mathrm dz\wedge\mathrm dt}
   \lim \frac{\norm{s_1}}{\abs{1+yz+t^2}} \lim \frac{\norm{s_2}}{\abs t} \\
& = \abs{z} \max\big( 2+y^2+z^2,t^2\big)^{1/2} 
 = \abs z \big( 2+(1/z)^2+z^2\big)^{1/2} \\
& = \big( 1+2z^2+z^4 \big)^{1/2} 
= 1+z^2, \end{align*}
so that the canonical measure on $D(\R)$ is given by
$\abs{\mathrm dz}/(1+z^2)$.
Hence, $\vol D(\R)=\pi$.

\subsubsection{}
Finally, we obtain that 
Shah's constant $C_{X^2+1}$ satisfies
\[ C_{X^2+1} = 2 \vol D(\R) \vol U(\AD_f), \]
compatibly with Theorem~\ref{theo.geometric-volume}.

\clearpage

\appendix
\def\lcoeff{\operatorname{lcoeff}}

\small

\section{Tauberian theorems}

\def\thesubsubsection{\thesection.\arabic{subsubsection}}
\def\theequation{\thesection.\arabic{equation}}

Let $(X,\mu)$ be a measured space and $f$ a positive measurable
function on~$X$. Define
\[ Z(s) = \int_{X} f(x)^{-s}\, d\mu(x) \quad\text{and}\quad
  V(B) = \mu(\{f(x)\leq B\}). \]
One has
\[ Z(s) = \int_0^{+\infty} B^{-s}\, dV(B). \]

\begin{theo}
\label{theo.tauber}
Let $a$ be a real number;
we assume that $Z(s)$ converges for~$\Re(s)>a$ and extends to a
meromorphic function in the neighborhood of 
the closed half-plane~$\Re(s)\geq a-\delta$,
for some positive real number~$\delta$.

If $a<0$, then $V(B)$ has the limit $Z(0)$ when $B\ra\infty$.

Furthermore, assume that:
\begin{enumerate}
\item 
$Z$ has a pole of order~$b$ at $s=a$ and no other pole in
the half-plane $\Re(s)\geq a-\delta$.
\item
$Z$ has moderate growth in vertical strips, \emph{i.e.},
there exists a positive real number~$\kappa$ such that for any~$\tau\in\R$,
\[ \abs{Z(a-\delta+i\tau)} \ll (1+\abs{\tau})^\kappa . \]
\end{enumerate}
Then, there exist a monic polynomial~$P$, a real number~$\Theta$
and a positive real number~$\eps$
such that, when $B\ra\infty$,
\[ V(B) = \begin{cases}
 \Theta B^a P(\log B) + \mathrm O(B^{a-\eps}) & \text{if $a\geq 0$;}\\
Z(0)+\Theta B^aP(\log B) + \mathrm O(B^{a-\eps}) & \text{if $a<0$.} \end{cases}
\]
Moreover, if $a\neq 0$,
then
\[ \deg(P) = b-1 \quad\text{and}\quad \Theta \, a\,  (b-1)! = \lim_{s\ra a}(s-a)^b Z(s)
\]
while 
\[ \deg(P)=b \quad\text{and}\quad \Theta \,b! = \lim_{s\ra a}(s-a)^b Z(s
)
\]
if $a=0$.
\end{theo}


For any integer~$k\geq 0$, let us define
\[ V_k(B)= \frac1{k!} \int_{f(x)\leq B} \left( \log\frac B{f(x)}\right)^k \,d\mu(x)
=  \frac1{k!} \int_X \left(\log^+ \frac B{f(x)}\right)^k\,d\mu(x), \]
where $\log^+(u)=\max(0,\log u)$ for any positive real number~$u$.

\begin{lemm}\label{lemm.tauber.k}
Let $k$ be an integer such that~$k>\kappa$. Then there exist
polynomials~$P$ and~$Q$ with real coefficients, a real number~$\delta>0$
such that 
\[ V_k(B) = B^a P(\log B) + Q(\log B) + \mathrm O(B^{a-\delta}). \]
Moreover, the polynomials $P$ and $Q$ satisfy
\begin{gather}
 P(T) = \begin{cases}
    \frac{1}{a^{k+1}(b-1)!} \Theta  T^{b-1}+\dots & \text{if $a\neq 0$,} \\
    0 & \text{if $a=0$;} \end{cases} \\
 Q(T)= \begin{cases}
       0 & \text{if $a>0$;} \\ 
    \frac{1}{(b+k)!} \Theta T^{b+k} + \dots  &\text{if $a=0$;} \\
  \frac1{k!} Z(0) T^{k}+\dots & \text{if $a<0$.} 
\end{cases}
\end{gather}
\end{lemm}
\begin{proof}
We begin with the classical integral
\[ \int_{\sigma+i\R} \lambda^s \frac{\mathrm ds}{s^{k+1}} = \frac{2i\pi}{k!} (\log^+(\lambda))^k, \]
where $\sigma$ and~$\lambda$ are positive  real numbers.
For $\sigma>\max(0,a)$, this implies that
\begin{align*}  
V_k (B) & = \frac1{k!} \int_{X} \log^+(B/f(x))^k\,\mathrm d\mu(x) \\
& = \int_X \frac{1}{2i\pi} \int_{\sigma+i\R} (B/f(x))^s \frac{\mathrm ds}{s^{k+1}} \, \mathrm d\mu(x) \\
& = \frac{1}{2i\pi} \int_{\sigma+i\R} B^s Z(s) \,\frac{\mathrm ds}{s^{k+1}}, 
\end{align*}
where the written integrals  converge absolutely.
We now move the contour of integration to the left of the pole $s=a$,
the estimates for $Z(s)$ in vertical strips allowing us to
apply the residue theorem. Only $s=a$ and $s=0$ may give a pole 
within the two vertical lines $\Re(s)=\sigma$ and $\Re(s)=a-\delta$
and we obtain
\[ V_k(B) = \sum_{\substack{u\in\{0,a\}\\ a-\delta < u <\sigma}} \Res_{s=u}\left(\frac{B^sZ(s)}{s^{k+1}}\right)
   + \frac{1}{2i\pi} \int_{a-\delta+i\R} B^s Z(s)\,\frac{\mathrm ds}{s^{k+1}}. \]

If $a>0$ and $a-\delta>0$, or if $a<0$,
one checks that there exists a polynomial~$P$
of degree~$b-1$ and of leading coefficient $\Theta/a^{k+1}(b-1)!$
such that
\[ \Res_{s=a} \left( \frac{B^s Z(s)}{s^{k+1}}\right)
 = B^a P(\log B). \]
In the case $a>0$, we moreover set $Q=0$.

If $a=0$, let us set $P=0$.
There exists a polynomial~$Q$ of degree~$b+k$
and of leading coefficient $\Theta/(b+k)!$ 
such that
\[ \Res_{s=0} \left(  \frac{B^s Z(s)}{s^{k+1}}\right)
 = Q(\log B). \]

Finally, if $a<0$, one verifies that there exists
a polynomial~$Q$ of degree~$k$ and of leading coefficient
$Z(0)/k!$ such that
\[ \Res_{s=0} \left(  \frac{B^s Z(s)}{s^{k+1}}\right)
 = Q(\log B). \]
This implies the lemma.
\end{proof}

\begin{lemm}\label{lemm.tauber.iteration}
Assume that there are polynomials~$P$ and~$Q$, and a positive
real number~$\delta$, such that
\[ V_k(B) = B^a P(\log B) + Q(\log B) + \mathrm O(B^{a-\delta}), \]
where we moreover assume that $P=0$ if $a=0$ and $Q=0$ if $a>0$.
Then for any positive real number~$\delta'$ such
that $\delta'<\delta/2$, one has the asymptotic expansion
\[ V_{k-1}(B) = B^a (aP(\log B)+P'(\log B)) + Q'(\log B) + \mathrm O(B^{a-\delta'}). \]
\end{lemm}
\begin{proof}
For any $u\in(-1,1)$, one has
\begin{align*}
 V_k(B(1+u))-V_k(B) & =  B^a\big(P(\log(B(1+u)))-P(\log B)\big) \\
& \qquad {} + Q(\log(B(1+u)))-Q(\log B)  + \mathrm O(B^{a-\delta}) \\
&= B^a P(\log B) \big( (1+u)^a-1\big)  \\
& \qquad {}+ B^a (1+u)^a \big( P(\log B+\log(1+u)) - P(\log B) \big) \\
& \qquad {} + \big( Q(\log B+\log(1+u))-Q(\log B) \big)
 +  \mathrm O\big(B^{a-\delta}\big) \\
&= B^a P(\log B)\big(au +\mathrm O(u^2)\big) 
 + B^a P'(\log B) u + B^a \mathrm O\big(\log B)^{\deg P-2}u^2\big)\\
& \qquad {}+ Q'(\log B) u + \mathrm O\big((\log B)^{\deg Q-2}u^2\big) 
 \qquad {}+ \mathrm O\big(B^{a-\delta}\big) .
\end{align*}
Let $\eps$ be a positive real number; if $u=\pm B^{-\eps}$,
we obtain
\begin{multline*} \frac{V_k (B(1+u))-V_k(B)}{\log(1+u)}
=  B^a (aP(\log B)+P'(\log B)) +\tilde{\mathrm O}(B^{a-\eps})) \\
   + Q'(\log B) + \tilde{\mathrm O}\big( B^{-\eps}\big)
 + \mathrm O\big( B^{a-\delta+\eps}\big), \end{multline*}
where the~$\tilde{\mathrm O}$ notation indicates unspecified
powers of~$\log B$. If $\delta'<\delta/2$, we may choose
$\eps$ such that $\delta'<\eps<\delta/2$ and then,
\begin{equation}\label{eq.diffVk}
   \frac{V_k (B(1+u))-V_k(B)}{\log(1+u)}
=  B^a (aP(\log B)+P'(\log B)) 
   + Q'(\log B) 
 + \mathrm O\big( B^{a-\delta'}\big). \end{equation}

On the other hand,
for any real number~$u$ such that $0<u<1$, and any positive
real number~$A$, one has 
\[ 
 \frac{ \log(A(1-u))^k-\log(A)^k}{\log(1-u)} 
\leq k \log(A)^{k-1} \leq
 \frac{ \log(A(1+u))^k-\log(A)^k}{\log(1+u)}, \]
which implies the inequality
\begin{equation}
 \frac{V_k(B(1-u))-V_k(B)}{\log(1-u)} \leq  V_{k-1}(B) \leq 
 \frac{V_k(B(1+u))-V_k(B)}{\log(1+u)} . \end{equation}
Plugging in the estimates of Equation~\eqref{eq.diffVk}, we deduce
the asymptotic expansion
\begin{equation}
  V_{k-1}(B) = B^a (aP(\log B)+P'(\log B)) + Q'(\log B) + 
 + \mathrm O\big( B^{a-\delta'}\big), \end{equation}
as claimed.
\end{proof}

\begin{proof}[Proof of Theorem~\ref{theo.tauber}]
We now can prove our Tauberian theorem.
Let $k$ be any integer such that $k>\kappa$.
Lemma~\ref{lemm.tauber.k} implies an asymptotic expansion
for $V_k(B)$; let $P,Q,\delta$ be as in this Lemma.
Applying successively $k$ times
Lemma~\ref{lemm.tauber.iteration},
we obtain the existence of  an asymptotic expansion for $V(B)$
of the form
\[ V(B) = B^a  D_a^k P(\log B) + D_0^k Q(\log B) + \mathrm O(B^{a-\delta'}),\]
for some positive real number~$\delta'$
(any positive real number~$\delta'$ such that $\delta'<\delta/2^k$
is suitable),
where we have denoted by $D_a$ and $D_0$ the differential
operators $P\mapsto aP+P'$ and $P\mapsto P'$.

For $a\neq 0$, the operator~$D_a$ doesn't change the degree
but multiplies the leading coefficient by~$a$.
Consequently, 
\[ D_a^k P (T)= \begin{cases}
   \frac1{a(b-1)!} \Theta^{b-1}+\dots & \text{if $a\neq 0$,} \\
 0 &\text{if $a=0$.} \end{cases}\]
Similarly, the operator~$D$ decreases the degree by~$1$
and multiplies the leading coefficient by the degree.
It follows that
\[ D_0^k Q(T) = \begin{cases} 
Z(0) & \text{if $a<0$;}\\
  \frac1{b!}\Theta T^b +\dots & \text{if $a=0$;} \\
0 & \text{if $a>0$.} \\
 \end{cases} \]
Theorem~\ref{theo.tauber} now follows easily.
\end{proof}

\bigskip

In ultrametric contexts, the function~$f$ usually takes values
of the form~$q^n$, with $n\in\Z$, where $q$ is a real number
such that~$q>1$.
In that case, the function~$Z$ is $(2i\pi)/\log q$-periodic,
its poles form arithmetic progressions of common difference~$2i\pi/\log q$,
and Theorem~\ref{theo.tauber}  doesn't apply.

Let $q$ be a real number with $q>1$ and
assume that $f(x)\in q^\Z$ for any $x\in X$.

Let $a$ be a nonnegative real number;
let us assume that $Z(s)$ converges for~$\Re(s)>a$ and extends to a
meromorphic function in the neighborhood of 
the closed half-plane~$\Re(s)\geq a-\delta$,
for some positive real number~$\delta$.

Suppose furthermore that the poles of~$Z$ belong to
finitely many arithmetic progressions of the form 
$a_j+\frac1{\log q} 2i\pi \Z $, where $a_1,\dots,a_t\in\C$
are complex numbers of real part~$a$, not two of them being congruent
mod $2i\pi/\log q$; let $b_j=\ord_{s=a_j} Z(s)$ and
let $c_j=\lim_{s\ra a_j} (s-a_j)^{b_j} Z(s)$.

Let us make the change of variable $u=q^{-s}$ and set $Z(s)=\Phi(u)$.
The function $\Phi$ is defined by
\[ \Phi(u)=\sum_{n\in\Z} Z_n u^n, \qquad \text{where}\quad
Z_n=\mu(\{ f(x)=q^{n} \}); \]
it is defined for $0<\abs u<q^{-a}$ and is holomorphic
in that domain; moreover, it extends to a meromorphic function
on the domain $0<\abs u<q^{\delta-a}$, with poles
at $q^{-a_j}$  (for $1\leq j\leq t$) such that
\[ \lim_{u\ra q^{-a_j}} (1-q^{a_j}u)^{b_j} \Phi(u)
 = \lim_{s\ra a_j} (1-q^{a_j-s})^{b_j} Z(s)
= (\log q)^{b_j}  c_j. \]
Consequently, there are polynomials $p_j$ of degree~$b_j$
and leading coefficient $(\log q)^{b_j} c_j$
such that
\[  Z(u) - \sum_{j=1}^t  \frac {p_j(u)}{(1-q^{a_j}u)^{b_j}} \]
is holomorphic in the domain defined by $0<\abs u<q^{\delta-a}$. 
Therefore, the Cauchy formula implies an asymptotic expansion of the form
\[ Z_n = \sum_{j=1}^t P_j(n)  q^{a_j n} + \mathrm O(q^{(a-\delta')n}), \]
where $\delta'$ is any positive real number such that $0<\delta'<\delta$
and, for $1\leq j\leq t$, $P_j$ is a polynomial 
of degree~$b_j-1$ and of leading coefficient $c_j (\log q)^{b_j} /(b_j-1)!$.

Since the Laurent series defining~$\Phi$ has nonnegative coefficients,
the inequality $\abs{Z(u)}\leq Z(\abs u)$ holds for any $u\in\C$
such that $0<\abs u<q^{-a}$.
In particular, we see, as is well known,
that $\Phi$ has a pole  on its circle of convergence.
This means that, unless $t=0$, we can assume that $a_1=a$
and that $b_j\leq b_1$ for all~$j$.

We are not  interested in the sequence~$(Z_n)$ itself, 
but rather on the sums 
$ V(q^n)=\sum_{m\leq n} Z_n$,
when $n\ra\infty$. (Their convergence 
follows from the fact that $\Phi(u)$ converges for arbitrary
small nonzero complex numbers~$u$.)

We begin by observing that for any integer~$j$, 
there exists a polynomial $Q_j$ such that
\[ P_j(m) = Q_j(m)- q^{-a_j}Q_j(m-1) \]
for any $m\in\Z$.
Moreover, if $q^{a_j}\neq 1$, then there
is only one polynomial satisfying these relations,
its degree satisfies $\deg(Q_j)=\deg(P_j)$
while its leading coefficient is equal to $\lcoeff(P_j)/(1-q^{-a_j})$;
however, if $q^{a_j}=1$, then $\deg(Q_j)=\deg(P_j)+1$
and $\lcoeff(Q_j)=\lcoeff(P_j)/(\deg(P_j)+1)$.

Let us now separate the discussion according to the value of~$a$.

\emph{Case $a<0$.} ---
Then, $V(q^n)$ has the limit $\Phi(1)$ when $n\ra+\infty$,
and
\[ V(q^n)=\Phi(1) - \sum_{m>n} Z_n
= \Phi(1) - \sum_{j=1}^t \sum_{m>n} q^{a_jm} P_j(m)+ \mathrm O(q^{(a-\delta')n}) , \]
provided $\delta'<\delta$ is chosen so that $a<a+\delta'<0$.
Since
\[ \sum_{m>n} q^{a_jm} P_j(m) = \sum_{m>n} \left(
     q^{a_jm}Q_j(m)-q^{a_j(m-1)}Q_j(m-1)\right)
= - q^{a_j n} Q_j(n), \]
we obtain that
\begin{equation}
 V(q^n)=\Phi(1) + \sum_{j=1}^t q^{a_j n}Q_j(n) + \mathrm O(q^{(a-\delta')n}).
\end{equation}

\emph{Case $a> 0$.} --- In that case, we have
\[ V(q^n)= \sum_{j=1}^t  \sum_{m\leq n} \left(
  q^{a_jm}Q_j(m)-q^{a_j(m-1)}Q_j(m-1)\right)+ \mathrm O(q^{a-\delta'n}), \]
so that
\begin{equation}\label{eq.q.a>0}
V(q^n)= \sum_{j=1}^t q^{a_jn} Q_j(n) + \mathrm O(q^{a-\delta'}n). \end{equation}
Moreover, $\deg(Q_j)=b_j$ for all~$j$.

\emph{Case $a=0$.} ---
Then $V(q^n)$ also satisfies the asymptotic expansion~\eqref{eq.q.a>0}.
However, $\deg Q_1=b_1+1$ and $\deg Q_j\leq b_j\leq b_1$ for $j\neq 1$,
and we obtain
\begin{equation} \lim_{n\ra\infty} V(q^n) n^{-b_1-1} \lcoeff(Q_1)
 = \frac1{b_1+1}\lcoeff(P_1) = c_1 \frac{(\log q)^{b_1}}{b_1+1}.\end{equation}

\begin{coro}\label{coro.theo.tauber2}
Let us retain the previous hypotheses, assuming moreover that $a>0$.
Then, we have the following weak asymptotic  behavior:
\[ 0< \liminf V(q^n) q^{-an} n^{-b_1} \leq \limsup V(q^n) q^{-an}n^{-b_1}<\infty. \]
\end{coro}

\begin{coro} \label{coro.theo.tauber2b}
Let us assume that for some positive integers~$b$ and~$d$, the function
$Z(s)(1-q^{(a-s)d})^b$ has a holomorphic expansion in
some neighborhood of the half-plane $\{\Re(s)\geq a-\delta\}$.
Then, where $n$ is restricted to belong to any  arithmetic progression
mod.~$d$, the sequence $(V(q^n)q^{-na}\log(q^n)^{-b})$
has a limit.
\end{coro}
\begin{proof}
The assumptions allow us to set $t=d$ 
and $a_j=a+2i\pi(j-1)/d\log(q)$, for $1\leq j\leq d$, and $b_j=b$.
Let us fix $m\in\N$ and let us write $n=m+kd$, where $k\in\N$
goes to infinity.
For $n\in\N$, we may write
\[ V(q^n) q^{-na}\log(q^n)^{-b}
\sum_{j=1}^d q^{(a_j-a)n} \frac{Q_j(n)}{(n\log q)^b} + \mathrm O(q^{-\delta'n}).
\]
The asserted convergences follow from the fact the observation that
for any~$j$, $q^{(a_j-a)n}$ is a $d$th root of unity.
More precisely, we find
\[ \lim_{\substack{n\ra\infty \\ n\equiv n_0\pmod d}}
  V(q^n) q^{-na} \log(q^n)^{-b}
 =\frac1{(b-1)!}\sum_{j=1}^d \exp(2i\pi (j-1)/d)  \frac{c_j}{1-q^{-a_j}}, \]
where
\[ c_j = \lim_{s\ra a_j} (s-a_j)^{b} Z(s). \]
\end{proof}

The following corollary has been inspired by work in progress
by  Cluckers, Comte  and Loeser.
\begin{coro}\label{coro.theo.tauber2c}
Let $V^*(q^n)$ be the Cesaro mean of $V(q^n)$, namely
\[ V^*(q^n)=\frac1{n+1}\sum_{m=0}^n V(q^m) .\]
If $a>0$,  then $V^*$ satisfies
\[ \lim_{n\ra\infty} V^*(q^n) q^{-an} n^{-b_1} 
= c_1 \frac{ (\log q)^{b_1}}{1-q^{-a}}. \]
If $a<0$, then 
\[ \lim_{n\ra\infty} \left(V^*(q^n)-Z(0)\right)q^{-an} n^{-b_1}
=c_1 \frac{ (\log q)^{b_1}}{1-q^{-a}}.
\]
\end{coro}
\begin{proof}
This follows from the main result and the fact for any
complex number $z\in\C$ such that $\abs z=1$ but $z\neq 1$,
the sequence $(z^n)$ Ces\'aro-converges to~$0$.
\end{proof}

\bigskip

We conclude this Appendix by a Tauberian result which is
useful in $S$-integral contexts.

\begin{theo}
\label{theo.tauber.S}
Let $a$ be a real number;
we assume that $Z(s)$ converges to a holomorphic
function for~$\Re(s)>a$.
Let us furthermore assume that it has a meromorphic
continuation of the following form:
there exists a positive integer~$b\geq 1$
and a finite family $(q_j,b_j)_{j\in J}$
where $q_j$ is a real number such that~$q_j>1$ and $b_j$
is an integer satisfying $1\leq b_j\leq b-1$ such that,
setting $b_0=b-\sum_{j\in J}b_j$, the function~$Z_0$ defined by
\[
Z_0(s)=Z(s) \left(\frac {s-a}{s-a+1}\right)^{b_0}
\prod_{j\in J} \left(1-q_j^{s-a}\right)^{b_j} 
\]
extends to a holomorphic function 
with
moderate growth in vertical strips, \emph{i.e.},
there exists a positive real number~$\kappa$ such that for any~$\tau\in\R$,
\[ \abs{Z(a-\delta+\mathrm i\tau)} \ll (1+\abs{\tau})^\kappa . \]
Assume also that for any two $j,j'\in J$,
$\log q_j/\log q_{j'}$ is not a Liouville number.

Then, there exist a monic polynomial~$P$, a real number~$\Theta$
and a positive real number~$\eps$
such that, when $B\ra\infty$,
\[ V(B) = \begin{cases}
 \Theta B^a P(\log B) + \mathrm O(B^{a}(\log B)^{\max(b_j)})  & \text{if $a\geq 0$;}\\
Z(0)+\Theta B^aP(\log B) + \mathrm O(B^{a}(\log B)^{\max(b_j)})  
& \text{if $a<0$.} \end{cases}
\]
Moreover, if $a\neq 0$,
then
\[ \deg(P) = b-1 \quad\text{and}\quad \Theta \, a\,  (b-1)! = \lim_{s\ra a}(s-a)^b Z(s)
\]
while 
\[ \deg(P)=b \quad\text{and}\quad \Theta \,b! = \lim_{s\ra a}(s-a)^b Z(s)
\]
if $a=0$.
\end{theo}
\begin{proof}
For any integer~$k\geq 0$, let us define
\[ V_k(B)= \frac1{k!} \int_{f(x)\leq B} \left( \log\frac B{f(x)}\right)^k \,d\mu(x)
=  \frac1{k!} \int_X \left(\log^+ \frac B{f(x)}\right)^k\,d\mu(x), \]
where $\log^+(u)=\max(0,\log u)$ for any positive real number~$u$.

As in Theorem~\ref{theo.tauber}, we begin by proving
an asymptotic expansion for $V_k(B)$, where $k$
is an integer satisfying $k>\kappa$.
As above, we have 
\[
V_k (B)  =  \frac{1}{2i\pi} \int_{\sigma+i\R} B^s Z(s) \,\frac{\mathrm ds}{s^{k+1}}, 
\]
for $\sigma>a$, and we will move the line of integration to the left
of $s=a$,
the novelty being the presence of infinitely many poles
on the line $\Re(s)=a$, namely at any complex number
of the form $\alpha_{j,m}=a+2\mathrm i m\pi/\log q_j$, for some $j\in J$
and some integer $m\in\Z$.

Let $F_k$ be the holomorphic function given by $F_k(s)=B^s Z(s)/s^{k+1}$.
Let $\mu$ be any common irrationality measure
for the real numbers $\log q_j/\log q_{j'}$, namely
a real number such that for any two integers
$m$ and~$m'$ such that $m\log q_j + m'\log q_{j'}\neq 0$, we
have
\[ \abs{m \log q_j + m'\log q_{j'}} \geq \max(\abs m,\abs{m'})^{-\mu}. \]

A straightforward computation based on Leibniz and Cauchy formulae shows
the existence of a positive real number~$c$ such that
for any $j\in J$ and any nonzero~$m\in\Z$,
\[ \abs{\Res_{s=\alpha{j,m}} F_k(s) } \leq c (1+\abs{\Im{\alpha_{j,m}}})^{\kappa-k-1-\mu \beta} B^a (\log B)^{\max(b_j)} , \]
where $\beta=\sum_{j\in J} b_j$.
Moreover, as in the proof of Lemma~\ref{lemm.tauber.k},
there exists a polynomial  $P$ of degree~$\deg(P)=b-1$ with leading
coefficient~$\Theta/a^{k+1}(b-1)!$ such that
\[ \Res_{s=a} F_k(s) = B^a P(\log B). \]
We choose $k$ such that $k>\kappa+\mu\beta$.
Since $b>\max(b_j)$, these upper bounds imply that the
sum of residues of~$F_k$ at all poles of real part~$a$ 
is dominated by the residue at $a=0$ and 
satisfies the following estimate
\[ \sum_{\Re(s)=a} \Res F_k(s) = B^a P(\log B). \]
If $a>0$, we set $Q=0$; if $a\leq 0$, there exist a polynomial~$Q$
of
degree~$b+k$ and leading coefficient~$\Theta/(b+k)!$ such that
\[ \Res_{s=0} F_k(s)= Q(\log B).\]

We may then continue as in the proof of Lemma~\ref{lemm.tauber.k}
and conclude that 
\[ V_k(B) =  B^a P (\log B)+Q(\log B) + \mathrm O(B^a(\log B)^\beta). \]

An application of Lemma~\ref{lemm.tauber.iteration} similar
to that of Theorem~\ref{theo.tauber} then implies
Theorem~\ref{theo.tauber.S}. 
\end{proof}

\normalsize


\clearpage

\def\noop#1{\ignorespaces}

\bibliographystyle{smfplain}

\bibliography{aclab,acl,integral}

\providecommand{\noopsort}[1]{}\providecommand{\url}[1]{\textit{#1}}
\providecommand{\bysame}{\leavevmode ---\ }
\providecommand{\og}{``}
\providecommand{\fg}{''}
\providecommand{\smfandname}{\&}
\providecommand{\smfedsname}{\'eds.}
\providecommand{\smfedname}{\'ed.}
\providecommand{\smfmastersthesisname}{M\'emoire}
\providecommand{\smfphdthesisname}{Th\`ese}
\begin{thebibliography}{10}

\bibitem{baker1975}
{\scshape A.~Baker} -- \emph{Transcendental number theory}, Cambridge
  Mathematical Library, Cambridge University Press, 1975.

\bibitem{batyrev99}
{\scshape V.~V. Batyrev} -- {\og Birational {C}alabi-{Y}au $n$-folds have equal
  {B}etti numbers\fg}, in \emph{New trends in algebraic geometry (Warwick,
  1996)}, Cambridge Univ. Press, Cambridge, 1999, p.~1--11.

\bibitem{batyrev-m90}
{\scshape V.~V. Batyrev {\normalfont \smfandname} {\relax Yu}.~I. Manin} --
  {\og Sur le nombre de points rationnels de hauteur born{\'e}e des
  vari{\'e}t{\'e}s alg{\'e}briques\fg}, \emph{Math. Ann.} \textbf{286} (1990),
  p.~27--43.

\bibitem{batyrev-t95b}
{\scshape V.~V. Batyrev {\normalfont \smfandname} {\relax Yu}.~Tschinkel} --
  {\og Rational points on bounded height on compactifications of anisotropic
  tori\fg}, \emph{Internat. Math. Res. Notices} \textbf{12} (1995),
  p.~591--635.

\bibitem{batyrev-t96}
\bysame , {\og {H}eight zeta functions of toric varieties\fg}, \emph{Journal
  Math. Sciences} \textbf{82} (1996), no.~1, p.~3220--3239.

\bibitem{birch62}
{\scshape B.~J. Birch} -- {\og Forms in many variables\fg}, \emph{Proc. London
  Math. Soc.} \textbf{265A} (1962), p.~245--263.

\bibitem{borovoi-r1995}
{\scshape M.~Borovoi {\normalfont \smfandname} Z.~Rudnick} -- {\og
  Hardy-{L}ittlewood varieties and semisimple groups\fg}, \emph{Invent. Math.}
  \textbf{119} (1995), no.~1, p.~37--66.

\bibitem{brauer1947}
{\scshape R.~Brauer} -- {\og On {A}rtin's ${L}$-series with general group
  characters\fg}, \emph{Ann. of Math.} \textbf{48} (1947), p.~502--514.

\bibitem{breteche98b}
{\scshape R.~de~la Bret{\`e}che} -- {\og Compter des points d'une vari\'et\'e
  torique\fg}, \emph{J. Number Theory} \textbf{87} (2001), no.~2, p.~315--331.

\bibitem{brion-kumar2005}
{\scshape M.~Brion {\normalfont \smfandname} S.~Kumar} -- \emph{Frobenius
  splitting methods in geometry and representation theory}, Progress in
  Mathematics, vol. 231, Birkh\"auser Boston Inc., Boston, MA, 2005.

\bibitem{bruhat-tits1972}
{\scshape F.~Bruhat {\normalfont \smfandname} J.~Tits} -- {\og Groupes
  r\'eductifs sur un corps local, {I}. {D}onn\'ees radicielles valu\'ees\fg},
  \emph{Publ. Math. Inst. Hautes {\'E}tudes Sci.} (1972), no.~41, p.~5--251.

\bibitem{chambert-loir-t2000b}
{\scshape A.~Chambert-Loir {\normalfont \smfandname} {\relax Yu}.~Tschinkel} --
  {\og \noopsort{D}{P}oints of bounded height on equivariant compactifications
  of vector groups, {II}\fg}, \emph{J. Number Theory} \textbf{85} (2000),
  no.~2, p.~172--188.

\bibitem{chambert-loir-t2002}
\bysame , {\og On the distribution of points of bounded height on equivariant
  compactifications of vector groups\fg}, \emph{Invent. Math.} \textbf{148}
  (2002), p.~421--452.

\bibitem{clemens1969}
{\scshape H.~Clemens} -- {\og Picard-{L}efschetz theorem for families of
  nonsingular algebraic varieties acquiring ordinary singularities\fg},
  \emph{Trans. Amer. Math. Soc.} \textbf{136} (1969), p.~93--108.

\bibitem{colliot-thelene-sansuc1987}
{\scshape J.-L. Colliot-Th{\'e}l{\`e}ne {\normalfont \smfandname} J.-J. Sansuc}
  -- {\og La descente sur les vari\'et\'es rationnelles. {II}\fg}, \emph{Duke
  Math. J.} \textbf{54} (1987), no.~2, p.~375--492.

\bibitem{deconcini-p83}
{\scshape C.~De~Concini {\normalfont \smfandname} C.~Procesi} -- {\og Complete
  symmetric varieties\fg}, in \emph{Invariant theory (Montecatini, 1982)},
  Lecture Notes in Math., vol. 996, Springer, Berlin, 1983, p.~1--44.

\bibitem{deligne74}
{\scshape P.~Deligne} -- {\og La conjecture de {W}eil, {I}\fg}, \emph{Publ.
  Math. Inst. Hautes {\'E}tudes Sci.} \textbf{43} (1974), p.~273--307.

\bibitem{denef87}
{\scshape J.~Denef} -- {\og On the degree of {I}gusa's local zeta function\fg},
  \emph{Amer. J. Math.} \textbf{109} (1987), p.~991--1008.

\bibitem{duke-r-s1993}
{\scshape W.~Duke, Z.~Rudnick {\normalfont \smfandname} P.~Sarnak} -- {\og
  Density of integer points on affine homogeneous varieties\fg}, \emph{Duke
  Math. J.} \textbf{71} (1993), no.~1, p.~143--179.

\bibitem{eskin-mcmullen1993}
{\scshape A.~Eskin {\normalfont \smfandname} C.~McMullen} -- {\og Mixing,
  counting and equidistribution on {L}ie groups\fg}, \emph{Duke Math. J.}
  \textbf{71} (1993), no.~1, p.~181--209.

\bibitem{eskin-m-s1996}
{\scshape A.~Eskin, S.~Mozes {\normalfont \smfandname} N.~Shah} -- {\og
  Unipotent flows and counting lattice points on homogeneous varieties\fg},
  \emph{Ann. of Math.} \textbf{143} (1996), no.~2, p.~253--299.

\bibitem{franke-m-t89}
{\scshape J.~Franke, {\relax Yu}.~I. Manin {\normalfont \smfandname} {\relax
  Yu}.~Tschinkel} -- {\og Rational points of bounded height on {F}ano
  varieties\fg}, \emph{Invent. Math.} \textbf{95} (1989), no.~2, p.~421--435.

\bibitem{gordon1980}
{\scshape G.~L. Gordon} -- {\og On a simplicial complex associated to the
  monodromy\fg}, \emph{Trans. Amer. Math. Soc.} \textbf{261} (1980), no.~1,
  p.~93--101.

\bibitem{gorodnik-m-o2009}
{\scshape A.~Gorodnik, F.~Maucourant {\normalfont \smfandname} H.~Oh} -- {\og
  {M}anin’s and {P}eyre’s conjectures on rational points and adelic
  mixing\fg}, \emph{Ann. Sci. {\'E}cole Norm. Sup.} \textbf{41} (2008), no.~3,
  p.~385--437.

\bibitem{gorodnik-o-s2006}
{\scshape A.~Gorodnik, H.~Oh {\normalfont \smfandname} N.~Shah} -- {\og
  Integral points on symmetric varieties and {S}atake compatifications\fg},
  2006, arXiv:math/0610497 [math.NT].

\bibitem{hassett-tschinkel2003}
{\scshape B.~Hassett {\normalfont \smfandname} Y.~Tschinkel} -- {\og Integral
  points and effective cone of moduli spaces of stable maps\fg}, \emph{Duke
  Math. J.} \textbf{120} (2003), no.~2, p.~577--599.

\bibitem{hindry-silverman2000}
{\scshape M.~Hindry {\normalfont \smfandname} J.~H. Silverman} --
  \emph{Diophantine geometry}, Graduate Texts in Mathematics, vol. 201,
  Springer-Verlag, New York, 2000, An introduction.

\bibitem{igusa1974a}
{\scshape J.-i. Igusa} -- {\og Complex powers and asymptotic expansions. {I}.
  {F}unctions of certain types\fg}, \emph{J. Reine Angew. Math.}
  \textbf{268/269} (1974), p.~110--130, Collection of articles dedicated to
  Helmut Hasse on his seventy-fifth birthday, II.

\bibitem{igusa1974b}
\bysame , {\og Complex powers and asymptotic expansions. {II}. {A}symptotic
  expansions\fg}, \emph{J. Reine Angew. Math.} \textbf{278/279} (1975),
  p.~307--321.

\bibitem{knop1991}
{\scshape F.~Knop} -- {\og The {L}una-{V}ust theory of spherical
  embeddings\fg}, in \emph{Proceedings of the Hyderabad Conference on Algebraic
  Groups (Hyderabad, 1989)} (Madras), Manoj Prakashan, 1991, p.~225--249.

\bibitem{lang1983}
{\scshape S.~Lang} -- \emph{Fundamentals of {D}iophantine geometry},
  Springer-Verlag, New York, 1983.

\bibitem{lang-w54}
{\scshape S.~Lang {\normalfont \smfandname} A.~Weil} -- {\og Number of points
  of varieties in finite fields\fg}, \emph{Amer. J. Math.} \textbf{76} (1954),
  p.~819--827.

\bibitem{maucourant2007}
{\scshape F.~Maucourant} -- {\og {Homogeneous asymptotic limits of {H}aar
  measures of semisimple linear groups and their lattices}\fg}, \emph{Duke
  Math. J.} \textbf{136} (2007), no.~2, p.~357--399.

\bibitem{milne80}
{\scshape J.~S. Milne} -- \emph{{\'E}tale cohomology}, Math. Notes, no.~33,
  Princeton Univ. Press, 1980.

\bibitem{peyre95}
{\scshape E.~Peyre} -- {\og Hauteurs et mesures de {T}amagawa sur les
  vari{\'e}t{\'e}s de {F}ano\fg}, \emph{Duke Math. J.} \textbf{79} (1995),
  p.~101--218.

\bibitem{zzproc-peyre98}
\bysame  (\smfedname) -- \emph{Nombre et r{\'e}partition des points de hauteur
  born{\'e}e}, Ast{\'e}risque, no. 251, 1998.

\bibitem{peyre98}
\bysame , {\og Terme principal de la fonction z{\^e}ta des hauteurs et torseurs
  universels\fg}, in \emph{Nombre et r{\'e}partition des points de hauteur
  born{\'e}e} \cite{zzproc-peyre98}, p.~259--298.

\bibitem{salberger98}
{\scshape P.~Salberger} -- {\og {T}amagawa measures on universal torsors and
  points of bounded height on {F}ano varieties\fg}, in \emph{Nombre et
  r{\'e}partition des points de hauteur born{\'e}e} \cite{zzproc-peyre98},
  p.~91--258.

\bibitem{serre1981}
{\scshape J.-P. Serre} -- {\og Quelques applications du th\'eor\`eme de
  densit\'e de {C}hebotarev\fg}, \emph{Inst. Hautes \'Etudes Sci. Publ. Math.}
  (1981), no.~54, p.~323--401.

\bibitem{serre1997}
\bysame , \emph{Lectures on the {M}ordell-{W}eil theorem}, third \smfedname,
  Aspects of Mathematics, Friedr. Vieweg \& Sohn, Braunschweig, 1997,
  Translated from the French and edited by Martin Brown from notes by Michel
  Waldschmidt, With a foreword by Brown and Serre.

\bibitem{shah2000}
{\scshape N.~A. Shah} -- {\og Counting integral matrices with a given
  characteristic polynomial\fg}, \emph{Sankhy\=a Ser. A} \textbf{62} (2000),
  no.~3, p.~386--412, Ergodic theory and harmonic analysis (Mumbai, 1999).

\bibitem{shalika-tb-t2007}
{\scshape J.~Shalika, R.~Takloo-Bighash {\normalfont \smfandname} Y.~Tschinkel}
  -- {\og Rational points on compactifications of semi-simple groups\fg},
  \emph{J. Amer. Math. Soc.} \textbf{20} (2007), no.~4, p.~1135--1186
  (electronic).

\bibitem{skorobogatov2001}
{\scshape A.~Skorobogatov} -- \emph{Torsors and rational points}, Cambridge
  Tracts in Mathematics, vol. 144, Cambridge University Press, Cambridge, 2001.

\bibitem{weil82}
{\scshape A.~Weil} -- \emph{Adeles and algebraic groups}, Progr. Math., no.~23,
  Birkh{\"a}user, 1982.

\bibitem{weil1967}
{\scshape A.~Weil} -- \emph{Basic number theory}, Classics in Mathematics,
  Springer-Verlag, Berlin, 1995, Reprint of the second (1973) edition.

\end{thebibliography}

\end{document}